\documentclass[11pt]{article} 
\usepackage{graphicx} 
\usepackage{fullpage} 
\usepackage{amsmath,amssymb,amsthm} 
\usepackage{mathrsfs} 
\usepackage{verbatim} 
\usepackage{bbm} 
\usepackage{marginnote} 

\usepackage[T1]{fontenc} 
\usepackage{kpfonts} 
\usepackage{microtype} 
\usepackage{authblk} 
\usepackage[dvipsnames]{xcolor}
\usepackage[square,numbers]{natbib} 
\usepackage[plainpages=false,pdfpagelabels,colorlinks=true,
linkcolor=RedOrange,urlcolor=RedOrange,citecolor=MidnightBlue]{hyperref}
\usepackage[shortlabels]{enumitem}

\numberwithin{equation}{section}
\newtheorem{theorem}{Theorem}[section]
\newtheorem{corollary}{Corollary}[section]
\newtheorem{lemma}{Lemma}[section]
\newtheorem{assumption}{Assumption}
\newtheorem{proposition}{Proposition}[section]
\newtheorem{definition}{Definition}

\newtheorem{remark}{Remark}[section]

\def\ddefloop#1{\ifx\ddefloop#1\else\ddef{#1}\expandafter\ddefloop\fi}
\def\ddef#1{\expandafter\def\csname c#1\endcsname{\ensuremath{\mathcal{#1}}}}
\ddefloop ABCDEFGHIJKLMNOPQRSTUVWXYZ\ddefloop
\def\ddef#1{\expandafter\def\csname s#1\endcsname{\ensuremath{\mathsf{#1}}}}
\ddefloop ABCDEFGHIJKLMNOPQRSTUVWXYZ\ddefloop
\def\ddef#1{\expandafter\def\csname b#1\endcsname{\ensuremath{\mathbf{#1}}}}
\ddefloop ABCDEFGHIJKLMNOPQRSTUVWXYZ\ddefloop
\def\ddef#1{\expandafter\def\csname b#1\endcsname{\ensuremath{\mathbf{#1}}}}
\ddefloop abcdefghijklmnopqrstuvwxyz\ddefloop
\def\argmin{\operatornamewithlimits{arg\,min}}
\def\argmax{\operatornamewithlimits{arg\,max}}
\def\E{\mathop{\mathbf{E}}}

\def\Reals{\mathbb{R}}

\newcommand{\twopartdef}[4]
{
	\left\{
		\begin{array}{ll}
			#1 & \mbox{if } #2 \\
			#3 & \mbox{if } #4
		\end{array}
	\right.
}
\newcommand{\psik}{\psi^{\downarrow}(\kappa)}
\newcommand{\gba}{\emph{\text{Boosting Algorithm}}}
\newcommand{\gbas}{\emph{\text{Boosting Algorithms}}}
\newcommand{\iL}{\Lambda^{-1} }
\newcommand{\irL}{\Lambda^{-1/2}}

\DeclareMathOperator{\sign}{sign}
\def\dif#1{\mathrm{d}{#1}\,}

\DeclareMathOperator{\prox}{\mathbf{prox}}

\begin{document}

\title{A Precise High-Dimensional Asymptotic Theory for Boosting and Minimum-$\ell_1$-Norm Interpolated Classifiers}
\date{}
\author[1]{Tengyuan Liang 
\thanks{{\tt \url{tengyuan.liang@chicagobooth.edu}}. T. Liang acknowledges support from the NSF Career award (DMS-2042473), the George C. Tiao faculty fellowship and the William S. Fishman faculty research fund at the University of Chicago Booth School of Business. T. Liang wishes to thank Yoav Freund, Bin Yu, Misha Belkin, as well as participants in the Learning Theory seminar at Google Research and  NSF-Simons Collaboration on Mathematics of Deep Learning for constructive feedback that greatly improved the paper.}}
\author[2]{Pragya Sur
\thanks{{\tt \url{pragya@fas.harvard.edu}}. P. Sur was partially supported by  the Center for Research on Computation and Society, Harvard John A. Paulson School of Engineering and Applied Sciences and by NSF DMS-2113426. P. Sur wishes to thank the organizers and participants of the Young Data Science Researcher Seminar, ETH Zurich, for constructive feedback.}}

\affil[1]{University of Chicago}
\affil[2]{Harvard University}

\maketitle \thispagestyle{empty}

\maketitle

\begin{abstract}
This paper establishes a precise high-dimensional asymptotic theory for boosting on separable data, taking statistical and computational perspectives. We consider a high-dimensional setting where the number of features (weak learners) $p$ scales with the sample size $n$, in an overparametrized regime. Under a class of statistical models, we provide an exact analysis of the generalization error of boosting when the algorithm interpolates the training data and maximizes the empirical $\ell_1$-margin. Further, we explicitly pin down the relation between the boosting test error and the optimal Bayes error, as well as the proportion of active features at interpolation (with zero initialization). In turn, these precise characterizations answer certain questions raised in \cite{breiman1999prediction, schapire1998boosting} surrounding boosting, under assumed data generating processes. At the heart of our theory lies an in-depth study of the maximum-$\ell_1$-margin, which can be accurately described by a new system of non-linear equations; to analyze this margin, we rely on Gaussian comparison techniques and develop a novel uniform deviation argument. Our statistical and computational arguments can handle (1) any finite-rank spiked covariance model for the feature distribution and (2) variants of boosting corresponding to general $\ell_q$-geometry, $q \in [1, 2]$. As a final component, via the Lindeberg principle, we establish a universality result showcasing that the scaled $\ell_1$-margin (asymptotically) remains the same, whether the covariates used for boosting arise from a non-linear random feature model or an appropriately linearized model with matching moments. 	
\end{abstract}


\section{Introduction}
\label{sec:intro}
Modern machine learning methods are regularly used for classification tasks. Typically, these algorithms are complex, and often produce solutions with zero training error, even for random labels. Prominent examples include ensemble learning, neural networks, and kernel machines. However, among the many solutions that interpolate the training data, not all exhibit superior generalization. Empirically, it has been commonly observed that practical algorithms---running even on large overparametrized models---favor minimal ways of interpolating the training data, which has been conjectured to be crucial for good generalization.  Different problem formulations and optimization algorithms favor distinct notions of minimalism, typically measured by specific norms of the classifier. This paper focuses on the celebrated boosting/AdaBoost algorithm in this minimum-norm interpolation regime, where we conduct a precise analysis of its statistical and computational properties under specific data-generating mechanisms.

Ensemble learning algorithms, recognized as powerful toolkits at the disposal of a data scientist, have found widespread usage across domains. Boosting is arguably one of the most powerful ensemble learning algorithms that combines weak learners using intelligent schemes and exhibits remarkable generalization performance. The groundbreaking AdaBoost paper, Freund and Schapire \cite{freund1995decision}, is widely regarded as the milestone in the boosting literature, which can be traced back even earlier \cite{schapire1990strength,freund1995boosting}.
AdaBoost is an iterative algorithm that updates the weights on the training examples adaptively based on the errors incurred at prior iterations. AdaBoost demonstrated preferable generalization capabilities over existing algorithms such as bagging \cite{schapire1998boosting}, which led to decades of research activities devoted to a better understanding of this algorithm and its variants.

The seminal papers \cite{breiman1996bias,drucker1996boosting,quinlan1996bagging} observed that AdaBoost achieves zero error on the training data within a few iterations, whereas the generalization error continues to decrease well beyond this interpolation timepoint. Recently, similar phenomena and puzzles resurfaced in the context of neural networks \citep{zhang2016understanding}, and motivated the study of interpolation and implicit regularization \citep{belkin2018understand,belkin2018reconciling,liang2018just,hastie2019surprises,bartlett2019benign,liang2019risk}. This peculiar and seemingly counter-intuitive phenomenon naturally piqued the interest of a broad community of statisticians and machine learners. Several explanations emerged over the past two decades.

\noindent \textbf{Margin-based analyses.}
In a breakthrough work, Schapire, Freund, Bartlett and Lee \cite{schapire1998boosting} proposed that the generalization performance of the algorithm is crucially tied to a measure of confidence in classification, that can be captured through the (normalized) empirical margin of the training examples. 
\cite{schapire1998boosting} observed that over the course of iterations, AdaBoost creates classifiers such that the fraction of training examples with a large margin increases, and the empirical margin distribution stabilizes to a limiting one rapidly. In particular, given any margin level $\kappa>0$, they discovered upper bounds on the prediction error that reveal interesting tradeoffs between two terms---(i) the fraction of training examples with margin below $\kappa$, and (ii) the term $\kappa^{-1} \mathrm{C}(\cH)/\sqrt{n}$ that  involves the complexity of the class $\mathrm{C}(\cH)$ and the sample size $n$ scaled by $\kappa$. A large empirical margin distribution was then conjectured to be a key factor behind the superior generalization performance of  certain classifiers.
These upper bounds provided extremely useful insights, nonetheless, \cite{schapire1998boosting} commented that the proposed upper bounds can be sub-optimal in general, and that \textit{``an important open problem is to derive more careful and precise bounds $\hdots$ Besides paying closer attention to constant factors, such an analysis might also involve the measurement of more sophisticated statistics."}
Breiman \cite{breiman1999prediction} subsequently contended these empirical margin distribution based explanations, using extensive simulations, and proposed to bound the generalization error using the \emph{minimum value of the margin} over the training set. Later, Koltchinskii and Panchenko \cite{koltchinskii2002empirical} improved the earlier bounds from \cite{schapire1998boosting}. Despite significant progress in this direction, since these results involved upper bounds, the qualitative question regarding key quantities that precisely determine the generalization behavior of AdaBoost remained unanswered.

\noindent \textbf{Consistency and early stopping.} In conjunction with the generalization error, statisticians and learning theorists deeply care about the consistency of AdaBoost, and in particular, about the precise relationship between the test error and the optimal Bayes error. The problem of consistency was posed by Breiman \cite{breiman2004population}, who studied convergence properties of the algorithm in the population case. The seminal papers Jiang \cite{jiang2004process}, Lugosi and Vayatis \cite{lugosi2004BayesriskConsistencyRegularized}, Zhang \cite{zhang2004statistical}, Koltchinskii and Besnozova \cite{koltchinskii2005exponential} considered different function classes and variants of boosting, and furthered this direction of research. \cite{jiang2004process} established that AdaBoost is process consistent, in the sense that, there exists a stopping time at which the prediction error approximates the optimal Bayes error in the limit of large samples. A parallel understanding emerged from empirical studies conducted in \cite{friedman2000additive, grove1998boosting, ratsch2001soft,mason2000boosting}---AdaBoost may overfit, particularly in complex model classes and high noise settings, when left to run for an arbitrary large number of steps. On the one hand, these naturally inspired subsequent work on appropriate regularization strategies for ``early stopping" as in Zhang and Yu \cite{zhang2005boosting}, Bartlett and Traskin \cite{bartlett2007adaboost}. On the other hand, as the model classes become complex and overparametrized, the test error of boosting algorithms may deviate from the optimal Bayes error. 
Despite an extensive bulk of work, a precise characterization of the test error and its relation to the Bayes error for the overparametrized case is still missing in the current literature.

\noindent \textbf{Connections with min-$\ell_1$-norm interpolation (and implications).}
In a venture to understand the path of boosting iterates better, Rosset, Zhu and Hastie \cite{rosset2004boosting}, Zhang and Yu \cite{zhang2005boosting} established that for linearly separable data, AdaBoost with infinitesimal step size converges to the minimum-$\ell_1$-norm interpolated classifier (Equation \eqref{eq:min-l1-norm}) when left to run forever. This interpolant is crucially related to the maximum $\ell_1$-margin on the data, $\kappa_{n,\ell_1}$ (Equation \eqref{eq:max-L1}). In fact, expressed differently, these results establish that the number of optimization steps necessary for AdaBoost to reach zero training error can be upper bounded by $O(\kappa_{n, \ell_1}^{-2})$. Together with the earlier results Breiman \cite{breiman1999prediction}, this leads to a plausible conjecture that the max-$\ell_1$-margin is a crucial quantity that determines both generalization and optimization behaviors of boosting algorithms. (See also \cite{telgarsky2013margins}, for methods to shrink step sizes so that AdaBoost produces approximate maximum margin classifiers.) Thus, understanding the precise value of this margin, and the iteration time necessary for convergence to the min-$\ell_1$-norm interpolant (on separable data) is crucial for settling such a conjecture. Furthermore, refined analyses of such quantities for various overparametrized models is expected to shed light on the effects of overparametrization on optimization, an understanding that has so far eluded the literature.

Rosset et al. \cite{rosset2004boosting} further discussed that the aforementioned convergence to min-$\ell_1$-norm interpolated classifiers indicates the following: boosting potentially converges (in direction) to a sparse classifier.  It would then be of interest to understand properties of the limiting solution better, for example, the analyst may wish to understand the number of weak learners deemed  important by the boosting solution. This is particularly crucial in today's context where producing interpretable classifiers in high-stakes decision making has critical social consequences \cite{lipton2018mythos,chouldechova2018frontiers,rudin2019stop,kleinberg2019simplicity,weller2019transparency}. Boosting has subsequently witnessed widespread development, and varying perspectives have emerged through several seminal works e.g. \cite{friedman2000additive,friedman2001greedy,buhlmann2003boosting,buehlmann2006boosting,rudin2007analysis,freund2017new}; see Section \ref{sec:literature} for further discussions.

\noindent \textbf{This paper.} Prior literature suggested that the min-$\ell_1$-norm interpolated classifier and the max-$\ell_1$-margin may form central characters behind boosting algorithms on linearly separable data. However, a thorough understanding of their exact relations with the boosting solution, whether these are key quantities, and how these objects behave, have so far been lacking. 
When there is label noise in $y$, conditional on the features $x$, linear separability only happens in an overparametrized regime where the number of features $p$ grows with the sample size $n$; to see this, note that a fixed $p$-dimensional linear model class, cannot shatter $n$-points with all possible signs when $n$ grows.

Furthermore, boosting has empirically demonstrated exceptional performance with many weak-learners. Therefore, to study properties of boosting on separable data, it is both theoretically necessary and empirically natural to analyze the algorithm in a high-dimensional (overparametrized) setting.
This paper studies these crucial questions surrounding boosting, in high dimensions, focusing on the case of binary classifications. Our theoretical contributions apply under specific data generating schemes detailed in Sections \ref{sec:crucial} and \ref{sec:extend}. Throughout the paper, boosting/$\gbas$ loosely refers to the version of AdaBoost described in Section \ref{sec:crucial}.

To describe our contributions, imagine that we observe $n$ i.i.d.~samples $(x_i,y_i)$ drawn from some joint distribution, with $x_i \in \mathbb{R}^p$ abstracting the vector of weak-learners, and labels $y_i \in \{+1, -1\}$. We seek to characterize various properties of boosting in a high-dimensional setting, and to capture a regime where $p$ is comparable to $n$, assume that
 $p$ diverges with $n$ at some fixed ratio 
\begin{equation}\label{eq:high-dim}
p/n \rightarrow \psi > 0.
\end{equation}
This is a natural high-dimensional setting for analyzing separable data \cite{candes2018phase,montanari2019generalization}, as argued above; this regime has also been  investigated for regression problems and other contexts (see for instance, \cite{donoho2009message,el2013robust,donoho2016high,wang2020bridge,el2018impact,sur2019modern,sur2019likelihood,feng2021unifying}, and the references cited therein) and is well-known to produce asymptotic predictions with remarkable finite sample performance. 
Since we are  primarily interested in overparametrized settings, we assume that  the data is (asymptotically) linearly separable in the sense of Eqn.~\eqref{eq:linsep}.
This is equivalent to the dimensionality $\psi$ lying above a threshold that depends on the underlying signal strength of the problem \cite{candes2018phase,deng2019ModelDoubleDescent,montanari2019generalization}; see Section \ref{sec:crucial} for further details. Define the \textit{min-$\ell_1$-norm interpolated classifier} to be
\begin{align}\label{eq:min-l1-norm}
	\hat \theta_{n,\ell_1} \in \argmin_{\theta} ~\| \theta \|_1,~~ \text{s.t.}~ y_i x_i^\top \theta \geq 1,~ 1\leq i \leq n \enspace.
\end{align}
Note that at a finite sample level the min-$\ell_1$-norm interpolants may not be unique, and our asymptotic theory works for any such $\hat \theta_{n,\ell_1}$.
It is not hard to see that the $\hat{\theta}_{n,{\ell_1}}$ direction solves the following \textit{max-$\ell_1$-margin} problem
\begin{align}\label{eq:max-L1}
 \kappa_{n,\ell_1} :=	\max_{\| \theta \|_1 \leq 1} \min_{1\leq i\leq n} ~y_i x_i^\top \theta\enspace, 
\end{align}
whenever $\kappa_{n,\ell_1}$ is positive. We first study a stylized model where each row of the design matrix follows a Gaussian distribution with a diagonal covariance, the response is binary, and the distribution of the response conditional on the covariates is given by a generalized linear model as in \eqref{eq:DGP} (see Section \ref{sec:crucial} for further details). Later, Section~\ref{sec:extend} presents extensions  to showcase that the precise asymptotic theory carries over to spiked covariance models and random feature models. Therefore, we think of the stylized diagonal Gaussian model as capturing the essence of the mathematical derivations without overwhelming the readers, but certainly not the sole situation where precise asymptotics can be derived.
In the aforementioned setting, this paper provides the following contributions to the statistical and computational understanding of boosting:

\begin{enumerate}
	\item[(i).] We characterize precisely the value of the max-$\ell_1$-margin (Theorem \ref{thm:l-1-margin}) in the high-dimensional regime \eqref{eq:high-dim}, answering a question raised in \cite{breiman1999prediction}.
	  Informally, we show that $\sqrt{p} \kappa_{n,\ell_1}$ converges almost surely to a constant $\kappa_{\star}$ that depends on $\psi$ and other problem parameters, such as the signal-to-noise ratio in the data generating model. Theorem \ref{thm:l-1-margin} explicitly pins down the limiting constant $\kappa_\star$; in fact, this can be entirely described by the fixed points of a complicated yet easy to solve non-linear system of equations that we will introduce in \eqref{eq:fix-points-l1}. This limiting characterization will prove crucial for understanding the properties of boosting on (asymptotically) separable data.	

	\item[(ii).] In parallel, we establish precise formulae for the generalization error of the min-$\ell_1$-norm interpolant $\hat{\theta}_{n,\ell_1}$ (Theorem \ref{thm:gen-error}), once again in the regime \eqref{eq:high-dim}. The formula illuminates that the generalization error is completely governed by the dimensionality parameter $\psi$ and the limit $\kappa_{\star}$ characterized in the preceding step. The consequences of this result for boosting will be discussed soon; notably, the min-$\ell_1$-norm interpolant has been conjectured to be crucial in other contexts (see Section \ref{sec:literature}), and therefore, we expect Theorem \ref{thm:gen-error} to be of wider importance beyond boosting. 
	
	\item[(iii).] Turning to boosting, we provide an exact characterization of a threshold $T$ such that for all iterations $t \geq T$, the boosting iterates (with a properly scaled step size) stay arbitrarily close to $\hat{\theta}_{n,\ell_1}$, in the large $n,p$ limit \eqref{eq:high-dim} (Theorem \ref{thm:optimization-convergence}). This characterization builds upon existing works on margin maximization that provide a $1/\sqrt{t}$ rate \cite{telgarsky2013margins,freund2013adaboost}, and uses the well-known rescaling technique, shrinkage technique and mirror descent connections of boosting (see \cite{zhang2005boosting,shalev2010equivalence,freund2013adaboost,gunasekar2018characterizing,ji2021characterizing,chizat2020implicit} for a non-exhaustive set of related works). However, together with Theorems \ref{thm:l-1-margin}-\ref{thm:gen-error},  this result provides an exact characterization of the generalization error of boosting, and improves upon the existing upper bounds by a margin \cite{schapire1998boosting,koltchinskii2002empirical}, in our setting. Crucially, this formula involves $\kappa_\star$ (through an  implicit non-linear function), and therefore, our results imply that, at least under the aforementioned data-generation scheme, the max-$\ell_1$-margin drives the generalization performance of boosting. Furthermore, the formula encodes a concrete recipe for comparing the test error of boosting with the Bayes error in high dimensions.
	
	\item[(iv).] The iteration threshold $T$ from the previous step can be described through a precise formula (in the large $n,p$ limit) that involves the limit of the max-$\ell_1$-margin $\kappa_\star$. Utilizing this, we demonstrate two curious phenomena regarding overparametrization, both not known earlier for boosting. (1) Keeping other problem parameters fixed, $T$ decreases with an increase in $\psi$, suggesting that \emph{overparametrization helps in optimization}. (2) We establish bounds on the fraction of activated coordinates in the boosting solution (with zero initialization) when it first interpolates the training data. 
	
	\item[(v).] Finally, we introduce a new class of boosting algorithms that converge to the max-$\ell_q$-margin direction (Section \ref{sec:lpalgos}) for $q > 1$. \cite{rosset2004boosting} discussed the importance of studying such notions of margins, since it is unclear which geometry induces a better solution (see also \cite{gunasekar2018characterizing}). Here, we construct such algorithms and provide precise analyses of their generalization (for the case $q \in [1, 2]$) and optimization properties (for all $q>1$) in a spirit similar to that for boosting done above. 
\end{enumerate}

On the theoretical end, our analyses for the above contributions build upon classical results in Gaussian comparison inequalities \cite{gordon1985some,gordon1988milman} that have been strengthened relatively recently \cite{stojnic2013framework,thrampoulidis2014gaussian,thrampoulidis2018precise}, leading to the \emph{Convex Gaussian Min-Max Theorem} (CGMT) (see Section \ref{sec:literature} for a discussion). The topic of max-$\ell_2$-margin has received considerable attention, dating back to \cite{gardner1988space,shcherbina2003rigorous}, and has more recently been analyzed in \cite{montanari2019generalization,deng2019ModelDoubleDescent}. Our proofs begin from these existing theory surrounding the max-$\ell_2$-margin, particularly  \cite{montanari2019generalization,deng2019ModelDoubleDescent}, however, the $\ell_2$ (coordinate invariant) and $\ell_q$ ($q \neq 2$, coordinate specific) geometries differ significantly. Therefore, considerable theoretical work is necessary to obtain the precise characterizations outlined above; our key contributions in this regard are highlighted in Section \ref{sec:deriv}. Specifically, we introduce a novel uniform deviation argument, which later (Section \ref{sec:extend}) allows us to extend our results to settings with non-diagonal covariance between features.

The aforementioned contributions rely on a specific data-generating scheme that, to a curious reader, might appear stylized. However, the qualitative message remains the same in several settings beyond this specific scheme. Section \ref{sec:extend} explores this in further detail. In particular, we establish similar characterization for the max-$\ell_1$-margin and the min-$\ell_1$-norm interpolant for a class of models where the feature covariance is a finite-rank perturbation of a diagonal (see Section \ref{subsec:gmmext} and Appendix \ref{app:gmmext}, where we call it the spiked covariance model). Our result can in turn be utilized to establish boosting properties analogous to point (iii) above, for these other data generation schemes. We remark that the simplest model in this class---the rank-one perturbation model---corresponds to the standard Gaussian mixture model, for which precise asymptotics for the max-$\ell_2$ margin was established in \cite{deng2019ModelDoubleDescent} .

In Section \ref{sec:universality}, we prove a universality result of the following form: the value of the max-$\ell_1$-margin remains the same (asymptotically) under two different settings where the distribution of the features entered in the boosting algorithm vary. To describe in detail, suppose the observed data $\{x_i,y_i \}$ still arises from the data-generating distribution considered for our aforementioned point-by-point contributions. However, the features feeding to the Boosting Algorithm (and thus in calculating the margin) are more complicated than the raw features $x_i$'s. We consider two different kinds of boosting features---(i) features $a_i$ that take the form of a random feature model $a_i = \sigma(F^\top x_i)$ \cite{rahimi2009WeightedSums,hu2020universality,mei2019generalization}, (ii) features $b_i = \mu_0 \boldsymbol{1}+\mu_1F^{\top}x_i+\mu_2 z_i$, where the constants $\mu_0,\mu_1,\mu_2$ are calibrated appropriately to match moments of $a_i$'s and $b_i$'s.  Here, $F$ is a random matrix in $\mathbb{R}^{p \times d}$ and $z_i$ has i.i.d.~$\mathcal{N}(0,1)$ entries, independent of everything else. In each case, the max-$\ell_1$-margin is calculated using the formula $\kappa_{n,\ell_1} (\{r_i, y_i \}_{1 \leq i \leq n}):=	\max_{\| \theta \|_1 \leq 1} \min_{1\leq i\leq n} ~y_i r_i^\top \theta,$ where $r_i = a_i$ (resp.~$b_i$) in Case (i) (resp.~Case (ii)). Section \ref{sec:universality} establishes that, when $p,d $ both scale linearly with $n$, the (scaled) max-$\ell_1$-margin has the same limiting value under both settings.

The aforementioned result holds under certain assumptions on the random feature matrix $F$ and the non-linearity $\sigma(\cdot)$. (see Section \ref{sec:universality} for details). But note that, conditional on $F$, $b_i$ is Gaussian whereas $a_i$ is not. This universality result suggests that the margin value is asymptotically insensitive, at least under some settings, to nuanced properties of the feature distribution. Thus, results that apply for the Gaussian case might be relevant for certain non-Gaussian feature distributions. We further validate this through empirical observations in Section \ref{sec:universality}. On the technical front, our universality result starts with a leave-one-out argument from \cite{hu2020universality}. However, \cite{hu2020universality} considered loss functions satisfying certain smoothness and strong-convexity assumptions, which are grossly violated in our setting. This leads to several technical challenges that we handle by establishing new analytic results (Section  \ref{sec:universality} and Appendix \ref{app:universality}).

\noindent \textbf{Finite sample performance.} Our results are asymptotic in nature, and here we test their applicability and accuracy in finite samples via a simple simulation. Consider a grid of values for the overparametrization ratio $\psi \in \Psi \subset [0, 6]$, and a data-generating process where the covariates $x_i \stackrel{\mathrm{i.i.d.}}{\sim} \mathcal{N}(\boldsymbol{0},\boldsymbol{I}_p)$, and the response $y_i |x_i = +1$ with probability $\sigma(x_i^{\top} \theta_\star)$ where $\sigma(t)=1/(1+e^{-t})$, and $y_i| x_i = -1$ otherwise. Each coordinate of $\theta_\star$ is drawn i.i.d. from a Gaussian $\cN(0, 1/p)$. For each $\psi \in \Psi$, we generate multiple samples of size $n=400$, and calculate the max-$\ell_1$-margin by two methods: (i) the numerical solution $\kappa_{n,\ell_1}$ to the corresponding linear program (LP) in \eqref{eq:max-L1}; the blue points in Figure \ref{fig:margin_n_error}(a) depict these values (appropriately scaled), and, (ii)  the asymptotic value $\kappa_\star(\psi, \mu)$ predicted by our analytic formula in Theorem~\ref{thm:l-1-margin}; the red points labeled as CGMT in Figure \ref{fig:margin_n_error}(a) represent these values. Calculating our theoretical predictions involves solving a complex \textit{non-linear system of equations} defined in \eqref{eq:fix-points-l1}. This involved computing integrals, which we approximate via Monte-Carlo sums (5000 samples). Figure \ref{fig:margin_n_error}(b) compares the corresponding out-of-sample prediction error: the blue points show the generalization error  $\mathbb{P}_{\bx, \by}\left( \by \cdot \bx^\top \hat \theta_{n,\ell_1} < 0 \right)$, when $\hat{\theta}_{n,\ell_1}$ is calculated from the LP, whereas the red points depict the asymptotic value  predicted by our theory (Theorem \ref{thm:gen-error}). In both cases, the points align remarkably well, demonstrating that our theory, albeit asymptotic, shows remarkable finite sample accuracy. In this example, the threshold for separability was around $0.43$ \cite{candes2018phase}. This is also evidenced in the plot---the max-$\ell_1$-margin is positive (resp.~zero) above (resp.~below) this threshold, and as expected, our theory matches the numerics accurately above the threshold.
 
\begin{figure}[h]
\centering
  \includegraphics[width=\linewidth]{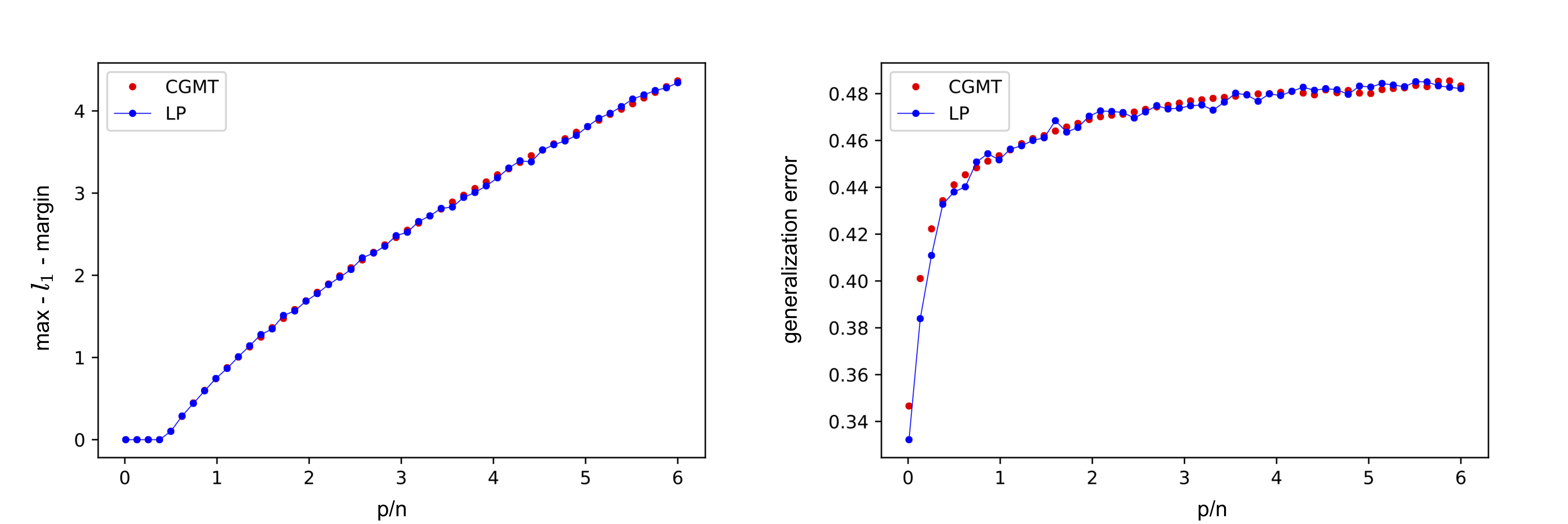}
 \caption{$x$-axis: Ratio $ p/n$.   $y$-axis: (a) Left: max-$\ell_1$-margin (as in Eqn.~\ref{eq:asymp-l1-margin}), the blue points are obtained  by solving the LP in \eqref{eq:max-L1} and averaging its solution over 10 independent simulation runs. The red points are obtained by numerically evaluating the formula in RHS of \eqref{eq:asymp-l1-margin}. (b) Right: Generalization error, the blue points are obtained by calculating the generalization error of $\hat{\theta}_{n,\ell_1}$ that forms the solution in $\theta$ of the LP \eqref{eq:max-L1}, and this is averaged over $10$ simulation runs. The red points are obtained by numerically evaluating the formula in RHS of \eqref{eq:asymp-l1-margin}.}
\label{fig:margin_n_error}
 \end{figure}

\noindent \textbf{Organization.} The rest of the paper is organized as follows. Section \ref{sec:crucial} introduces some crucial preliminaries that are heavily used through the rest of the paper. Section \ref{sec:main} presents our main results, whereas a  proof sketch and description of our technical contributions is presented in Section \ref{sec:deriv} (details are deferred to the Appendix). Section~\ref{sec:literature} discusses relevant literature that has been omitted from this introduction.
 Finally, Section \ref{sec:discussion} concludes with a discussion on possible directions for future work.


\section{Formal Setup and Preliminaries}\label{sec:crucial}

This section introduces our formal setup. Unless otherwise mentioned, we consider a sequence of problems $\{y(n), X(n), \theta_{\star}(n)\}_{n \geq 1}$, such that  $y(n) \in \Reals^n, \theta_{\star}(n) \in \Reals^{p(n)}  $ and $X(n) \in \Reals^{n \times p(n)} $, where the $i$-th row $x_i \sim \mathcal{N}(0,\Lambda(n))$, and the $i$-th entry of $y(n)$ satisfies 
\begin{equation}\label{eq:DGP}
y_i|x_i \stackrel{i.i.d.}{\sim} \begin{cases}
	+1, & \text{w.p.}~ f\big(\langle \theta_{\star}(n),x_i \rangle \big)  \\
	-1, & \text{w.p.}~ 1 - f \big( \langle \theta_{\star}(n),x_i \rangle \big)
\end{cases} \enspace.
\end{equation}
Above, $\Lambda(n) \in \Reals^{p(n) \times p(n)}$ is a diagonal covariance matrix and $f$ is any non-decreasing continuous function bounded between $0$ and $1$. Recall that we consider the asymptotic regime \eqref{eq:high-dim}, that is, $p(n)/n \rightarrow \psi \in (0,\infty)$. We require certain structural assumptions on the covariate distributions and the regression vector sequence that is described below. Conceptually, four factors determine the structure of the problem: overparametrization $\psi$, signal strength $\rho$, link function $f$, and a limiting measure $\mu$ defined in Assumption~\ref{asmp:W2-limit}. Later, Section \ref{sec:extend} will investigate models beyond \eqref{eq:DGP}.

\begin{assumption}
	\label{asmp:cov}
Let  $\lambda_i(n)$  denote the eigenvalues of $\Lambda(n)$.  Assume that there exists a positive constant $0<c<1$ such that $c \leq \lambda_i (n)\leq 1/c, ~\forall 1 \leq i \leq p(n)$ and for all $n$ and $p$. 
\end{assumption}
\begin{assumption}
	\label{asmp:W2-limit}
Define $\rho(n) \in \Reals$ and $\bar{w}(n) \in \Reals^{p(n)} $ such that
\begin{align}
		\rho{(n)} := \left(\theta_{\star}(n)^\top \Lambda(n) \theta_{\star}(n) \right)^{1/2} \qquad \text{and } \qquad  \bar{w}_i(n) := \sqrt{p} \frac{\sqrt{\lambda_i(n)} \langle \theta_{\star}(n), e_{i,p} \rangle }{\rho(n)},
	\end{align}
	where $e_{i,p}$ denotes the canonical vector in $\Reals^p$ with $1$ in the $i$-th entry and $0$ elsewhere. Assume
	 \begin{equation}\label{eq:snr}
	 \rho(n) \rightarrow \rho
	 \end{equation}
	 with $0 < \rho < \infty$. 
	Assume in addition that the empirical distribution of $\{ (\lambda_i(n), \bar{w}_i(n)) \}_{i=1}^{p(n)}$ converges to a probability distribution $\mu$ on $\mathbb{R}_{>0} \times \mathbb{R}$, in the Wasserstein-$2$ distance, that is,
	\begin{align}
		\label{eq:limit-measure-2}
		\frac{1}{p} \sum_{i=1}^p \delta_{(\lambda_i, \bar{w}_i)} \stackrel{W_2}{\Rightarrow} \mu \enspace.
	\end{align} 
\end{assumption}

\begin{remark}
Note that Assumption 1 and \eqref{eq:snr} together imply that $\sum_{j=1}^p \theta_{\star}(n)_{j}^2 = O(1)$. If all the entries of $\theta_{\star}$ are of the same order, this yields $\theta_{\star,i} = O(1/\sqrt{p})$. This also justifies why we include $\sqrt{p}$ in the numerator of $\bar{w}_i$. The convergence in $W_2$ equivalently means weak convergence and convergence of the second moments (see for instance, \citep{villani2008optimal,montanari2019generalization}). In particular, this implies that $\int w^2 \mu(d \lambda, dw) = 1$. 
\end{remark}

\begin{assumption}
	\label{asmp:boundedness}
Finally, assume that 
\begin{align}\label{eq:Wconditions}
	 \|\bar{w}(n)\|_{\infty}  \leq C',  \qquad   \text{and} \qquad  \|\bar{w}(n)\|_1/p> C''
	\end{align}
	for all $n$ and $p$, for some constants $C' , C'' > 0$. 
\end{assumption}

\noindent \textbf{Linear separability.} 
We assume that our sequence of problem instances  is (asymptotically) linearly separable in the following sense
\begin{align}\label{eq:linsep}
	\lim_{n, p(n) \rightarrow \infty} ~ \mathbb{P}\left( \exists \theta \in \mathbb{R}^p, ~y_i x_i^\top \theta > 0~\text{for}~1\leq i\leq n  \right) = 1 \enspace.
\end{align}
For the model specified in \eqref{eq:DGP}, it turns out that \eqref{eq:linsep} is satisfied if and only if the overparametrization ratio exceeds a phase transition threshold $\psi > \psi^\star(\rho, f)$. It is well-known that the separability event is equivalent to the event that the maximum likelihood estimate is attained at infinity \cite{albert1984existence}, and this has been a problem of intense study in classical statistics and information theory \cite{cover1965geometrical,santner1986note,lesaffre1989partial}. More recently, \cite{candes2018phase} derived the separability threshold $\psi^\star(\rho, f)$ for a logistic regression model (when $f$ is the sigmoid function). A similar phenomenon extends to other functions $f$ as well, as subsequently characterized by \cite{montanari2019generalization}. To describe this phase transition threshold, consider the following bivariate function $F_{\kappa}: \mathbb{R} \times \mathbb{R}_{\geq 0} \rightarrow \mathbb{R}_{\geq 0}$ defined for any  $\kappa \geq 0$,
\begin{align}
	\label{eq:YZ}
	F_{\kappa}(c_1, c_2) & := \left( \mathbb{E} \left[ (\kappa - c_1 YZ_1 - c_2 Z_2)_{+}^2 \right] \right)^{\frac{1}{2}} \quad \text{where} \nonumber\\
	& \begin{cases}
		Z_2 \perp (Y, Z_1) \\
		Z_i \sim \mathcal{N}(0, 1), ~i=1,2 \\
		\mathbb{P}(Y = +1|Z_1) = 1 - \mathbb{P}(Y = -1|Z_1) = f(\rho \cdot Z_1)
	\end{cases}
	\enspace.
\end{align}
Then 
\begin{equation}\label{eq:septhreshold}
\psi^\star(\rho, f) = \min_{c \in \mathbb{R}} F_0^2(c,1).
\end{equation} As an example, recall that $\psi^\star(\rho, f) \approx 0.43$ in the setting of Figure \ref{fig:margin_n_error}. The above function $F_{\kappa}: \mathbb{R} \times \mathbb{R}_{\geq 0} \rightarrow \mathbb{R}_{\geq 0}$ will prove crucial in our subsequent theory.

\noindent \textbf{Boosting algorithm.}
For the convenience of the readers, we describe here the general $\gbas$ we work with. We begin by briefing the steps in AdaBoost \citep{freund1996experiments,freund1995desicion}. 
Suppose that each weak learner outputs a continuous decision $X_{ij} = x_i[j] \in \mathbb{R}$ and $y_i \in \{-1, +1\}$. Let $\Delta_n $ be the standard probability simplex given by $\Delta_n: = \{\boldsymbol{p} \in [0,1]^n: \sum_{i=1}^n p_i =1 \}.$ Suppose $Z = y \circ X \in \mathbb{R}^{n\times p}$ denotes multiplying each element in the $i$-th row of $X$ by $y_i$, $i\in [n]$.
At each step, AdaBoost adaptively chooses the best feature as follows:
\begin{enumerate}
\item Initialize: data weight $\eta_0 = 1/n \cdot \mathbf{1}_n \in \Delta_n$, parameter $\theta_0 = 0$.
\item At time $t \geq 0$: 
\begin{enumerate}
	\item Feature Selection: 
	$
 v_{t+1} := \argmax_{v \in \{ e_j \}_{j\in [p]}}~ |\eta_t^\top Z v|  \enspace ;
	$
	\item Adaptive Stepsize $\alpha_{t}$: 
	$
	\alpha_t := \eta_t^\top Z v_{t+1} \enspace ;
	$
	\item Coordinate Update: 
	$
		\theta_{t+1} = \theta_{t} + \alpha_{t} \cdot v_{t+1}   \enspace ;
	$
	\item Weight Update:
	$
		\eta_{t+1}[i] \propto \eta_{t}[i] \exp(-\alpha_{t} y_i x_i^\top v_{t+1})  ,
	$
	normalized such that $\eta_{t+1} \in \Delta_n$. 
\end{enumerate}
\item Terminate after $T$ steps, and output the vector $\theta_T$. 
\end{enumerate}


\section{Main Results}\label{sec:main}
This section will provide precise analyses of the max-$\ell_1$-margin $\kappa_{n,\ell_1}$ and the min-$\ell_1$-norm interpolant $\hat{\theta}_{n,\ell_1}$, as well as the generalization and optimization performance of $\gbas$, in terms of the problem parameters $(\psi, \rho, \mu, f)$ introduced in Section \ref{sec:crucial}. 

\subsection{Max-$\ell_1$-margin and min-$\ell_1$-norm interpolant}\label{subsec:margin}
Recall the definition of the max-$\ell_1$-margin from \eqref{eq:max-L1}. We establish that $\kappa_{n,\ell_1}$, when appropriately scaled, converges almost surely to a limit that can be explicitly characterized in terms of $\psi, \mu$ and $f$. To describe this limit, consider the following function first introduced in \cite{montanari2019generalization}: for any $(\psi, \kappa)$ pair that satisfies $\psi > \psik$ (See Equation \ref{eq:psis}), define $T: (\psi, \kappa) \rightarrow \mathbb{R}$ to be	
\begin{align}\label{eq:Tfunction}
		T(\psi, \kappa) :=  \psi^{-1/2} \left[ F_{\kappa}(c_1, c_2) - c_1 \partial_1 F_{\kappa}(c_1, c_2)  - c_2 \partial_2 F_{\kappa} (c_1, c_2) \right] - s.
	\end{align}
Above, $c_1 \equiv c_1(\psi, \rho, \mu,\kappa), c_2\equiv c_2 (\psi, \rho, \mu,\kappa), s \equiv s(\psi, \rho, \mu, \kappa)$ form the \emph{unique} solution to the non-linear system of equations introduced in \eqref{eq:fix-points-l1} (Proposition \ref{prop:uniqueness} establishes uniqueness of the solution). A detailed description of this system is deferred until Section \ref{sec:system}; the key point is that, the system takes as input the quantities $\psi,\rho,\mu,\kappa$, and solves  three equations in three unknowns, producing a triplet $c_1,c_2,s$. Throughout, $\mu$ and $\rho$ will be defined via \eqref{eq:limit-measure-2} and \eqref{eq:snr} respectively, and if these are fixed,  $c_1,c_2,s$ then simply form functions of $\psi,\kappa$. Note that we drop the dependence on $f$ for simplicity of the exposition; however, it is important to emphasize that $f$ enters the definition of $F_{\kappa}(\cdot, \cdot)$, which in turn affects the equation system. 
\begin{theorem}
	\label{thm:l-1-margin}
	Suppose Assumptions \ref{asmp:cov}-\ref{asmp:boundedness} hold and that our sequence of problem instances obeys \eqref{eq:linsep}, that is, $\psi > \psi^{\star}(\rho, f)$. Then, under the asymptotic regime \eqref{eq:high-dim}, the max-$\ell_1$-margin admits the limiting characterization 
	\begin{align}
		\label{eq:asymp-l1-margin}
		\lim_{n \rightarrow \infty}~ p^{1/2}\cdot \kappa_{n,\ell_1} \stackrel{\mathrm{a.s.}}{=} \kappa_\star(\psi, \rho, \mu) \enspace,
			\end{align}
			where 
			\begin{align}
		\label{eq:kappa_star}
		\kappa_\star(\psi, \rho, \mu) = \inf \{ \kappa \geq 0 ~:~ T(\psi, \kappa) = 0  \} \enspace.
	\end{align}
\end{theorem}

The max-$\ell_1$-margin was conjectured to be a central quantity for boosting \cite{breiman1999prediction}---Theorem \ref{thm:l-1-margin} provides a precise high-dimensional characterization of this object under our data-generating scheme. For typical data instances, it is crucial to undertand how such margin scales with the overparametrization, both theoretically and empirically, which is answered by the above Theorem. This limiting result will lead to precise characterizations of statistical and computational properties of $\gbas$ in high dimensions, as we shall shortly see in Section \ref{subsec:boosting}. Although the result is asymptotic, the empirical margin (scaled) $\sqrt{p} \kappa_{n,\ell_1}$ shows remarkable agreement with the limiting value $\kappa_\star(\psi, \rho, \mu)$, even for datasets with moderate dimensions (e.g. $n=400$), as demonstrated by Figure \ref{fig:margin_n_error}.

Some comments regarding the limit $\kappa_{\star}(\psi,\rho,\mu)$ are in order. First, the limit  is well-defined, owing to properties of $T(\psi,\kappa)$: Section \ref{sec:system} presents an argument towards this claim.
Next, \eqref{eq:kappa_star} clearly demonstrates the dependence of $\kappa_{\star}(\psi,\rho,\mu)$ on the overparametrization ratio $\psi$. Its dependence on the signal strength $\rho$ and the distribution $\mu$ is encoded through $F_{\kappa}(\cdot,\cdot)$, and the parameters $c_1 \equiv c_1(\psi, \rho, \mu,\kappa), c_2\equiv c_2 (\psi, \rho, \mu,\kappa), s \equiv s(\psi, \rho, \mu, \kappa)$, which appear in the definition of $T(\psi,\kappa)$ \eqref{eq:Tfunction}.

We now proceed to study the min-$\ell_1$-norm interpolated classifier \eqref{eq:min-l1-norm}, and its precise generalization behavior in our asymptotic regime \eqref{eq:high-dim}. Define
\begin{align}
		\label{eq:analyic-gen-error}
		{\rm Err}_\star(\psi, \rho, \mu) = \mathbb{P} \left( c_1^\star YZ_1 + c_2^\star Z_2< 0 \right)\enspace,
	\end{align}
	where
	$
		c_i^\star := c_i(\psi, \rho,\mu, \kappa_\star(\psi, \rho, \mu)), ~i = 1,2.
	$
	Together with a third parameter \newline $s^{\star} \equiv$ $ s(\psi, \rho,\mu, \kappa_\star(\psi, \rho, \mu))$, $c_1^{\star}, c_2^{\star}, s^{\star}$ form the unique solution to the system of equations \eqref{eq:fix-points-l1}, when the inputs to the system are $\psi, \rho, \mu$ and $ \kappa_\star(\psi, \rho, \mu)$, \eqref{eq:asymp-l1-margin}.
	 Furthermore, $(Y, Z_1, Z_2)$ follows the joint distribution specified in \eqref{eq:YZ}; note that this depends on the problem parameters through $\rho$.

\begin{theorem}
	\label{thm:gen-error}
	Under the assumptions of Theorem \ref{thm:l-1-margin}, the generalization error of any min-$\ell_1$-interpolated classifier $\hat{\theta}_{n,\ell_1}$, defined in \eqref{eq:min-l1-norm}, converges almost surely to ${\rm Err}_{\star}(\psi, \rho,\mu)$, that is, for a new data point $(\bx,\by) $ drawn from the data-generating distribution specified in Section \ref{sec:crucial},
	\begin{align}
		\label{eq:asymp-gen-err}
		\lim_{ n \rightarrow \infty} ~ \mathbb{P}_{(\bx, \by) }\left( \by \cdot \bx^\top \hat \theta_{n,\ell_1} < 0 \right) \stackrel{\mathrm{a.s.}}{=} {\rm Err}_\star(\psi, \rho,\mu)  \enspace.
	\end{align}
\end{theorem}
Theorem \ref{thm:gen-error} provides an exact quantification of the generalization behavior of the min-$\ell_1$-norm interpolant under our data-generating scheme. Earlier works \cite{rosset2004boosting,zhang2005boosting} already characterized the long time and infinitesimal step size limit of AdaBoost on separable data. Later, Section \ref{subsec:boosting} will establish a further precise connection between $\hat{\theta}_{n,\ell_1}$ and the AdaBoost iterates (with suitably chosen learning rates). Informally, the AdaBoost iterates arrive arbitrarily close to the min-$\ell_1$-norm interpolant, beyond a certain time threshold. Therefore, Theorem \ref{thm:gen-error} provides two important contributions to the boosting literature, described as follows.

First, an open question was posed by Schapire et al. \cite{schapire1998boosting}, Breiman \cite{breiman1999prediction} regarding which quantity truly governs the generalization performance of AdaBoost. Observe that in Theorem \ref{thm:gen-error}, ${\rm Err}_\star(\psi, \rho,\mu)$ crucially depends on $\kappa_\star(\psi, \rho, \mu)$ \eqref{eq:asymp-l1-margin} through the constants $c_i^{\star}$. Therefore, the asymptotic max-$\ell_1$-margin precisely determines the generalization error. Since our result is asymptotically exact, Theorem \ref{thm:gen-error} provides an answer to the question posed in \cite{schapire1998boosting,breiman1999prediction} under our assumed model. To contrast, the existing margin-based generalization upper bounds \cite{schapire1998boosting,koltchinskii2002empirical} (that do not assume strong conditions on the data-generating distribution) scale as
\begin{align}\label{eq:classical_bound}
	\frac{1}{\sqrt{n} \kappa_{n, \ell_1}} {\rm Poly(\log n)} \asymp \frac{\sqrt{\psi}}{\kappa_\star(\psi, \rho, \mu)} {\rm Poly(\log n)} \gg {\rm Err}_\star(\psi, \rho,\mu) \enspace.
\end{align}
In fact, note that the inverse of the $y$-axis in Figure \ref{fig:normalized} corresponds to the classical upper bound $(\sqrt{n} \kappa_{n, \ell_1})^{-1}$ on the generalization error, as given by Eqn. \eqref{eq:classical_bound}, but this upper bound is vacuous in our setting (even overlooking the $\log$ factors) since it is worse than $0.5$.

As a crucial remark, note that despite its asymptotic nature, Theorem \ref{thm:gen-error} also exhibits remarkable finite sample performance, as already seen in Figure \ref{fig:margin_n_error}.
Second, the constants $c_1^{\star},c_2^{\star}$ carry elegant geometric and statistical interpretations. 
Towards establishing Theorem \ref{thm:gen-error}, it can also be shown that the angle between the interpolated solution $\hat \theta_{n,\ell_1}$ and the target $\theta_\star$ converges 
in the following sense 
\begin{equation}\label{eq:angle}
\frac{\langle \hat \theta_{n,\ell_1}, \theta_\star \rangle_{\Lambda}}{\| \hat \theta_{n,\ell_1} \|_{\Lambda} \| \theta_{\star} \|_{\Lambda}} \stackrel{\mathrm{a.s.}}{\rightarrow}  \frac{c_1^\star}{\sqrt{(c_1^\star)^2+ (c_2^\star)^2}}\enspace,
\end{equation}
where $\langle \theta_1, \theta_2 \rangle_{\Lambda} := \theta_1^\top \Lambda \theta_2$. 
Furthermore, $c_2^{\star}$ can be interpreted as the orthogonal projection, in the sense that, 
$\left\|\Pi_{(\Lambda^{1/2} \theta_{\star})^{\perp}} (\Lambda^{1/2}\hat{\theta}_{n,\ell_1}) \right\| \stackrel{\mathrm{a.s.}}{\rightarrow} c_2^\star \enspace.$

Finally, recall the Bayes error formula, and contrast it with the test error formula \eqref{eq:analyic-gen-error} proved in Theorem \ref{thm:gen-error},
\begin{align}
	{\rm Err}_{\rm Bayes}(\rho) = \mathbb{P} \left(  YZ_1 < 0 \right), \quad \quad {\rm Err}_\star(\psi, \rho, \mu) = \mathbb{P} \left( (c_2^\star)^{-1} c_1^\star YZ_1 +  Z_2< 0 \right).
\end{align}
Then, it is clear to see that $(c_2^\star)^{-1} c_1^\star$ exactly determines how the test error of $\hat{\theta}_{n,\ell_1}$ differs from the optimal Bayes error. Therefore, Theorem \ref{thm:gen-error}  advances the literature on how the test error of boosting relates to the Bayes error \cite{breiman2004population, jiang2004process,lugosi2004BayesriskConsistencyRegularized,zhang2004statistical}: the optimality of Boosting (w.r.t. the optimal Bayes classifier) is entirely determined by the magnitude of $(c_2^\star)^{-1} c_1^\star$.

The curious reader may wonder about the accuracy of our asymptotic theory for design matrices excluded from our assumptions. We further investigate this sensitivity along few directions---violation of independence between the features, violation of Gaussianity of the covariates used for boosting, and misspecification in the model due to missing a fraction of the relevant variables. We defer the readers to  Section \ref{sec:extend} for more details on these.


\subsection{The non-linear system of equations}\label{sec:system}
We will now introduce a non-linear system of equations that is key to the study of the max-$\ell_1$-margin and the min-$\ell_1$-norm interpolant in high dimensions, as delineated in Theorems \ref{thm:l-1-margin}--\ref{thm:gen-error}.

\begin{definition}
	\label{defn:system-of-equations}
	For any $\psi > 0$ and $\kappa \geq 0$, define the following system of equations in variables $(c_1, c_2, s)\in \mathbb{R}^3$,  
	\begin{align}
		\label{eq:fix-points-l1}
		c_1 &= - \E\limits_{(\Lambda, W, G) \sim \cQ}\left( \frac{\Lambda^{-1/2} W \cdot \mathcal{T}}{\psi^{-1/2} c_2^{-1} \partial_2 F_\kappa(c_1, c_2)  }  \right) \nonumber \\
		c_1^2 + c_2^2& = \E\limits_{(\Lambda, W, G) \sim \cQ} \left(  \frac{\Lambda^{-1/2} \mathcal{T}}{ \psi^{-1/2} c_2^{-1} \partial_2 F_\kappa(c_1, c_2)  } \right)^2  \enspace,  \\
		1 &= \E\limits_{(\Lambda, W, G) \sim \cQ} \left| \frac{ \Lambda^{-1}\mathcal{T}}{ \psi^{-1/2} c_2^{-1} \partial_2 F_\kappa(c_1, c_2)  }   \right| \nonumber
	\end{align}
	\text{where}
	\begin{align}\label{eq:prox}
	\prox_{\lambda}(t) =  \argmin_{s} \left\{ \lambda | s | + \frac{1}{2} (s-t)^2 \right\} = \sign(t)\left( |t| - \lambda \right)_{+} \enspace,
\end{align}
	$$\quad \mathcal{T}  =  \prox_{s}\left( \Lambda^{1/2}  G + \psi^{-1/2} [\partial_1 F_{\kappa}(c_1, c_2) - c_1 c_2^{-1} \partial_2 F_{\kappa}(c_1, c_2)] \Lambda^{1/2}  W  \right),$$
	and the expectation is over $(\Lambda, W, G)\sim \mu  \otimes \mathcal{N}(0, 1) =: \cQ$ with $\mu$ and $F_{\kappa}(\cdot,\cdot)$ defined as in \eqref{eq:limit-measure-2}, and \eqref{eq:YZ} respectively.
	\end{definition}
Note that $\Lambda$ denotes both the random variable in \eqref{eq:fix-points-l1} and the covariance matrix in Assumption~\ref{asmp:cov}. 
Such overload of notations will prove useful in the technical derivations.

This equation system is fundamental in characterizing all of the limiting results in Section \ref{subsec:margin}. At this point, the system may seem mysterious to the readers, but it arises rather naturally in the analysis of \eqref{eq:min-l1-norm}-\eqref{eq:max-L1}; this will be detailed in Section \ref{sec:deriv}. The max-$\ell_2$-margin has received considerable attention in the past \citep{montanari2019generalization, shcherbina2003rigorous, gardner1988space}, however,  \eqref{eq:fix-points-l1} differs significantly from the equation system considered in case of the $\ell_2$ geometry. This is natural, due to the intrinsic differences between the $\ell_2$ and $\ell_1$ geometries, and this also leads to significant additional technical callenges in our setting (Section \ref{sec:deriv}). Analogous systems arise in the study of high-dimensional statistical models in  the proportional regime \eqref{eq:high-dim}; here, the most relevant ones are the analysis of the Lasso under non-linear measurement models \cite{thrampoulidis2015lasso}, and that of the MLE, LRT \citep{sur2019modern,zhao2020asymptotic} and convex regularized estimators \citep{sur2019thesis,salehi2019impact} for logistic regression.

\noindent \textbf{Uniqueness.} Theorems \ref{thm:l-1-margin}-\ref{thm:gen-error} expressed our limiting results in terms of the solution to the system \eqref{eq:fix-points-l1}. It is, therefore, crucial to establish that the solution will indeed be unique. To this end, introduce the constants $\zeta$ and $\omega$ as follows:
 
 \begin{align}\label{eq:zeta}
\zeta  & := \left(\E\limits_{(\Lambda, W) \sim \mu } |\Lambda^{-1/2} W| \right)^{-1}  \nonumber \\
\omega & := \left( \E_{(\Lambda, W) \sim \mu} \big( W - \zeta\Lambda^{-1/2} \text{sign}(\zeta\Lambda^{-1/2} W) \big)^2 \right)^{1/2}
  \end{align}
Define the functions $\psi_{+}(\kappa): \mathbb{R}_{>0} \rightarrow \mathbb{R}$, $ \psi_{-} : \mathbb{R}_{>0} \rightarrow \mathbb{R} $ and  $\psi^{\downarrow}(\kappa): \mathbb{R}_{>0} \rightarrow \mathbb{R}_{\geq 0}$ as follows 
    \begin{align}\label{eq:psis}
  \psi_{+}(\kappa) & =  \twopartdef{0}{\partial_1 F_{\kappa}(\zeta,0) > 0}{\partial_2^2F_{\kappa}(\zeta,0)-\omega^2\partial_1^2F_{\kappa}(\zeta,0)}{\text{otherwise}}, \nonumber \\
  \psi_{-}(\kappa) & =\twopartdef{0}{\partial_1 F_{\kappa}(-\zeta,0) < 0}{\partial_2^2F_{\kappa}(-\zeta,0)-\omega^2\partial_1^2F_{\kappa}(-\zeta,0)}{\text{otherwise}}, \nonumber\\
  \psi^{\downarrow}(\kappa) & = \max\{\psi^{\star}(\rho,f), \psi_{+}(\kappa), \psi_{-}(\kappa)\},
 \end{align}
where $\psi^{\star}(\rho,f)$ is given by \eqref{eq:septhreshold}.

\begin{proposition}\label{prop:uniqueness}
 For any $(\psi, \kappa)$ pair satisfying $\psi > \psik$, under Assumptions \ref{asmp:cov}-\ref{asmp:boundedness}, the system of equations \eqref{eq:fix-points-l1} admits a unique solution that satisfies $(c_1,c_2,s) \in \Reals \times \Reals_{>0} \times \Reals_{>0}$.
\end{proposition}

Our proof for Proposition \ref{prop:uniqueness} adapts insights from \cite{montanari2019generalization} to the case of $\ell_1$ geometry, however, the definition of $\omega,\zeta$ in the threshold $\psik$, \eqref{eq:psis}, differs from the case of $\ell_2$ geometry. 
Now, it can be shown that $F_{\kappa}(\cdot,\cdot)$ satisfies: (i) $(\psi,\kappa) \mapsto T(\psi,\kappa)$ is continuous on its domain, (ii) for any fixed $\kappa > 0$, $T(\psi,\kappa)$ is strictly decreasing in $\psi$, (iii) for any fixed $\psi > 0$, $T(\psi,\kappa)$ is strictly increasing in $\kappa$ (\citep[Section B.5, Proposition 4.1]{montanari2019generalization}). Further, using the definition of $\psik$, and once again properties of $F_{\kappa}(\cdot,\cdot)$, one can establish that $\lim_{\psi \rightarrow \infty}T(\psi,\kappa) < 0$, whereas  $\lim_{\psi \downarrow \psik} T(\psi,\kappa) > 0$ and moreover, $\lim_{\kappa \rightarrow \infty} T(\psi,\kappa) = \infty$. Putting all of these together yields that the region $\{ (\psi, \kappa) :  \psi >  \psik\}$ contains the region $\{ (\psi, \kappa): T(\psi, \kappa) = 0 \}.$ This ensures \eqref{eq:kappa_star} is well-defined, and that $c_1^\star, c_2^\star, s^\star$ are unique. We defer to the Appendix for proof of Proposition \ref{prop:uniqueness}.


\subsection{Boosting in high dimensions}\label{subsec:boosting} 
We turn our attention to the $\gba$ described in Section~\ref{sec:crucial}. The path of boosting iterates was studied in infinite time and infinitesimal stepsize in \cite{rosset2004boosting,zhang2005boosting}. Here, we establish a sharp analysis of the number of iterations necessary for the AdaBoost iterates to approximately maximize the $\ell_1$-margin with arbitrary accuracy.

\begin{theorem}
	\label{thm:optimization-convergence}
 Under the assumptions of Theorem \ref{thm:l-1-margin}, with a suitably chosen learning rate (specified in Cor.~\ref{cor:boost-converge-max-margin}), the sequence of iterates $\{ \hat \theta^{t} \}_{t \in \mathbb{N}}$ obtained from the $\gba$ obeys the following property: 
	for any $0<\epsilon<1$, when the number of iterations $t$ satisfies
	\begin{align}\label{eq:threshold}
		t \geq T_\epsilon(n) \quad \text{with}~ \lim_{n \rightarrow \infty}~ \frac{T_\epsilon(n)}{n \log^2 n} \stackrel{\mathrm{a.s.}}{=}  \frac{12 \psi }{ \kappa_\star^2(\psi, \rho, \mu)} \epsilon^{-2} \enspace,
	\end{align}
	the solution $\hat \theta^{t}/\| \hat \theta^{t} \|_1$ forms $(1-\epsilon)$-approximation to the Min-$\ell_1$-Interpolated Classifier, that is, almost surely,
	\begin{align*}
		(1-\epsilon)\cdot \kappa_\star(\psi, \rho, \mu) & \leq \liminf_{n \rightarrow \infty}~ \left( p^{1/2} \cdot \min_{i \in [n]}\frac{y_i x_i^\top \hat \theta^{t}}{\| \hat \theta^{t} \|_1} \right) \\
		& \leq \limsup_{n \rightarrow \infty}~ \left( p^{1/2} \cdot \min_{i \in [n]}\frac{y_i x_i^\top \hat \theta^{t}}{\| \hat \theta^{t} \|_1} \right) \leq  \kappa_\star(\psi, \rho, \mu) \enspace.
	\end{align*}
\end{theorem}

The above result is obtained by combining our Theorem \ref{thm:l-1-margin} with a careful non-asymptotic
analysis of AdaBoost allowing for an explicitly-specified learning rate, that builds upon existing works on margin maximization rates, rescaling and shrinkage techniques, and the mirror descent connections of AdaBoost (see \cite{zhang2005boosting,telgarsky2013margins,freund2013adaboost,gunasekar2018characterizing,chizat2020implicit,ji2021characterizing} and references cited therein).
Together with Theorem \ref{thm:gen-error}, this result establishes a precise characterization of the computational and statistical behavior of AdaBoost for all iterations above the threshold $T_{\epsilon}(n)$, and notably complements the classical margin upper bounds \cite{schapire1998boosting,koltchinskii2002empirical}. Thus, Theorem \ref{thm:optimization-convergence} reinforces a crucial conclusion from Section \ref{subsec:margin}---the max-$\ell_1$-margin is the key quantity governing the generalization error of AdaBoost in our setting. 

Aside from strengthening this conclusion, for separable data with a large and comparable number of samples and features, the Theorem informs a stopping rule for $\gbas$ that ensures good generalization behavior.  Note that, for any numerical accuracy $\epsilon$,  the stopping time $T_\epsilon(n)$ has an asymptotic characterization (even in terms of constants), which contributes new insight to the computational properties of AdaBoost. To see this, Figure \ref{fig:normalized} plots the scaled margin limit $\psi^{-1/2} \kappa_{\star}(\psi,\rho,\mu)$ as a function of $\psi$, in the setting of Figure \ref{fig:margin_n_error}. The increase in this (scaled) limit as a function of $\psi$, together with \eqref{eq:threshold},
directly implies that the larger the overparametrization ratio, the smaller the threshold $T_{\epsilon}(n)$. Therefore, \emph{overparametrization leads to faster optimization}. Furthermore, even in terms of the optimization performance, the max-$\ell_1$-margin is once again the central quantity in our setting, as elucidated by \eqref{eq:threshold}.

\begin{remark} A natural question may arise at this point: does the max-$\ell_1$-margin studied here, when appropriately scaled, differ significantly from the $\ell_2$-margin \citep{montanari2019generalization}?
Note that the rescaled $\ell_1$-margin is always larger than the $\ell_2$-margin, denoted by $\kappa_{n,\ell_2}$, since
\begin{align}
	 \kappa_{n,\ell_2} \leq \sqrt{p} \cdot \kappa_{n,\ell_1} \enspace, \qquad \text{where} \qquad  \kappa_{n,\ell_2} :=	\max_{\| \theta \|_2 \leq 1} \min_{1\leq i\leq n} ~y_i x_i^\top \theta\enspace.
\end{align}
A comparison of Figure \ref{fig:normalized} with \cite[Fig.~1]{montanari2019generalization} shows that the range for the $\ell_1$-margin is roughly twice that for the $\ell_2$ case, demonstrating that these behave differently, even after appropriate scaling. 
\end{remark}

\begin{figure}[h]
	\centering
	\includegraphics[width = 0.55\textwidth]{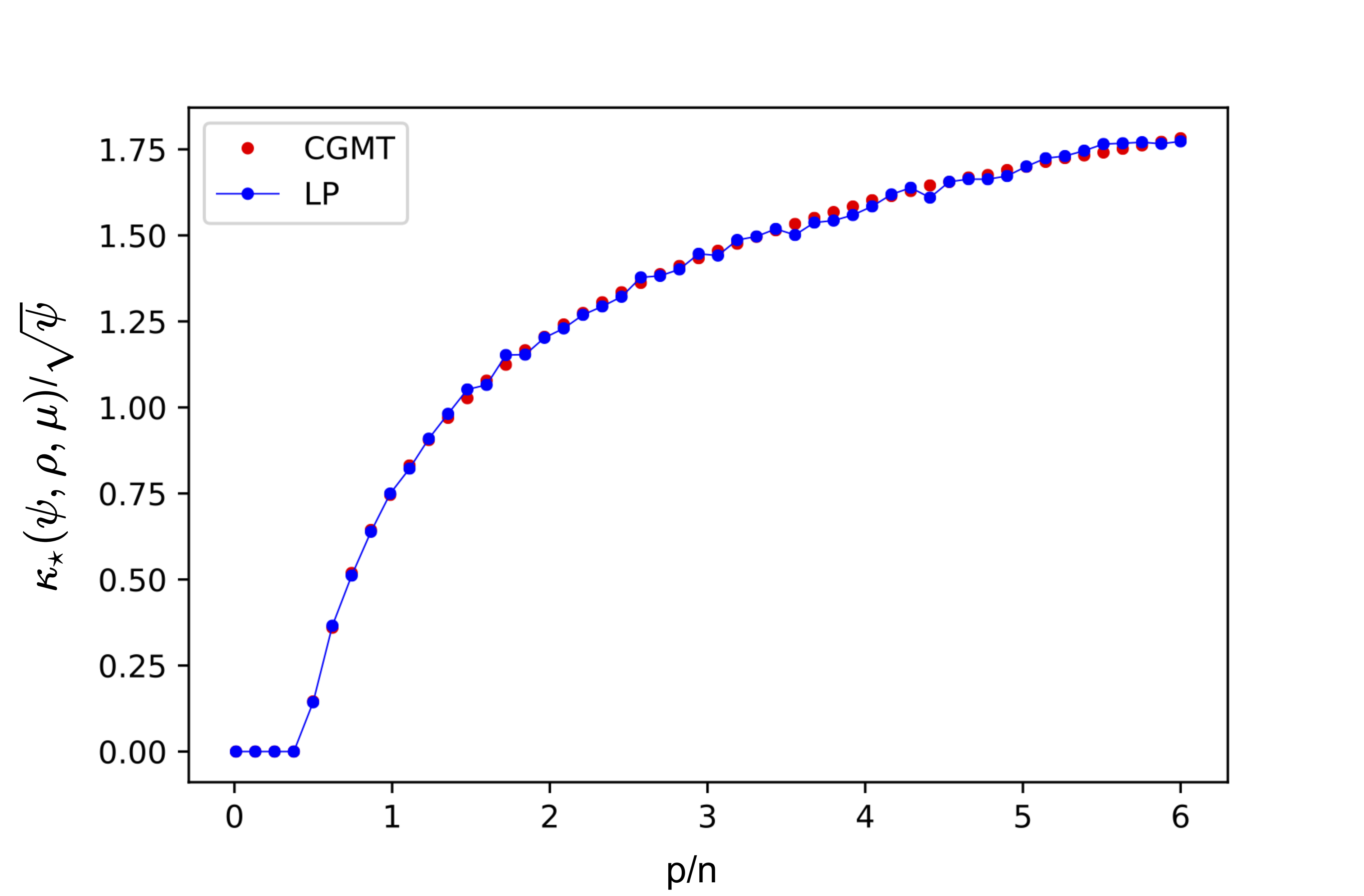}
	\caption{$x$-axis: varying ratio $\psi:= p/n$. $y$-axis: $\kappa_\star(\psi, \rho, \mu)$/$\sqrt{\psi}$ (as in Eqn.~\ref{eq:classical_bound}). The setting is the same as in Figure \ref{fig:margin_n_error}. See Figure \ref{fig:margin_n_error}(a) for details on calculation of the blue and red points.}
	\label{fig:normalized}
\end{figure}

\noindent \textbf{Proportion of activated features for AdaBoost.} The connection between the boosting solution and max-$\ell_1$-margin suggests that AdaBoost effectively converges to a sparse classifier. Motivated to understand the geometry of the solution better, the following theorem studies the proportion of active features when the training error vanishes along the path of AdaBoost. 
\begin{corollary}
	\label{thm:feature-selection}
	Let $S_0(p)$ denote the number of features selected the first time $t$ when the $\gba$ achieves zero training error (with an initialization of $\hat \theta^0 = 0$),  in the sense that, 
	\begin{align}
		S_0(p) := \# \left\{ j \in [p]: \hat \theta^{t}_j \neq 0 \right\} \enspace, \qquad \text{where} \qquad \frac{1}{n} \sum_{i=1}^n \mathbbm{I}_{y_i x_i^\top \hat\theta^t \leq 0} = 0.
	\end{align}
Under the assumptions of Theorem~\ref{thm:optimization-convergence}, $S_0(p)$, scaled appropriately, is asymptotically bounded by
	\begin{align}
		\limsup_{p \rightarrow \infty} ~\frac{S_0(p)}{p \log^2 p} \leq \frac{12}{\kappa^2_\star(\psi, \rho,\mu)}, ~~a.s.
	\end{align}
\end{corollary}
This corollary provides specific insights into the geometry of the boosting solution, by quantifying the maximum number of coordinates that may be non-zero. Note once again that the bound involves the max-$\ell_1$-margin limit, and suggests that the larger the margin, the sparser the solution (with zero training error). Thus, our limit $\kappa_{\star}(\psi,\rho,\mu)$ may even be central for determining the geometric structure of the boosting solution (at least under our data-generating scheme), beyond its foregoing roles in terms of  generalization and optimization. Note also that the margin grows as a function of $\psi$ (Fig.~\ref{fig:margin_n_error})---this further suggests that larger the overparametrization, less the number of activated coordinates for certain data-generating processes.


\subsection{A new class of boosting algorithms}\label{sec:lpalgos}

This section studies variants of AdaBoost that converge to the max-$\ell_q$-margin direction for general $q \geq1$. We also characterize the generalization error and optimization performance of a class of such algorithms, through a study of the max-$\ell_q$-margin and the min-$\ell_q$-norm interpolant beyond the case of $q=1$. This complements the study of general $\ell_q$ constraints, that was initiated by \cite{rosset2004boosting} (see also \cite{gunasekar2018characterizing} and references therein). To this end, define the max-$\ell_q$-margin to be 
\begin{align}\label{eq:max-Lq}
 \kappa_{n,\ell_q} :=	\max_{\| \theta \|_q \leq 1} \min_{1\leq i\leq n} ~y_i x_i^\top \theta\enspace,
\end{align}
and the corresponding min-$\ell_q$-norm interpolant to be 
\begin{align}\label{eq:min-lq-norm}
	\hat \theta_{n,\ell_q} \in \argmin_{\theta} ~\| \theta \|_q,~~ \text{s.t.}~ y_i x_i^\top \theta \geq 1,~ 1\leq i \leq n \enspace.
\end{align}
Denote $q_\star \geq 1$ to be the conjugate index of $q$, with $1/q_\star + 1/q = 1$, and consider the following algorithm.

\noindent \textbf{AdaBoost variant corresponding to $\ell_q$ geometry:}
\begin{enumerate}
\item Initialize: $\eta_0 = 1/n \cdot \mathbf{1}_n \in \Delta_n$, and parameter $\theta_0 = 0$.
\item At time $t \geq 0$: 
\begin{enumerate}
	\item Update Direction: 
	$
 v_{t+1} :=\argmax_{v \in \mathbb{R}^p, \| v \|_q = 1}~ \langle Z^\top \eta_t, v \rangle \enspace ;
	$
	\item Adaptive Stepsize:
	$
		\alpha_{t}(\beta) = \beta \cdot \| Z^\top \eta_t \|_{q_\star}  \enspace,
	$
	with $0<\beta<1$ being a shrinkage factor.
	\item Parameter Update: 
	$
		\theta_{t+1} = \theta_{t} + \alpha_{t} \cdot v_{t+1}   \enspace ;
	$
	\item Weight Update:
	$
		\eta_{t+1}[i] \propto \eta_{t}[i] \exp(-\alpha_{t} y_i x_i^\top v_{t+1})  ,
	$
	normalized such that $\eta_{t+1} \in \Delta_n$. 
\end{enumerate}
\item Terminate after $T$ steps, and output the vector $\theta_T$. 
\end{enumerate}
This algorithm converges to the max-$\ell_q$-margin direction, as indicated by the following corollary.
\begin{corollary}[Boosting Converges to max-$\ell_q$-margin Direction]
	\label{cor:Lq-margin}
	Let $q\geq 1$.
	Consider the aforementioned Boosting algorithm with learning rate 
	$\alpha_{t}(\beta) : = \beta \cdot \eta_t^\top Z v_{t+1},$ where $\beta<1$.
	Assume that $|X_{ij}| \leq M$ for $i \in [n], j \in [p]$. 
	Then after $T$ iterations, the Boosting iterates $\theta_T$ converge to the max-$\ell_q$-margin Direction in the following sense: for any $0 <\epsilon < 1$,
	\begin{align}
	 \kappa_{n,\ell_q} \geq \min_{i \in [n]} \frac{y_i x_i^\top \theta_T}{\| \theta_T \|_q} >  \kappa_{n,\ell_q} \cdot (1- \epsilon),
	\end{align}
	where 
	$
		T \geq \log(1.01 ne) \cdot \frac{2 p^{\frac{2}{q_\star}} M^2  \epsilon^{-2}}{\kappa_{n,\ell_q}^2} \enspace.
	$	The shrinkage factor is chosen as  $\beta = \frac{\epsilon}{p^{\frac{2}{q_\star}} M^2}$.
\end{corollary}

Utilizing arguments similar to that for Theorems \ref{thm:l-1-margin}--\ref{thm:gen-error}, it can be shown that the max-$\ell_q$-margin and the corresponding min-$\ell_q$-norm interpolant admit analogous characterizations with a system of equations that differs from \eqref{eq:fix-points-l1}, all else remaining the same. To introduce the equation system corresponding to general $\ell_q$ geometry, define the proximal mapping operator of the function $f_\lambda(t) = \lambda |t|^q$, for $\lambda>0, q \geq 1$, to be 
\begin{align}
	\prox_{\lambda}^{(q)}(t) := \argmin_{s} \left\{ \lambda |s|^q +  \frac{1}{2}(s - t)^2 \right\}.
\end{align}
With 
\begin{align*}
	 t^\star &:= -\frac{ \Lambda^{-1/2}  G + \psi^{-1/2} [\partial_1 F_{\kappa}(c_1, c_2) - c_1 c_2^{-1} \partial_2 F_{\kappa}(c_1, c_2)] \Lambda^{-1/2}  W  }{ \psi^{-1/2} c_2^{-1} \partial_2 F_\kappa(c_1, c_2)  },\\
	 \lambda^\star &:= \frac{\Lambda^{-1} s}{\psi^{-1/2} c_2^{-1} \partial_2 F_\kappa(c_1, c_2) } 
\end{align*}
define
\begin{align*}
	h^\star = \prox_{\lambda^\star}^{(q)}(t^\star).
\end{align*}
Consider the system of equations 
\begin{align}\label{eq:system-Lq}
	c_1 = \langle \Lambda^{1/2}  h^{\star}, W \rangle_{L_2(\cQ_{\infty})}, \qquad 	c_1^2 + c_2^2 = \|  \Lambda^{1/2}  h^{\star} \|^2_{L_2(\cQ_{\infty})}, \qquad \| h^{\star}\|_{L_q(\cQ_{\infty})} = 1,
\end{align}
where $\cQ_{\infty} = \mu \times \mathcal{N}(0,1)$. It is not hard to see that this system reduces to \eqref{defn:system-of-equations} for $q=1$.

\begin{corollary}\label{cor:lq}
Under the assumptions of Theorem \ref{thm:l-1-margin} and for $1\leq q \leq 2$, the max-$\ell_q$-margin obeys,
\begin{equation}\label{eq:max-lq}
 p^{\frac{1}{q} - \frac{1}{2}} \kappa_{n,\ell_q}  \stackrel{\mathrm{a.s.}} {\rightarrow}\kappa_\star^{(q)}(\psi, \rho, \mu), 
 \end{equation}
where $\kappa_{\star}^{(q)}(\psi,\rho,\mu)$ satisfies \eqref{eq:kappa_star}, with $T(\psi,\kappa)$ of the same form as in \eqref{eq:Tfunction}, but with $c_1, c_2, s $ given by the solution to \eqref{eq:system-Lq}. 
Simultaneously, the generalization error of the min-$\ell_q$-norm interpolant can be characterized using \eqref{eq:asymp-gen-err}, but when $c_1^\star,c_2^{\star},s^{\star}$ is replaced by the solution to \eqref{eq:system-Lq}, when $\kappa_{\star}^{(q)}(\psi,\rho,\mu)$ is input instead of $\kappa_{\star}(\psi,\rho,\mu)$.
\end{corollary}

Corollary \ref{cor:Lq-margin} then establishes that all properties of AdaBoost presented in Section \ref{subsec:boosting} continue to hold (after appropriate scalings) for the generalized versions of AdaBoost considered here for $1\leq q \leq 2$, with \eqref{eq:fix-points-l1} swapped for \eqref{eq:system-Lq}. Once again, observe that the max-$\ell_q$-margin is crucial for understanding properties of these variants of AdaBoost. In terms of proofs, our technical contributions in the context of the max-$\ell_1$-margin are sufficiently general, and can be adapted to establish the results in this section. Extensions to the case of $q>2$ may be feasible if one imposes a condition stronger than convergence in $W_2$ (in Assumption~\ref{asmp:W2-limit}).

\begin{remark}
Note that Corollary \ref{cor:lq} assumes the data is asymptotically linearly separable, that is, $\psi > \psi^{\star}(\rho,f)$. This separability threshold is an inherent property of the sequence of problem instances, and does not depend on the geometry under which the max-margin is considered in \eqref{eq:max-lq}. 
\end{remark}

\subsection{Robustness to assumptions}
\label{sec:extend}

The theory presented so far provides precise characterizations of the $\ell_1$ margin, interpolant and in turn AdaBoost, but relies, nonetheless, upon assumptions on the data generating process \eqref{eq:DGP}. This section explores relaxations of these assumptions along a few natural directions---(a) going beyond the assumption of independence between the covariates, (b) analyzing sensitivity to the Gaussianity assumption, (c) understanding implications of certain model misspecification. For the latter, we explore a common source of misspecification that occurs when the model misses a fraction of relevant variables.
Studying AdaBoost and the max-$\ell_1$-margin under such varied settings, we will uncover that the general insights underlying our proposed theory persist across the board, suggesting the possibility of extending our analyses to a broader class of data generation schemes. 

\subsubsection{Beyond independent covariates}\label{subsec:gmmext}

This section will focus on data-generation schemes with dependent covariates. Our exact asymptotics continue to hold for a class of such design matrices. We present results in the context of two models in an increasing order of complexity---the first (resp.~second) involves a feature covariance matrix that is a rank-one (resp.~rank-two) perturbation of a diagonal. Extensions to rank-$\ell$ perturbations are feasible (see Appendix \ref{app:gmmext}, which we refer to as spiked covariance models. The reader should take this section as a proof of concept that our results can be extended to dependent covariates in certain settings. 

As a first step towards understanding dependent covariates, consider a simple Gaussian mixture model:
\begin{align}
	 \mathbb{P}(y_i = +1) = 1- \mathbb{P}(y_i = -1) =  \upsilon \in (0,1) \label{eq:gmm1}\\
	 x_i | y_i \sim \cN(y_i \cdot \theta_\star, \Lambda), \label{eq:gmm2}
\end{align}
where $\Lambda \in \mathbb{R}^{p\times p}$ is a diagonal matrix. By the Bayes' formula, the conditional distribution of $y_i|x_i$ can be captured through a logistic model, with $\mathbb{P}(y_i=+1|x_i) = f\big(  \log \tfrac{v}{1-v} + \langle \Lambda^{-1} \theta_\star, x_i \rangle  \big)$ and $f(t) = 1/(1+e^{-t})$.
The covariate distribution obeys a mixture of Gaussians but the marginal covariance is given by $\text{Cov}(x_i) = 4v(1-v)\theta_{\star}\theta_{\star}^{\top} + \Lambda$ (thus called the spiked covariance model). Compared to the diagonal covariance as in \eqref{eq:DGP},  the setting considered here therefore goes beyond independent covariates by introducing a rank-one spike to the diagonal covariance $\Lambda$.

Similar to Assumption~\ref{asmp:W2-limit}, let $p(n)/n = \psi$ and denote
\begin{align}
	\label{eqn:GMM-mu}
	\frac{1}{p} \sum_{i=1}^p \delta_{(\lambda_i, \sqrt{p} \theta_\star^\top e_i)}  \stackrel{W_2}{\Rightarrow} \mu.
\end{align}
Define a new function $\bar{F}_{\kappa}: \mathbb{R} \times \mathbb{R}_{\geq 0} \rightarrow \mathbb{R}_{\geq 0}$ with parameter $\kappa \geq 0$,
\begin{align}\label{eq:Fkgmm}
	\bar{F}_{\kappa}(c_1, c_2) := \left( \mathbb{E} \left[ (\kappa - c_1  - c_2 Z)_{+}^2 \right] \right)^{\frac{1}{2}} ~ \text{where}~
		Z \sim \mathcal{N}(0, 1).
\end{align}
Denote a triplet of random variables $(\Lambda, \Theta, G) \sim \mu \otimes \cN(0, 1) =: \cQ$ with $\mu$ given by \eqref{eqn:GMM-mu}, and for any $\psi >0$, define the following system of equations in variables $(c_1, c_2, s)\in \mathbb{R}^3$,
	\begin{align}
		\label{eq:fix-points-gmm}
		c_1 &= - \E\limits_{(\Lambda, \Theta, G) \sim \cQ}\left( \frac{\Lambda^{-1} \Theta \cdot \prox_{s}\left( \Lambda^{1/2}  G + \psi^{-1/2} \partial_1 \bar{F}_\kappa(c_1, c_2)  \Theta \right) }{\psi^{-1/2} c_2^{-1} \partial_2 F_\kappa(c_1, c_2)  }  \right) \nonumber \\
	c_2^2& = \E\limits_{(\Lambda, \Theta, G) \sim \cQ} \left( \frac{\Lambda^{-1/2} \prox_{s}\left( \Lambda^{1/2}  G + \psi^{-1/2} \partial_1 \bar{F}_\kappa(c_1, c_2)  \Theta \right)}{ \psi^{-1/2} c_2^{-1} \partial_2 \bar{F}_\kappa(c_1, c_2)  } \right)^2  \enspace \\
		1 &= \E\limits_{(\Lambda, \Theta, G) \sim \cQ} \left| \frac{\Lambda^{-1} \prox_{s}\left( \Lambda^{1/2}  G + \psi^{-1/2} \partial_1 \bar{F}_\kappa(c_1, c_2)  \Theta \right)}{ \psi^{-1/2} c_2^{-1} \partial_2 \bar{F}_\kappa(c_1, c_2)  }   \right| \nonumber.
	\end{align}
Then, in the regime where the data is asymptotically linearly separable (see \cite[Proposition 3.1]{deng2019ModelDoubleDescent} for the linear separability threshold for this problem), the max-$\ell_1$-margin and min-$\ell_1$-norm interpolant obey the  limiting characterizations from Theorems \ref{thm:l-1-margin}-\ref{thm:gen-error}, with the system of equations given by \eqref{eq:fix-points-gmm}, and $F_{\kappa}(c_1,c_2)$ substituted by \eqref{eq:Fkgmm}. Note that \cite{deng2019ModelDoubleDescent} analyzed the max-$\ell_2$-margin for a (misspecified) logistic model and the Gaussian mixture model \eqref{eq:gmm1}-\eqref{eq:gmm2} through a unified CGMT based analysis.  Due to crucial differences between $\ell_1$ and $\ell_2$ geometries, the $\ell_1$ case, (or for that matter, any $\ell_q$ with $q \neq 2$) does not follow directly from these results.  We will elaborate on this point in Section \ref{sec:deriv}.

We can further this characterization to analogous settings where the marginal covariance between features contains a finite rank perturbation of a diagonal matrix. To provide a precise description, consider an extension of \eqref{eq:gmm1}-\eqref{eq:gmm2}, where \eqref{eq:gmm1} remains the same but \eqref{eq:gmm2} changes to 
\begin{equation}\label{eq:rank2gmm}
x_i = y_i \theta_{\star} + m_i \tilde{\theta} + \tilde{x}_i,
\end{equation}
 with $(y_i,m_i,\tilde{x}_i)$ independent of each other, $m_i$ any random variable symmetric around zero, $\tilde{x}_i \sim \mathcal{N}(0,\Lambda)$ with $\Lambda$ diagonal. The observed data contains only $(y_i,x_i)$ and thus, the $m_i$'s may be thought of as latent random variables. Note in this case, $\text{Cov}(x_i) = 4v(1-v) \theta_\star\theta_\star^{\top} + \text{Var}(m_i)\tilde{\theta}\tilde{\theta}^{\top} + \Lambda$, a rank-two perturbation of a diagonal covariance matrix. The aforementioned characterization can be naturally extended with appropriate analogues of \eqref{eqn:GMM-mu}-\eqref{eq:fix-points-gmm}. We assume that the Wasserstein-2 limit of the empirical distribution sequence $\sum_{j=1}^p \delta_{(\lambda_i,\sqrt{p}\theta_{\star}^{\top}e_i,\sqrt{p}\tilde{\theta}^{\top}e_i)}/p$  exists, denote it by $\tilde{\mu}$, and let $(\Lambda, h_{\star},\tilde{h},G) \sim \tilde{Q}= \tilde{\mu} \otimes \mathcal{N}(0,1)$. Define the following analogue of \eqref{eq:Fkgmm}, 
 \begin{equation}
 \tilde{F}_{\kappa}(c_1,c_2,c_3) = \sqrt{\mathbb{E}[(\kappa - c_1 - c_2 \tilde{Z}- c_3 M )_{+}^2]},
 \end{equation}
 where $M \stackrel{\text{d}}{=} m_i, \tilde{Z} \sim\mathcal{N}(0,1)$, independent of $M$. Then, our  Theorems \ref{thm:l-1-margin}-\ref{thm:gen-error} once again characterize the max-$\ell_1$-margin and min-$\ell_1$-norm interpolant behavior (see Appendix \ref{app:gmmext} for further details) on substituting $F_{\kappa}(c_1,c_2)$ for $\tilde{F}_{\kappa}(c_1,c_2,c_3)$ and \eqref{eq:fix-points-l1} for the following system of four equations
\begin{align}\label{eq:gmmrank2}
c_1 = \mathbb{E}_{(\Lambda,h_{\star},\tilde{h},G) \sim \tilde{Q}} \big[ h_\star h_{\mathrm{sol}} \big], \quad & \quad c_2^2 = \mathbb{E}_{(\Lambda,h_{\star},\tilde{h},G) \sim \tilde{Q}} \left[ \big( \Lambda^{1/2} h_{\mathrm{sol}} \big)^2 \right], \nonumber \\
c_3 =\mathbb{E}_{(\Lambda,h_{\star},\tilde{h},G) \sim \tilde{Q}} \big[ \tilde{h} h_{\mathrm{sol}} \big], \quad & \quad 1= \mathbb{E}_{(\Lambda,h_{\star},\tilde{h},G) \sim \tilde{Q}} \big[ |h_{\mathrm{sol}}| \big] ,
\end{align}
$$\text{where} \qquad  h_{\mathrm{sol}}   =  - \frac{\mathrm{prox_s} (\Lambda^{1/2}G +\psi^{-1/2} (\partial_1 \tilde{F}_{\kappa}(c_1,c_2,c_3) h_{\star} + \partial_{3}\tilde{F}_{\kappa}(c_1,c_2,c_3)\tilde{h}))}{\Lambda \psi^{-1/2}c_2^{-1}\partial_2 \tilde{F}_{\kappa}(c_1,c_2,c_3)}. $$
Conceptually, adding an extra spike to $\text{Cov}(x_i)$ increases the complexity of the equation system by introducing a new variable $c_3$. We will observe a similar phenomenon if we were to look at more complicated analogues of \eqref{eq:rank2gmm} with a higher rank perturbation. In general, a rank-$\ell$ perturbation leads to an $(\ell+2)$-dimensional equation system analogous to \eqref{eq:gmmrank2}. Due to space constraints, we defer the general treatment to Appendix \ref{app:gmmext}.

For both the aforementioned models, the boosting algorithm satisfies Theorem \ref{thm:optimization-convergence} with the respective limiting characterization of the max-$\ell_1$-margin. A common theme across these settings is that the behavior of the margin and interpolant can be accurately characterized by a fixed point equation system, the solution to which possesses precise physical meanings (see \eqref{eq:angle} and the discussion thereafter). 
The form of the systems vary from one model to another; however, principles underlying its origin and key proof steps remain essentially the same (Section \ref{sec:deriv}). Once again, this is the power of our theoretical analysis in the $\ell_1$ case: we introduce a new uniform deviation argument with sufficient generality so that our proof can be adapted across several modeling schemes, as illustrated through this section.
\begin{figure}[h]
\centering 
\begin{minipage}{.45\textwidth} 
\centering 
\includegraphics[width=\linewidth]{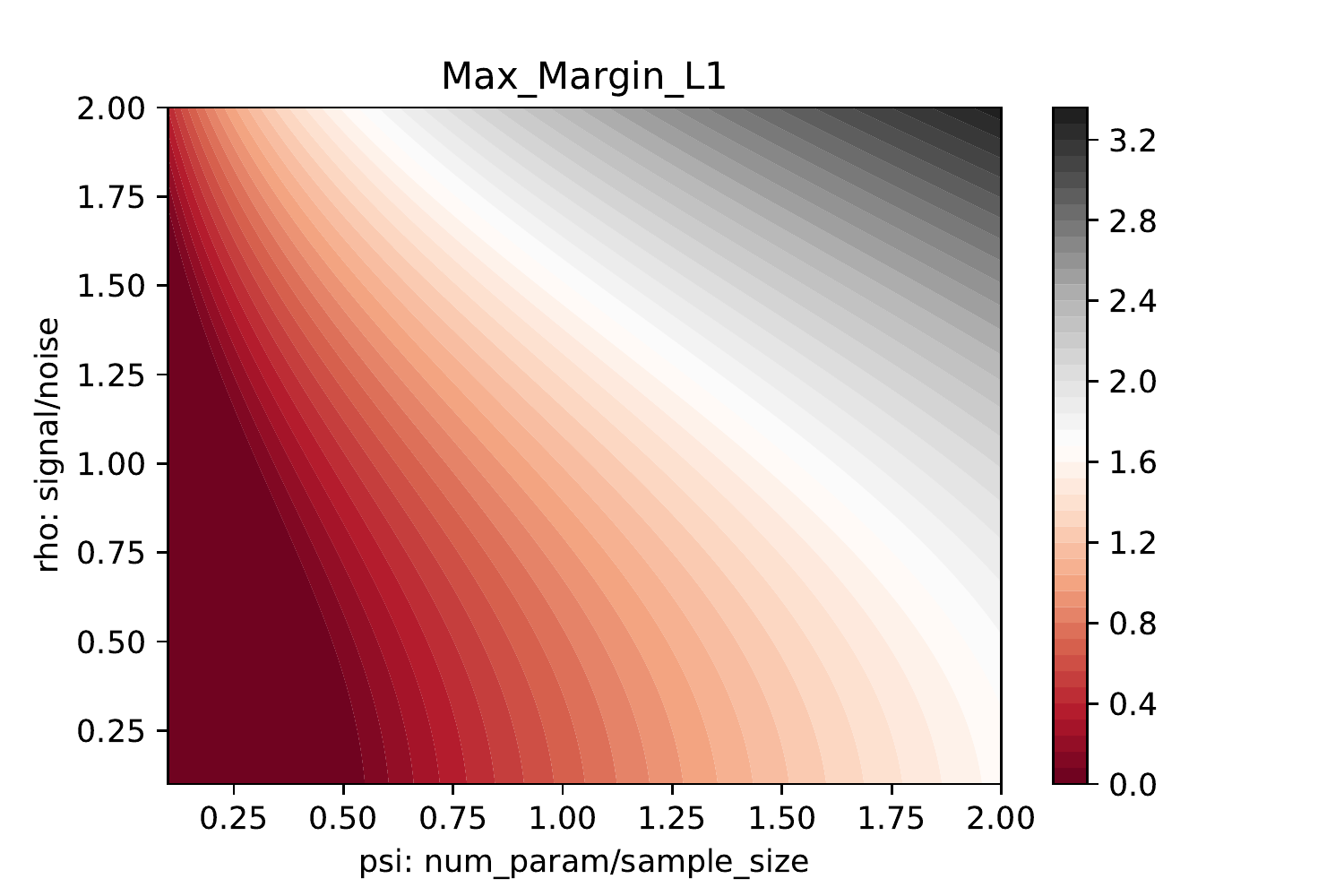}
\end{minipage}
\begin{minipage}{.45\textwidth} 
\centering 
\includegraphics[width=\linewidth]{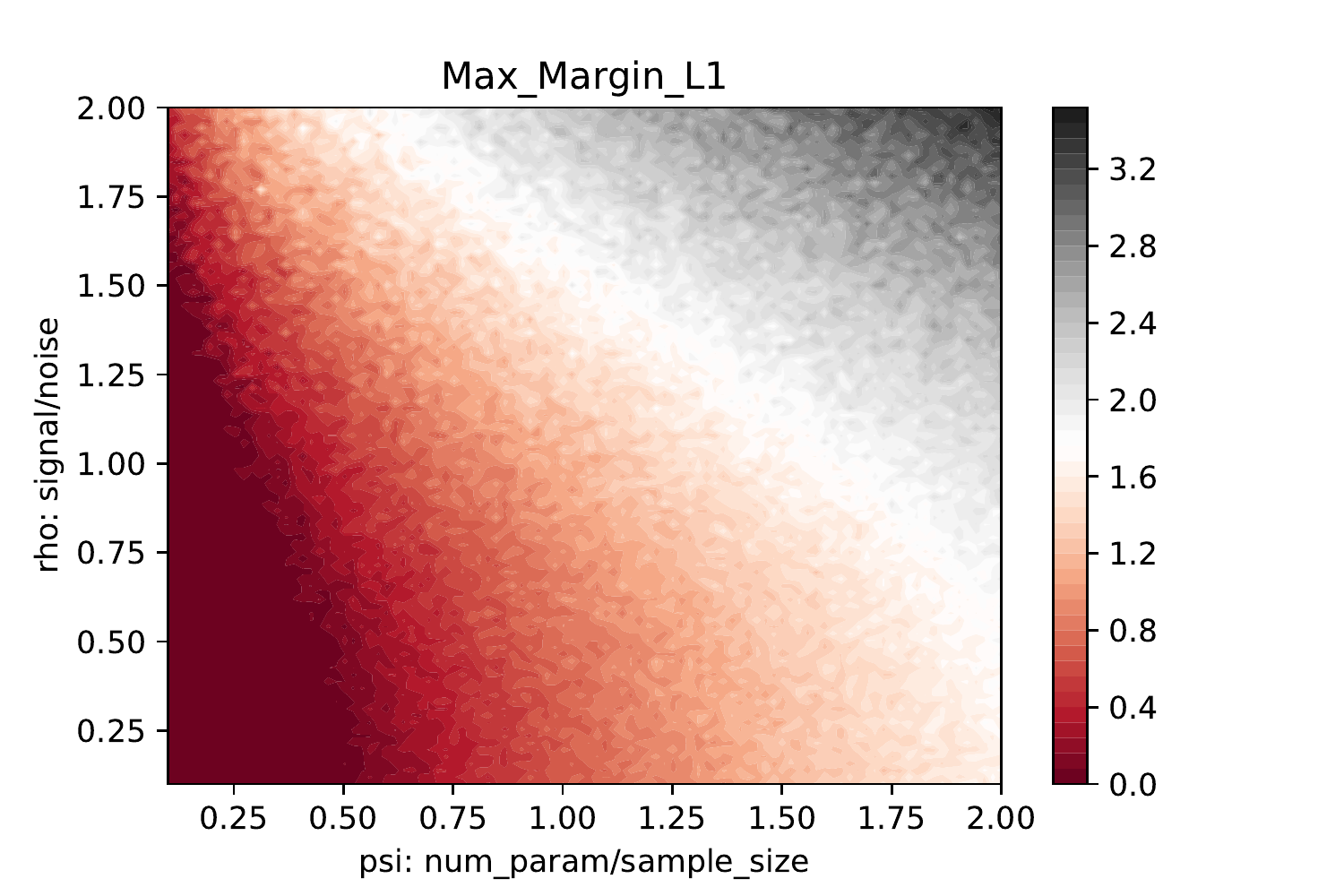}
\end{minipage}
\begin{minipage}{.45\textwidth} 
\centering 
\includegraphics[width=\linewidth]{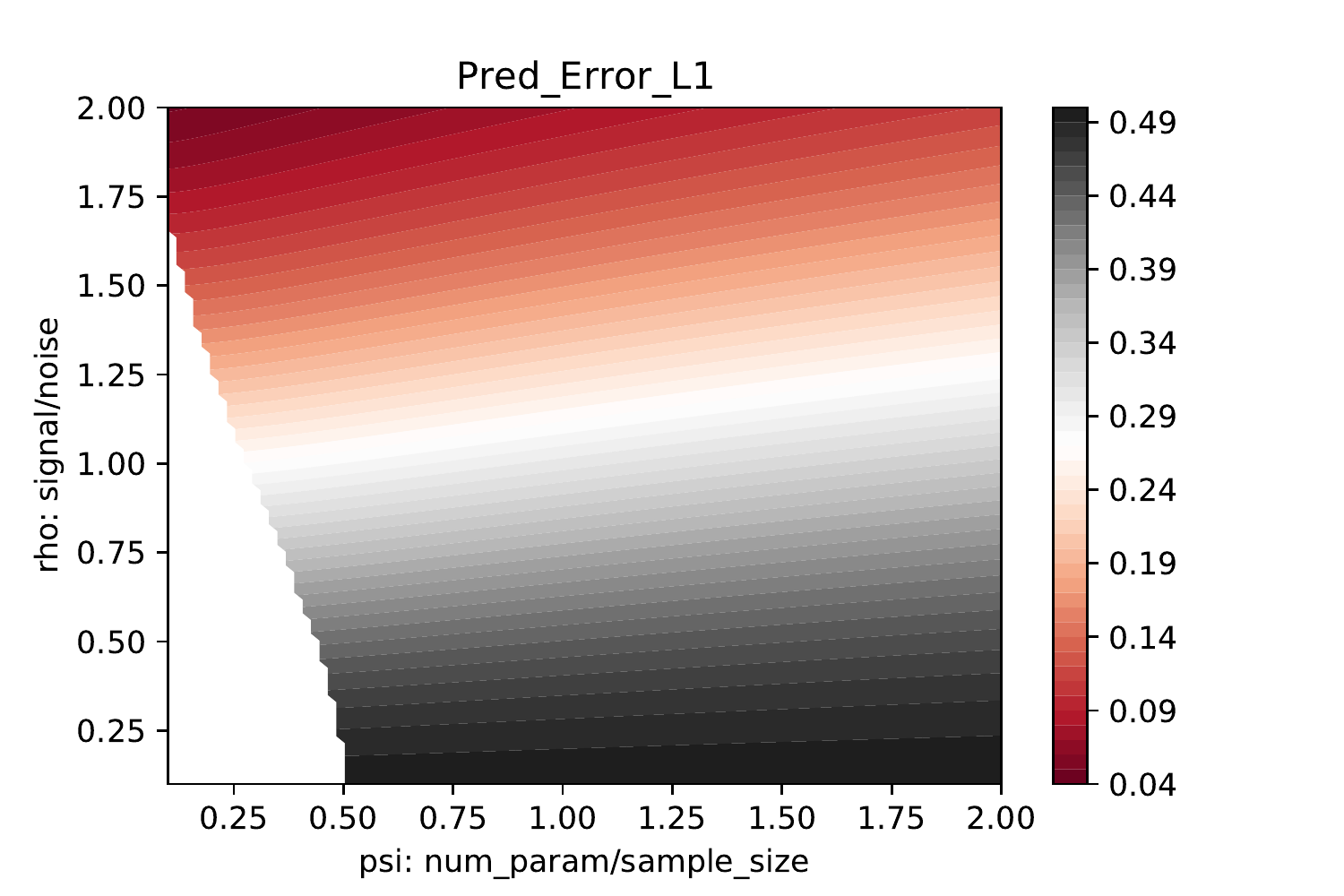}
\end{minipage}
\begin{minipage}{.45\textwidth} 
\centering 
\includegraphics[width=\linewidth]{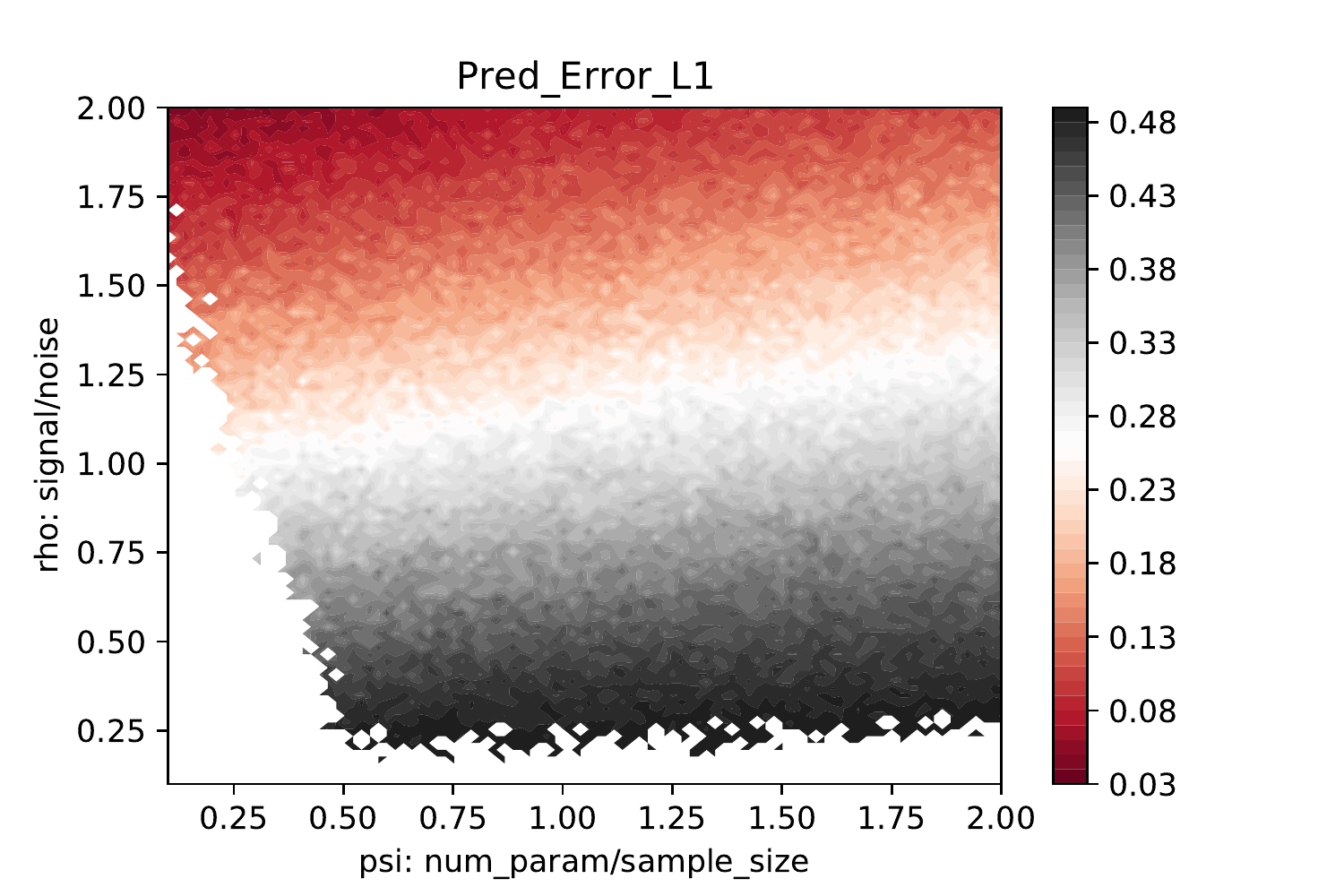}
\end{minipage}
\caption{$x$-axis: Ratio $ p/n$, $y$-axis: Signal-to-noise ratio $\rho = \left(\| \sqrt{p} \theta_\star \|^2/{\rm Tr}(\Lambda) \right)^{1/2}$. The top row shows max-$\ell_1$-margin and bottom row the prediction error of the corresponding interpolant. The left panel plots the limits of these objects, as characterized by our asymptotic theory, while the right panel shows the corresponding finite sample values obtained by solving  \eqref{eq:max-L1} using linear programming (averaged over two independent simulation runs to reduce noise).}
\label{fig:rank-one}
\end{figure}

To conclude this section, we showcase the numerical accuracy of our results for the rank-one spike case \eqref{eq:gmm1}-\eqref{eq:gmm2}. The example is illustrated in Fig.~\ref{fig:rank-one}.
Here, $\Lambda$ is always taken to be the identity matrix. The x-axis denotes the overparametrization ratio $\psi = p(n)/n$, y-axis the signal-to-noise ratio $\rho = \left(\| \sqrt{p} \theta_\star \|^2/{\rm Tr}(\Lambda) \right)^{1/2}$, and the color encodes the value of the max-$\ell_1$-margin (top row)  or prediction error of the corresponding min-$\ell_1$-norm interpolant (bottom row) respectively. Thus, for each value on the $y$-axis, we choose a different signal $\theta_{\star}$ so that the signal-to-noise ratio matches the given value of $\rho$. The left panel numerically solves the fixed-point equation \eqref{eq:fix-points-gmm} and presents the limits of the margin and prediction error from Theorems \ref{thm:l-1-margin}-\ref{thm:gen-error}, obtained upon replacing \eqref{eq:fix-points-l1} for the equation system in this rank-one spike case, \eqref{eq:fix-points-gmm}. The right panel presents the max-$\ell_1$-margin in finite samples, obtained by solving the LP  \eqref{eq:max-L1}, along with the corresponding prediction error, and these are averaged over two independent simulation runs. As Fig.~\ref{fig:rank-one} illustrates, the finite-sample results conform to our asymptotic characterization remarkably well. We defer further extensions to general feature covariance matrices not covered here or in Appendix \ref{app:gmmext} for future work. Remark \ref{rem:corrrem} (Section  \ref{sec:cgmtproof}) explains the additional difficulty faced in such extensions.

\subsubsection{Beyond Gaussian covariates} \label{sec:universality}

This section investigates the universality of the max-$\ell_1$-margin when the Boosting covariates are non-linear random features, which extends beyond Gaussianity. Non-linear random features are widely used in machine learning practice due to its connection to one-hidden-layer neural networks. To make the presentation clear, let us distinguish two concepts: the observed covariate-response pair $(x_i, y_i)$, and the Boosting covariate-response pair $(a_i, y_i)$. To this end, consider the covariate-response pair $\{ a_i \in \Reals^d, y_i\}_{i=1}^n$ fed into the Boosting Algorithm as stated in \eqref{sec:crucial} (with the substitution $Z := [y_1 a_1, \ldots, y_n a_n ]^\top \in \Reals^{n \times d}$ therein). Here we take  these ``actual covariates for boosting''  to be of the  form  $a_i = \sigma(F^\top x_i)$, with a non-linear activation function $\sigma(\cdot)$ applied entry-wise, and a random weight matrix $F \in \mathbb{R}^{p \times d}$ sampled independent of the observed $x_i$'s; thus, we call this random features. Note due to the non-linearity of $\sigma$, the boosting features $a_i$'s are non-Gaussian even when  $x_i$'s are Gaussian. 

This section will show that the max-$\ell_1$-margin for the above non-linear random features model, in the asymptotic sense, equals that of an analogous Gaussian features model, conditioned on $F$. To be concrete, we show the asymptotic equivalence of max-$\ell_1$-margin for two models: (i) random features $a_i = \sigma(F^{\top}x_i) \in \Reals^d$, and (ii) analogous Gaussian features $b_i = \mu_0 \boldsymbol{1}+\mu_1F^{\top}x_i+\mu_2 z_i \in \Reals^d$, where $z_i \sim \mathcal{N}(0,I_d), \mu_0 = \mathbb{E}[\sigma(Z)], \mu_1 = \mathbb{E}[Z \sigma(Z)], \mu_2 =\sqrt{\mathbb{E}(\sigma^2(Z))-\mu_0^2-\mu_1^2}$, with $Z \sim \mathcal{N}(0,1)$ independent of everything else. Here $\mu_0, \mu_1$ are top-two Hermite coefficients of $\sigma(\cdot)$, and $\mu_2$ is the $\ell_2$ norm of the remaining Hermite coefficients. The max-$\ell_1$-margin under each model is calculated using 
$\kappa_{n,\ell_1} (\{r_i, y_i \}_{1 \leq i \leq n}):=	\max_{\| \theta \|_1 \leq 1} \min_{1\leq i\leq n} ~y_i r_i^\top \theta,$ 
where $r_i$ equals $a_i$ or $b_i$ depending on the model. We establish that the asymptotic value of the margin (scaled by $\sqrt{p}$) remains the same irrespective of the choice of the features included in the calculation. 

To formalize this result, we consider a sequence of problem instances $\{y(n), X(n),\theta^{\star}(n) \}_{n \geq 1}$ satisfying the conditions in Section \ref{sec:crucial}, and in addition consider feature matrices $A(n), B(n)$ with the $i$-th row of $A(n)$ (resp.~$B(n)$) given by $a_i$ (resp.~$b_i$) described above. The sequence of random feature matrices $F(n)$ in the definition of $A(n)$ are taken to be of the form $F(n) = [f_1,\hdots,f_{d(n)}]$, where $f_i \sim \mathcal{N}(0,\boldsymbol{I}_p/p)$, and both $p(n), d(n)$ scale linearly with $n$.  In the sequel,  we suppress the dependence on $n$, whenever clear from context.

\begin{theorem}\label{thm:universalitythm}
Under the aforementioned conditions, if the non-linear function $\sigma(\cdot)$ is odd, compactly supported, and has bounded first, second and third derivatives, then the (rescaled) max-$\ell_1$-margin under both fitting procedures (i) and (ii) admit the same limit in probability, that is,  
\begin{equation}\label{eq:univmainresult}
p^{1/2} \cdot \kappa_{n, \ell_1}(\{a_i, y_i \}_{1 \leq i \leq n})  - p^{1/2} \cdot \kappa_{n,\ell_1} (\{b_i, y_i \}_{1 \leq i \leq n}) \stackrel{\mathbb{P}}{\rightarrow} 0. 
\end{equation}
\end{theorem}

The above theorem asserts that, asymptotically, both the non-linear feature matrix $A(n)$ and its Gaussian counterpart $B(n)$ yield the same margin value. We next provide a brief outline of the proof. In Section \ref{sec:cgmtproof}, we mention that studying the limiting value of the margin is equivalent to studying whether $\xi_{\psi, \kappa}^{(n,p)}(R)= \min_{\| \theta \|_1 \leq \sqrt{p}} \frac{1}{\sqrt{p}}\| (\kappa \mathbf{1} - (y\odot R)\theta)_+ \|_2$ is strictly positive or not, where $R$ denotes the feature matrix used in the margin definition. This is equivalent to studying
 $\{\xi_{\psi, \kappa}^{(n,p)}(R)\}^2 = \min_{\| \tilde\theta \|_1 \leq p} \frac{1}{p}\sum_{i=1}^n (\kappa  - \tfrac{1}{\sqrt{p}} y_i r_i^{\top} \tilde\theta)_{+}^2$, where we apply the change of variable $\tilde\theta = \sqrt{p} \theta$. Denote the Lagrange form for this problem with multiplier $\lambda$ to be $\Phi_n(R,\lambda)$.
 We claim that to show \eqref{eq:univmainresult}, it suffices to show that for all $\lambda$ 
\begin{equation}\label{eq:sufficient}
 \Phi_n(A,\lambda) - \Phi_n(B,\lambda)  \stackrel{\mathbb{P}}{\rightarrow}  0,
 \end{equation}
where $A,B$ are the feature matrices defined under the fitting procedures (i) and (ii) respectively. To see this, denote $\lambda_{A}$ to be the solution to the optimization problem 
\begin{equation}\label{eq:opt}
\frac{1}{p} \min_{ \| \tilde \theta \|_1 \leq p } \sup_{\lambda \geq 0}~    \sum_{i=1}^n (\kappa  - \tfrac{1}{\sqrt{p}} y_i a_i^{\top} \tilde \theta)_{+}^2 + \lambda \sum_{j =1}^p (  |\tilde \theta_j| - 1) 
\end{equation}
 Then, by duality of convex programs, we have that $\{\xi_{\psi, \kappa}^{(n,p)}(A)\}^2 = \Phi_n(A,\lambda_A)$. Furthermore, $\Phi_n(B,\lambda_A) \leq \frac{1}{p} \min_{ \| \tilde \theta \|_1 \leq p } ~    \sum_{i=1}^n (\kappa  - \tfrac{1}{\sqrt{p}} y_i b_i^{\top} \tilde \theta)_{+}^2 + \lambda_A \sum_{j =1}^p (  |\tilde \theta_j| - 1)  \leq  \{\xi_{\psi, \kappa}^{(n,p)}(B)\}^2 $. So far we have proved $\{\xi_{\psi, \kappa}^{(n,p)}(A)\}^2 \leq \{\xi_{\psi, \kappa}^{(n,p)}(B)\}^2 + o_{\mathbb{P}}(1)$. Analogously, denoting $\lambda_B$ to be the solution to the optimization problem in \eqref{eq:opt} with $a_i$ replaced by $b_i$, and applying \eqref{eq:sufficient} with $\lambda=\lambda_B$, we obtain that $\{\xi_{\psi, \kappa}^{(n,p)}(A)\}^2 - \{\xi_{\psi, \kappa}^{(n,p)}(B)\}^2 \stackrel{\mathbb{P}}{\rightarrow} 0$.

To prove \eqref{eq:sufficient}, we start with a leave-one-out argument adapted from \cite{hu2020universality}, which in turn builds upon \cite{el2018impact}. In \cite{hu2020universality}, the authors prove 
that the training and generalization errors are asymptotically equivalent in a random features model and a corresponding linearized model, where the covariates have matching moments and are Gaussian conditional on the random features. However, \cite{hu2020universality} defined the training error to be based on the objective function of a penalized empirical risk minimization problem, where the loss admits derivatives upto the third order and the regularizer is strongly convex. In our setting, neither of these properties hold, and this leads to several technical challenges. To handle these, we use a specific smoothing argument and develop several new analytic results (Appendix \ref{app:universality}). 

To supplement our universality result, Theorem \ref{thm:universalitythm}, we empirically check universality of our result across different covariate distributions used for the data-generation process. Note that this is different from the premise of Theorem \ref{thm:universalitythm}. For that Theorem, we considered the same data-generating distribution but different feature distribution for the covariates used in boosting, and established universality of the (asymptotic) max-$\ell_1$-margin across these settings. Now, we consider the setting of Figure \ref{fig:margin_n_error}, where the data is generated using a logistic model, and calculate the max-$\ell_1$-margin based on the linear program \eqref{eq:max-L1} (left subfigure), as well as difference between the test error and Bayes error (right subfigure), under two different settings. In the setting titled ``Rademacher", each entry of the observed design is taken to be  $\pm1$ with probability $1/2$, independently of each other. In the setting titled ``Gaussian", the  corresponding entries are i.i.d.~draws from a Gaussian distribution with  first and second moments matching that of the Rademacher. In both cases, the margin values from the linear program are averaged over $10$ independent runs. Observe the close match between the two settings, suggesting the applicability of our theory for a broader class of covariate distributions, beyond our theoretical results.

\begin{figure}[h]
\centering
 \includegraphics[width=\linewidth]{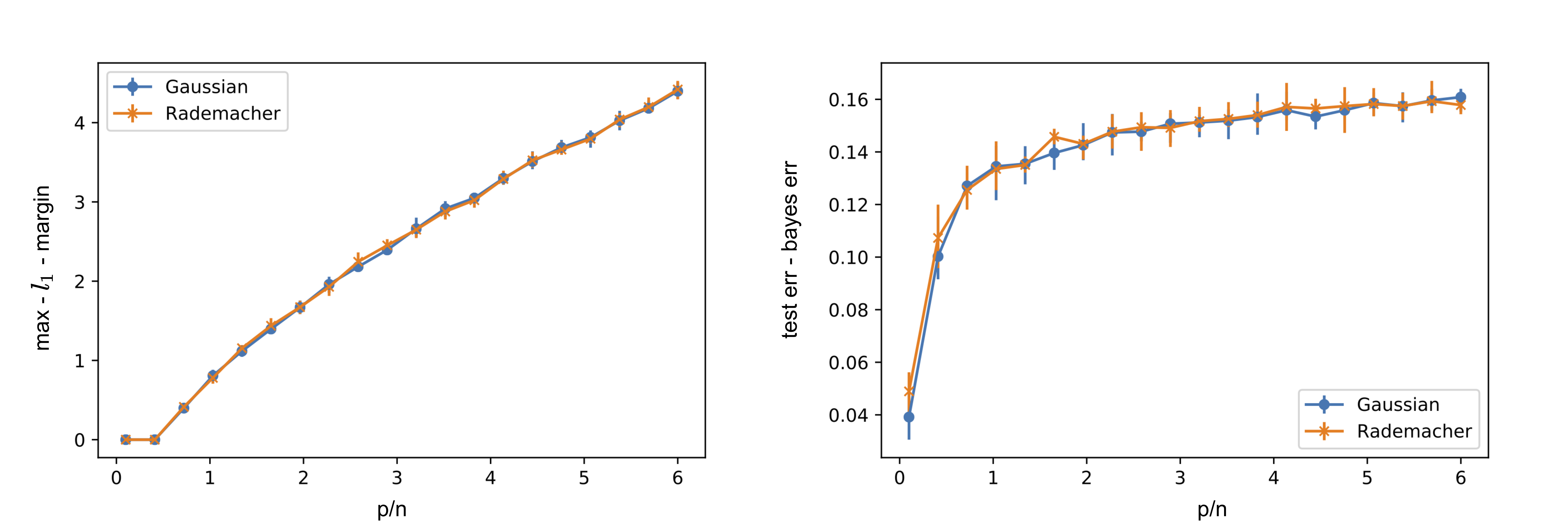}
 \caption{$x$-axis: Ratio $p/n$.   $y$-axis: (Left subfigure) max-$\ell_1$-margin, (Right subfigure) Test error minus the Bayes error. The figure has the same setting as in Figure \ref{fig:margin_n_error}, except the covariate distribution. Here, the observed design matrix has i.i.d.~entries drawn either from a Rademacher distribution or a Gaussian with matching first and second moments. The figure demonstrates universality of the margin value and the test error across these settings.}
\label{fig:robustness-check}
\end{figure}

\subsubsection{Model Misspecification}\label{subsec:modelmisspec}

Consider the following data generating process: denote $\tilde{x}_i = (x_i^\top, z_i^\top)^\top$ where $x_i \in \mathbb{R}^p$ and $z_i \in \mathbb{R}^q$, with $x_i\sim \cN(0, \Lambda_x)$ and $z_i \sim \cN(0, \Sigma_z)$ independent Gaussian vectors. Here we assume that $\Lambda_x$ is a diagonal matrix. Suppose that $y$ arises from the following conditional distribution
\begin{align}\label{eq:thetastar}
	\mathbb{P}(y_i =  +1 | \tilde x_i) = f( \tilde{x}_i^\top \theta_\star), ~\text{with}~ \theta_\star := (\theta_{x,\star}^\top , \theta_{z,\star}^\top)^\top.
\end{align}
The observed data contains $n$ i.i.d.~samples $(x_i \in \mathbb{R}^p, y_i \in \mathbb{R}), 1\leq i \leq n$, that is, only a part of the features $\tilde{x}_i$ that generate $y_i$ are included. Assume that both the seen and unseen components of the features have dimension that is large and comparable to the sample size. To model this, we assume that 
 \begin{align*}
	p(n)/n = \psi > 0, ~q(n)/n = \phi > 0.
\end{align*}
Consider that both components of $\theta_{\star}$, \eqref{eq:thetastar}, contribute a non-trivial signal strength, in the sense that
\begin{align*}
	\lim_{n \rightarrow \infty}  \big( \theta_{x, \star}^\top \Lambda_x \theta_{x, \star} \big)^{1/2} = \rho \enspace, \quad
	\lim_{n \rightarrow \infty}  \big( \theta_{z, \star}^\top \Sigma_z \theta_{z, \star} \big)^{1/2} = \gamma \enspace,
\end{align*}
where $0 < \rho, \gamma < \infty$. For any $\kappa \geq 0$, define a new function $\tilde{F}_{\kappa}: \mathbb{R} \times \mathbb{R}_{\geq 0} \rightarrow \mathbb{R}_{\geq 0}$, 
\begin{align}\label{eq:newF}
	\tilde{F}_{\kappa}(c_1, c_2) & := \left( \mathbb{E} \left[ (\kappa - c_1 YZ_1 - c_2 Z_3)_{+}^2 \right] \right)^{\frac{1}{2}} \nonumber \\ 
	\text{where}  \,\, \,\, & \,\,\,\,
	\begin{cases}
		Z_3 \perp (Y, Z_1, Z_2) \\
		Z_i \stackrel{i.i.d.}{\sim}\mathcal{N}(0, 1), ~i=1,2,3 \\
		\mathbb{P}(Y = +1|Z_1, Z_2) = 1 - \mathbb{P}(Y = -1|Z_1, Z_2) = f(\rho \cdot Z_1 + \gamma \cdot Z_2)
	\end{cases}
\end{align}
Consider the regime where the observed data is asymptotically linearly separable, that is, $\psi+\phi$ lies above the separability threshold for this problem. We do not describe the threshold here in detail, the interested reader may find its characterization in \cite[Proposition 3.1]{deng2019ModelDoubleDescent}. Then the max-$\ell_1$-margin and min-$\ell_1$-norm interpolant, computed using the observed data $\{(x_i,y_i)\}_{i=1}^n$ obey the same limiting characterizations as in Theorems \ref{thm:l-1-margin}-\ref{thm:gen-error}, with the system of equations remaining the same as in \eqref{eq:fix-points-l1}, but with $F_{\kappa}(c_1,c_2)$ substituted by the new function \eqref{eq:newF}. Thus, the form of the equation system \eqref{eq:fix-points-l1} once again remains unchanged, once we pin down the right analogue of $F_{\kappa}(c_1,c_2)$ in this new setting.


\section{Related Literature}
\label{sec:literature}

 This section discusses  prior literature that is relevant to our problem, but were omitted from Section \ref{sec:intro}.

\noindent \textbf{Boosting.} Since its introduction in \citep{freund1995desicion, freund1996experiments}, there has been a vast and expansive literature on Boosting. \cite{breiman1996arcing} studied bias and variance of general arcing classifiers. A wonderful survey of early works on generalization performance of boosting, and comparisons to the optimal Bayes error can be found in \cite{jiang2001some}.
Margin-based analyses were furthered in \cite{ratsch2001soft,koltchinskii2005complexities,ratsch2005efficient,reyzin2006boosting}.  For analysis of boosting algorithms based on smooth margin functions, see \cite{rudin2007analysis} and the references cited therein. Consistency properties were extensively studied in
\cite{mannor2001geometric,mannor2002ConsistencyGreedyAlgorithms,mannor2002existence,bickel2006some}. Aside AdaBoost, several variants of boosting emerged over the years, accompanied by many other perspectives.  Boosting for two class classifications may be viewed as   additive modeling
on the logistic scale \cite{friedman2000additive}. Subsequently, \cite{friedman2001greedy} developed a general gradient boosting framework.  The rate of convergence of regularized boosting classifiers was explored in \cite{blanchard2003rate}, where the authors uncovered that some versions of boosting work especially well in high-dimensional logistic additive models.
$\ell_2$-boosting, sparse boosting, twin boosting, and their properties in high dimensions were extensively studied in \cite{buhlmann2003boosting,buhlmann2007boosting,buhlmann2006sparse,buehlmann2006boosting,buhlmann2010twin}.
We remark that our setting is different in nature from this high-dimensional Boosting literature, where a notion of sparsity (often in $\ell_1$ geometry) is typically assumed on the unknown parameter $\theta_\star$. On the contrary, the $\ell_1$ connection arises naturally  in our setting, due to the nature of the AdaBoost/boosting algorithm. The rate of convergence of AdaBoost to the minimum of the exponential loss was investigated  in \cite{mukherjee2011rate}. Robust versions of boosting were proposed and extensively explored in \cite{li2018boosting}. In recent times, \cite{freund2017new} developed novel insights into boosting, by connecting classic boosting algorithms for linear regression to subgradient optimization and its siblings, which might be more amenable to mathematical analysis in several settings. 

\noindent \textbf{Convex Gaussian Minmax Theorem.} The Convex Gaussian Min-max Theorem is a generalized and tight version of the classical Gaussian comparison inequalities \cite{gordon1985some,gordon1988milman}, and is obtained by extending Gordon's inequalities with the presence of convexity. The idea of merging these seemingly disparate threads dates back to \cite{stojnic2013framework,stojnic2013meshes,stojnic2013upper}, where it was used to analyze the performance of the constrained LASSO in high signal-to-noise ratio regimes. The seminal works \cite{thrampoulidis2015regularized,thrampoulidis2014gaussian,thrampoulidis2018precise} built and significantly
extended on this idea to arrive at the CGMT, which was extremely useful for studying mean-squared errors of regularized M-estimators in high-dimensional linear models. As discussed earlier, \citep{montanari2019generalization} studied the asymptotic properties of the max-$\ell_2$-margin in binary classification settings, building upon CGMT-based techniques, and furthered the work by \citep{gardner1988space}.
In a similar setting,  \cite{deng2019ModelDoubleDescent} studied the excess risk obtained by running gradient descent, and explored the double descent phenomenon with a peak around the separability threshold. The CGMT has been used in several other contexts, both in high-dimensional statistics and information theory,  e.g. to characterize the performance of the SLOPE estimator in sparse linear regression \citep{hu2019asymptotics}, to study high-dimensional regularized estimators in logistic regression \cite{salehi2019impact}, and to establish performance guarantees for PhaseMax \citep{dhifallah2018phase}. The CGMT has proved useful in the study of high-dimensional convex problems, since it decouples a complex Gaussian process defined by a min-max objective function to a much simpler Gaussian process with essentially the same limit, yet much easier to analyze. However, this is merely a starting point or a basic building block.
The study of the reduced optimization problem is  entirely problem-specific and is usually rather challenging in most high-dimensional settings, often requiring the development of non-trivial probabilistic analysis (see Section \ref{sec:deriv} for specific details in our case).

 \noindent \textbf{Min-norm interpolation.} This paper investigates the min-$\ell_1$-norm interpolated classifier, which characterizes the limit of the Boosting solution on separable data. In recent years, min-norm interpolated solutions and their statistical properties have been extensively studied---see \citep{belkin2018understand, belkin2018does,liang2018just, belkin2019two, hastie2019surprises, bartlett2019benign,liang2019risk,liang2020mehler,bunea2020interpolation} for the regression problem,  and \citep{montanari2019generalization,deng2019ModelDoubleDescent, chatterji2020finite} for the classification problem. It has been conjectured that the implicit "min-norm" regularization, a version of the Occam's razor principle, is responsible for the superior statistical behavior of complex over-parametrized models \cite{zhang2016understanding,belkin2018understand,liang2018just}. To the best of our knowledge, the current paper is the first to provide sharp statistical results for interpolated classifiers induced by the $\ell_1$ geometry (rather than the $\ell_2$), which has been argued to be a more suitable geometry \cite{bach2017breaking, gunasekar2018characterizing, dou2020TrainingNeuralNetworks,chizat2020implicit,amid2020winnowing} for the limit of gradient flow on shallow neural networks with 2-homogenous activations. In this light, we expect our results to be of much broader utility beyond the context of boosting.


\section{Proof Sketch for Theorems \ref{thm:l-1-margin} and \ref{thm:gen-error} }
\label{sec:deriv}

The proofs of Theorems \ref{thm:l-1-margin} and \ref{thm:gen-error} rely on the \textit{Convex Gaussian Min-Max Theorem} (CGMT) \citep{thrampoulidis2015regularized,thrampoulidis2014gaussian}, which is a refinement of Gordon's classical Gaussian comparison inequality \citep{gordon1988milman}. Our analysis is partially influenced by the seminal work of \citep{montanari2019generalization}, which characterized the max-$\ell_2$-margin using CGMT-based techniques. However, characterizing the asymptotics for the $\ell_1$ case requires establishing a novel and stronger form of a uniform deviation argument (outlined in Step 3 below); this relies on a key \emph{self-normalizing} property of $F_{\kappa}$, which might be of standalone interest (we establish this in Lemma~\ref{lem:self-normalized-uniform-devation}). Additionally, our analysis is general and extendable to the max-$\ell_q$-margin case with $1\leq q\leq 2.$ Below, we provide a sketch of the main proof ideas.

\subsection{Proofs of Theorems~\ref{thm:l-1-margin} and \ref{thm:gen-error}}\label{sec:cgmtproof}

\noindent \textbf{Step 1: A basic reduction.} 
To begin with, define 
\begin{align}
	\label{eq:original-n-p}
	\xi_{\psi, \kappa}^{(n,p)} &:= \min_{\| \theta \|_1 \leq \sqrt{p}} ~\max_{\| \lambda \|_2 \leq 1, \lambda \geq 0} ~ \frac{1}{\sqrt{p}} \lambda^T (\kappa \mathbf{1} - (y\odot X)\theta) \\
	&= \min_{\| \theta \|_1 \leq \sqrt{p}} \frac{1}{\sqrt{p}}\| (\kappa \mathbf{1} - (y\odot X)\theta)_+ \|_2 \enspace. \nonumber
\end{align}
It is not hard to see that
\begin{align}\label{eq:equivalence}
	\xi_{\psi, \kappa}^{(n,p)} = 0, ~~\text{if and only if}~~ \kappa \leq p^{1/2}\cdot \kappa_{\ell_1}\left(\{x_i, y_i\}_{i=1}^n \right) \enspace, \nonumber\\
	\xi_{\psi, \kappa}^{(n,p)} > 0, ~~\text{if and only if}~~ \kappa > p^{1/2}\cdot \kappa_{\ell_1}\left(\{x_i, y_i\}_{i=1}^n \right) \enspace.
\end{align}
Thus, to study the rescaled max-$\ell_1$-margin, it suffices to examine the value of  $\xi_{\psi, \kappa}^{(n,p)}$.

Now, defining $z_i := \Lambda^{-1/2} x_i ~\forall i \in [n]$, where $\Lambda$ is the covariance matrix, we may express 

\begin{align}\label{eq:newwdef}
	x_i^\top \theta_\star = z_i^\top \Lambda^{1/2} \theta_\star =\rho(n) \cdot z_i^\top w, ~~\text{where}~ w:= \Lambda^{1/2} \theta_\star/ \| \Lambda^{1/2} \theta_\star \| \enspace.
\end{align}

Using the fact that  $y \odot X = (y \odot Z) \Lambda^{1/2} ~{\buildrel d \over =}~ \left( (y \odot z) w^\top + Z \Pi_{w^\perp} \right)  \Lambda^{1/2}$  (such a trick was first used in the literature in \cite{thrampoulidis2015lasso}), where  $z \in \mathbb{R}^{n}$,$Z \in \mathbb{R}^{n \times p}$ are independent of each other, each containing independent standard Gaussian entries,
Eqn.~\eqref{eq:original-n-p} then reduces to
\begin{align}
	\label{eq:cgmt-form}
	\xi^{(n,p)}_{\psi, \kappa}(z, Z) := \min_{\| \theta \|_1 \leq \sqrt{p}} ~\max_{\| \lambda \|_2 \leq 1, \lambda \geq 0} ~ \frac{1}{\sqrt{p}} \lambda^T \left(\kappa \mathbf{1} - (y\odot z) \langle w,  \Lambda^{1/2}\theta \rangle \right) \nonumber \\
	 -  \frac{1}{\sqrt{p}} \lambda^T Z \Pi_{w^\perp}(\Lambda^{1/2}\theta)  \enspace.
\end{align}
 \begin{remark}
 The rescaling by $\sqrt{p}$ is required to ensure a well-defined limit for the max-$\ell_1$-margin (in general, a rescaling by $p^{1/q - 1/2}$ is required for general $\ell_q$ margin, as evidenced via Corollary \ref{eq:max-lq}, and this immediately shows that no rescaling is required for the $\ell_2$ case \cite{{montanari2019generalization}}).
 \end{remark}

\noindent \textbf{Step 2: Reduction to Gordon's problem.}
Due to the min-max form of \eqref{eq:cgmt-form}, one can use Gordon's Gaussian comparison inequality \citep{thrampoulidis2015regularized,thrampoulidis2014gaussian, gordon1988milman} to further simplify the problem. To this end, introduce the following ``de-coupled'' optimization problem

\begin{align}
	& \hat \xi^{(n,p)}_{\psi, \kappa}(z, \tilde z, g)  := \min_{\| \theta \|_1 \leq \sqrt{p}} ~\max_{\| \lambda \|_2 \leq 1, \lambda \geq 0} ~ \frac{1}{\sqrt{p}} \lambda^T\mathcal{V} +  \frac{1}{\sqrt{p}} \| \lambda \|_2 \langle g, \Pi_{w^\perp}(\Lambda^{1/2}\theta)\rangle  \nonumber \\
	&= {\small \left[ \min_{\| \theta \|_1 \leq \sqrt{p}}  \frac{1}{\sqrt{p}}\left\| \mathcal{V}_+ \right\|_2 
	 + \frac{1}{\sqrt{p}} \left\langle \Pi_{w^\perp}(g), \Lambda^{1/2} \theta \right\rangle \right]_+}, \label{eq:cgmtred}
\end{align}
where $\mathcal{V}= \left(\kappa \mathbf{1} - (y\odot z) \langle w,  \Lambda^{1/2}\theta \rangle -  \tilde z \| \Pi_{w^\perp}(\Lambda^{1/2}\theta) \|_2 \right)$, $z, \tilde{z}\in \mathbb{R}^n$ and $g \in \mathbb{R}^p$ are independent isotropic Gaussian vectors.
By CGMT \citep[Theorem 3]{thrampoulidis2015regularized} (see Theorem \ref{thm:cgmt} in the Appendix), we have
\begin{align}
	\mathbb{P}\left( \xi_{\psi, \kappa}^{(n,p)}(z, Z) \leq t |y, z \right) \leq 2 \mathbb{P}\left( \hat \xi_{ \psi, \kappa}^{(n,p)}(z, \tilde z, g) \leq t |y, z \right) \\
	\mathbb{P}\left( \xi_{ \psi, \kappa}^{(n,p)}(z, Z) \geq t |y, z \right) \leq 2 \mathbb{P}\left( \hat \xi_{ \psi, \kappa}^{(n,p)}(z, \tilde z, g) \geq t |y, z \right).
\end{align}
Marginalizing over $y$ and $ z$, this suggests that it suffices to study \eqref{eq:cgmtred}.

\noindent \textbf{Step 3: The key step---large $n,p$ limit, new uniform deviation result.}

Recall the function $F_{\kappa}(\cdot,\cdot)$ from \eqref{eq:YZ}, and define the empirical version
\begin{align}\label{eq:Fkhat}
	\widehat F_{\kappa}(c_1, c_2) :=  \left( \widehat \E_n[ (\kappa - c_1 YZ_1 - c_2 Z_2)_{+}^2] \right)^{1/2} \enspace,
\end{align}
where $\widehat{ \E}_n$ means that the expectation over $Y, Z_1, Z_2$ is taken with respect to the empirical distribution of $\{(Y_i, Z_{1,i},Z_{2,i}) \}_{i=1}^n$, with entries $(Y_i, Z_{1,i},Z_{2,i})$ arising from the joint distribution specified in \eqref{eq:YZ}.
Then with $\lambda = {\rm diag}(\Lambda)$ denoting the vectorized $\Lambda$, we can express $\hat \xi^{(n,p)}_{\psi, \kappa}(z, \tilde z, g)$ as the positive part of the following expression
\begin{align}
	\label{eq:finite}
&	\hat \xi^{(n,p)}_{\psi, \kappa}(\lambda, w, g) := \nonumber\\ 
&	\min_{\|\theta\|_1 \leq \sqrt{p}} \left[ \psi^{-1/2} \widehat F_{\kappa} \left( \langle w,  \Lambda^{1/2}\theta \rangle,\| \Pi_{w^\perp}(\Lambda^{1/2}\theta) \|_2  \right) + \frac{1}{\sqrt{p}} \left\langle \Pi_{w^\perp}(g), \Lambda^{1/2} \theta \right\rangle \right] \enspace.
\end{align}
Note that $\hat \xi^{(n,p)}_{\psi, \kappa}(\lambda, w, g)$ is a random quantity, here we denote $\lambda, w, g$ as arguments to make explicit the dependence. 

We seek to study \eqref{eq:finite} in the large sample and feature limits  $n, p\rightarrow\infty$ with $p/n \rightarrow \psi$. On taking limits naively, one can reach the following infinite-dimensional convex problem, 
\begin{align}\label{eq:infinite}
&	\tilde \xi^{(\infty, \infty)}_{\psi, \kappa}( \Lambda, W, G) :=  \nonumber\\
&	 \min_{\| h \|_{L_1(\cQ)} \leq 1} \left[  \psi^{-1/2} F_{\kappa} \left( \langle W,  \Lambda^{1/2} h \rangle_{L_2(\cQ)},\| \Pi_{W^\perp}(\Lambda^{1/2}h) \|_{L_2(\cQ)}  \right) +  \left\langle \Pi_{W^\perp}(G), \Lambda^{1/2} h \right\rangle_{L_2(\cQ)} \right].
\end{align}
Here, the optimization variable is the set of function $\{ h : \Reals^3 \rightarrow \Reals, h \in \mathcal{L}^2(\cQ) \}$, where $\cQ= \mu \otimes \mathcal{N}(0,1)$ with $\mu $ defined as in \eqref{eq:limit-measure-2}.

Proposition~\ref{prop:large-n-p-limit} rigorously proves that the empirical optimization problem $\hat \xi^{(n,p)}_{\psi, \kappa}(\lambda, w, g)$ converges to the infinite dimensional problem $\tilde \xi^{(\infty, \infty)}_{\psi, \kappa}( \Lambda, W, G)$,
 almost surely, that is,
\begin{align}\label{eq:finitetoinfiniteopt}
	\lim_{n\rightarrow \infty, p(n)/n = \psi}~ \hat \xi^{(n,p)}_{\psi, \kappa}(\lambda, w, g)  \stackrel{\mathrm{a.s.}}{=}  \tilde \xi^{(\infty, \infty)}_{\psi, \kappa}( \Lambda, W, G). 
\end{align}
We provide an outline of the proof below, deferring the details to Section~\ref{sec:prob-analysis}.

Our \textit{technical innovation} lies in the development of \eqref{eq:finitetoinfiniteopt}, which requires establishing a uniform deviation bound over an unbounded region. To describe further, observe that  $\hat \xi^{(n,p)}_{\psi, \kappa}(\lambda, w, g)$ involves $\hat{F}_\kappa$ evaluated at the points $c_1 =  \langle w,  \Lambda^{1/2}\theta \rangle$ and $c_2 = \| \Pi_{w^\perp}(\Lambda^{1/2}\theta) \|_2$. It is clear that both under the $\ell_2$-constraint $\|  \theta \|_2 \leq 1$ (the setting of \citep{montanari2019generalization}) and the $\ell_1$-constraint $\| \theta \|_1 \leq \sqrt{p}$ (our setting), $c_1$ is bounded in the sense $|c_1|\leq M$ for all $p(n), n$ and some constant $M>0$; for the $\ell_1$ case,  this follows by noting that  $$|\langle w,  \Lambda^{1/2}\theta \rangle| \leq \frac{1}{c} \cdot \| w \|_\infty \| \theta\|_1 = \frac{1}{c} \cdot \|\bar w \|_\infty/\sqrt{p} \cdot \| \theta\|_1 \leq C'/c,$$ by Assumption~\ref{asmp:boundedness}. Turning to the second variable $c_2$, we see that under our $\ell_1$-constraint,  $c_2$  may potentially grow as $\sqrt{p}$ whereas it remains bounded when the $\ell_2$-norm of $\theta$ is bounded.
Naturally, the unbounded region for $c_2$ creates significant challenges in establishing \eqref{eq:finitetoinfiniteopt} in our setting. Naive covering arguments to establish the aforementioned uniform deviation for $c_2 \in [0, \infty)$ fail to deliver sharp results. To overcome this technical challenge, we discover a key self-normalization property of the partial derivatives of $F_{\kappa}$ (Appendix \ref{sec:prob-analysis}), utilizing the structure of this function, and prove the following. 
\begin{lemma}[Self-normalization and uniform deviation]
	\label{lem:self-normalized-uniform-devation}
	For $i = 1, 2$, with probability at least $1 - n^{-2}$, 
	\begin{align}\label{eq:derivconverge}
		\sup_{|c_1|\leq M, c_2>0} |\partial_i \widehat F_{\kappa}(c_1, c_2) - \partial_i F_{\kappa}(c_1, c_2)| \leq C \cdot \frac{\log n}{\sqrt{n}} \enspace,
	\end{align}
	where $C$ is a constant that does not depend on $n$.
\end{lemma}

Our proof proceeds as follows: (a) The first and key step is to establish Lemma \ref{lem:self-normalized-uniform-devation}.
 (b) Thereafter, we establish that the ``empirical fixed point (fp) equations" obtained by analyzing the KKT conditions for \eqref{eq:finite} (this finite $n,p$ problem is not convex in $\theta$, therefore, the KKT conditions are merely necessary conditions in this case) converge uniformly, over an unbounded region for $c_2$, to the corresponding ``fp equations obtained from the KKT conditions for \eqref{eq:infinite}". (The KKT conditions are both necessary and sufficient in this case. See Appendix \ref{sec:prob-analysis} for details.) The convergence here is in the sense of \eqref{eqn:uniform-deviation-summary}. The analysis uses the key Lemma \ref{lem:self-normalized-uniform-devation}. See Step 4 for description of these KKT equations.  
(c)  Leveraging (b), we show that any solution $(\hat{c}_1, \hat{c}_2, \hat{s})$ of the empirical fp equations converges to the unique solution $(c_1^\star,c_2^\star,s^\star)$ of the fp equations from \eqref{eq:infinite}. See Appendix \ref{sec:uniqueness-result} for uniqueness of the solution.  (d) Now, \eqref{eq:finite} can be expressed as functions of $\hat{s}$ and $\hat{F}_\kappa, \partial_i \hat{F}_{\kappa}, i=1,2$, evaluated at $(\hat{c}_1, \hat{c}_2)$, and similarly, for \eqref{eq:infinite} with  $s^\star$ and $F_\kappa, \partial_i F_{\kappa}, i=1,2$ evaluated at $(c_1^\star, c_2^\star)$. Given (c), we have proved that $(\hat{c}_1, \hat{c}_2, \hat{s})$ will be bounded for sufficiently large $n$, and  therefore, uniform deviation bounds for $|\hat{F}_\kappa -F_{\kappa}|$ can also be established.  This series of arguments enables us to establish \eqref{eq:finitetoinfiniteopt}, under a potentially complicated $\ell_1$ geometry.
A critical, and perhaps surprising, consequence of our uniform deviation results is a localization property: any optimizer of \eqref{eq:finite} possesses finite $\ell_2$-norm.

\noindent \textbf{Step 4: Fixed point equations and final step.}

By standard analysis arguments (see Appendix \ref{sec:uniqueness-result}), the KKT conditions for the optimization problem \eqref{eq:infinite} can be expressed as 

\begin{align}
  \Pi_{W^\perp}(G) + \psi^{-1/2}  \left[ \partial_1 F_{\kappa}(c_1, c_2) W + c_2^{-1} \partial_2 F_{\kappa}(c_1, c_2) (\Lambda^{1/2} h - c_1 W) \right] \nonumber \\
  + s \cdot \Lambda^{-1/2} \partial \| h \|_{L_1(\cQ)} = 0, \nonumber \\
	\mathrm{and} \quad  \| h\|_{L_1(\cQ)} = 1, \quad \mathrm{where} \quad c_1 := \langle \Lambda^{1/2}  h, W \rangle_{L_2(\cQ)}, \quad c_2 := \| \Pi_{W^\perp}(\Lambda^{1/2} h) \|_{L_2(\cQ)}. \label{eq:threeequations}
\end{align}

From properties of the proximal mapping operator, the KKT conditions suggest that the solution must satisfy (see Appendix \ref{sec:uniqueness-result} for a derivation of this claim, and the proof of uniqueness of the solution)

\begin{align}\label{eq:hsol}
	 h= -\frac{ \Lambda^{-1} \prox_{s}\left( \Lambda^{1/2}  G + \psi^{-1/2} [\partial_1 F_{\kappa}(c_1, c_2) - c_1 c_2^{-1} \partial_2 F_{\kappa}(c_1, c_2)] \Lambda^{1/2}  W  \right)}{ \psi^{-1/2} c_2^{-1} \partial_2 F_\kappa(c_1, c_2)  } .
\end{align}

Plugging this in the three equations displayed in \eqref{eq:threeequations}, leads to the ``fp equations $\hdots$ for \eqref{eq:infinite}", referred to in Step 3, which is the exact same as the equation system \eqref{eq:fix-points-l1}, thus explaining the origin of the system. A similar analysis  for \eqref{eq:finite} leads to the ``empirical fp equations" referred to in Step 3 (see \eqref{eq:empiricalform} for the specific form). 
Finally, Corollary \ref{cor:values} shows that $\tilde \xi^{(\infty, \infty)}_{\psi, \kappa}( \Lambda, W, G) = T(\psi,\kappa)$; together with \eqref{eq:equivalence} and \eqref{eq:finitetoinfiniteopt}, this completes the proof. 

Note that \eqref{eq:hsol} explains how $s^{\star}$ from Section \ref{subsec:margin} (the third component in the solution to our system of equations \eqref{eq:fix-points-l1}) corresponds to Lagrange multipliers induced by the $\ell_1$ constraint in~\eqref{eq:threeequations}.

\begin{remark}\label{rem:corrrem}
As described in Section \ref{subsec:gmmext}, the above proof path can accommodate a broad class of feature covariance matrices that are finite-rank perturbations of a diagonal (see Appendix \ref{app:gmmext} for the general case extension). However, further extensions beyond this class poses an additional challenge: a crucial step in our proofs lies in establishing a large sample limit of a finite dimensional optimization problem (e.g. that in \eqref{eq:finite} or \eqref{eq:finitedim}). Here, the limit is described in terms of an optimization problem over $\{ h : \mathbb{R}^3 \rightarrow \mathbb{R}, h \in \mathcal{L}^2(\mathcal{Q}) \}$, where $\mathcal{Q}= \mu \otimes \mathcal{N}(0,1)$ (\eqref{eq:infinite} or \eqref{eq:limitgmm}). In the case of $\ell_1$ geometry (and other $\ell_q$ geometry for $q \neq 2$), going over to this function space is feasible for a diagonal covariance matrix or non-diagonal matrices with a special structure. Instead, if $\Lambda$ were a general non-diagonal matrix (not included in the classes considered in Section \ref{subsec:gmmext} and Appendix \ref{app:gmmext}), this leads to an added challenge. The finite sample optimization problem still retains a similar form as in \eqref{eq:finite}, however, it is unclear how to express its limit in a convenient way and handle the terms $\Lambda^{1/2}\theta$. Given that the $\ell_1$ theory requires several technical contributions over prior works, as described in this section, and Sections \ref{subsec:gmmext} and \ref{sec:universality}, we defer the case of more general covariance matrices for future work. We comment that  the challenge faced here is similar in spirit to that seen in the context of the Lasso under arbitrary covariance, when studied under our proportional asymptotics regime. Here,  one can be establish that this problem is asymptotically equivalent to a Gaussian sequence model with correlated errors. Now, this latter model is complicated and extracting neat characterizations from this equivalent problem remains quite a challenge (see for instance \cite{celentano2020lasso,alrashdi2020precise,huang2021lasso} for some progress in this direction).
\end{remark}

\subsection{Proofs of Theorems~\ref{thm:optimization-convergence} and Corollary \ref{thm:feature-selection}}

\citep{zhang2005boosting} employs a re-scaling technique to establish that Boosting with infinitesimal stepsize agrees with the \textit{min-$\ell_1$-norm} direction asymptotically. Since we care about the actual number of iterations in the Boosting algorithm (which translates to the number of selected features), here we use a simple yet general analysis of Boosting as a special instance of Mirror Descent (the connection between AdaBoost and mirror descent is well-known and it is infeasible to provide a complete list of references establishing and utilizing this. We refer the interested reader to Sections \ref{sec:intro} and \ref{subsec:boosting} for a partial list of related works) in conjunction with the re-scaling technique \citep{zhang2005boosting} and the shrinkage technique \citep{telgarsky2013margins} (note this latter work also develops a $1/\sqrt{t}$ margin maximization rate). Our analysis
  is similar in spirit to \citep{collins2002logistic}, but with different executions. One benefit of our analysis is that it is easily generalizable to a variant of boosting algorithm that maximizes $\ell_q$ margin with $q\geq 1$.

\begin{proposition}\label{prop:AdaBoost} 
Consider the $\gba$ stated in Section~\ref{sec:crucial}. Assume that $|X_{ij}| \leq M$  for $i \in [n], j \in [p]$. Consider the learning rate $\alpha_{t}(\beta) = \beta \cdot \eta_t^\top Z v_{t+1}$, with $\beta = 1/M^2$. When

\begin{align}
T \geq \frac{2M^2}{ \kappa_{n,\ell_1}^2 }\log \frac{ne}{\epsilon},
\end{align}
the $\gba$ iterates $\theta_{T}$ will satisfy $
\sum_{i \in [n]} 1_{x_i^\top \theta_{T} \leq 0} \leq \epsilon.$
\end{proposition}

\begin{corollary}[Boosting converges to max-$\ell_1$-margin direction]
	\label{cor:boost-converge-max-margin}
	Consider the general Boosting algorithm with  learning rate 
	$\alpha_{t}(\beta) : = \beta \cdot \eta_t^\top Z v_{t+1},$ where $\beta<1$.
	Assume that $|X_{ij}| \leq M$ for $i \in [n], j \in [p]$. 
	Then after $T$ iterations, the Boosting iterates $\theta_T$ converge to the max-$\ell_1$-margin Direction in the following sense: for any $0 <\epsilon < 1$,
	\begin{align}
	 \kappa_{n,\ell_1} \geq \min_{i \in [n]} \frac{y_i x_i^\top \theta_T}{\| \theta_T \|_1} >  \kappa_{n,\ell_1} \cdot (1- \epsilon),
	\end{align}
	where 
	$
		T \geq \log(1.01 ne) \cdot \frac{2M^2  \epsilon^{-2}}{\kappa_{n,\ell_1}^2},
	$	with $\beta = \frac{\epsilon}{M^2}$.
\end{corollary}

To obtain Theorems~\ref{thm:optimization-convergence} and Corollary \ref{thm:feature-selection}, we choose $M(\delta) = \sqrt{(3+\delta)\log(np)}$ for arbitrarily small $\delta>0$ and recall that $\sqrt{p}  \kappa_{n,\ell_1} \stackrel{\text{a.s.}}{\rightarrow} \kappa_\star(\psi, \rho, \mu) $. Now, the entries $X_{ij}$ are uniformly bounded above by $M$  asymptotically almost surely, since $\mathbb{P}( \sup_{i \in [n], j\in [p]}~|X_{ij}| \leq M(\delta)) \leq np \exp(-M^2(\delta)/2) = n^{-1-\delta}$ and $\sum_{n \geq 1} n^{-1-\delta} < \infty$. Plugging in $\epsilon = 0.99$ in Proposition~\ref{prop:AdaBoost}, with the aforementioned $M$, establishes the almost sure result in Theorem \ref{thm:feature-selection}. The constant $12$ can be justified since $\lim_{\delta \rightarrow 0}~2M^2(\delta)/\log n= 12$.


\section{Discussion}\label{sec:discussion}
This paper establishes a high-dimensional asymptotic theory for AdaBoost and develops precise characterizations for both its generalization and optimization properties.
This is achieved through an in-depth study of the max-$\ell_1$-margin, the min-$\ell_1$-norm interpolant, and a  sharp analysis of the time necessary for AdaBoost to approximate this interpolant arbitrarily well. 
In doing so, this work identifies the exact quantities that govern the generalization behavior of AdaBoost for a class of data-generation models, and the relationship between this test error and the optimal Bayes error.  On the optimization front, we further uncover how overparametrization leads to faster optimization. The proposed theory demonstrates commendable finite sample behavior, applies for a broad class of statistical models, and is empirically robust to violations of certain assumptions. Natural variants of AdaBoost that correspond to max-$\ell_q$-margins for $q > 1$, are further analyzed. 

We conclude with a couple of directions of future research: it would be of interest (a) to rigorously characterize analogous properties of AdaBoost for covariate distributions with arbitrary correlations; this is a particularly challenging task for general $\ell_q$ geometry when $q \neq 2$, as explained in Remark \eqref{rem:corrrem}, and  (c) to complement such characterizations via data-driven schemes for estimating the parameters $c_1^{\star},c_2^{\star}$ that govern properties of the $\ell_1$ margin and interpolant, as well as the generalization performance of AdaBoost. Such estimation schemes are expected to be useful for providing recommendations regarding algorithm choice to practitioners.


\bibliography{reference}


\newpage
\appendix

\section{Main Proofs}

\subsection{The Convex Gaussian Min-Max Theorem}
For the convenience of the readers, we state the convex Gaussian min-max theorem below \citep[Theorem 4]{thrampoulidis2015regularized} (see also \cite{gordon1988milman})
\begin{theorem}\label{thm:cgmt}
Let $\Omega_1 \subset \Reals^n, \Omega_2 \subset \Reals^p$ be two compact sets and let $U: \Omega_1 \times \Omega_2 \rightarrow \Reals$ be a continuous function. Let $Z=(Z_{i,j})\in \Reals^{n \times p}, g \sim \mathcal{N}(0,I_n) $ and $ h \sim\mathcal{N}(0, I_p)$ be independent vectors and matrices with standard Gaussian entries. Define 

\begin{align*}
V_1(Z) & = \min_{w_1 \in \Omega_1} \max_{w_2 \in \Omega_2} w_1^{\top}Zw_2  + U(w_1,w_2)\enspace, \\
V_2(g,h) & = \min_{w_1 \in \Omega_1} \max_{w_2 \in \Omega_2} \|w_2 \|g^{\top} w_1 + \| w_1\| h^{\top} w_2 + U(w_1,w_2) \enspace.
\end{align*}

Then 
\begin{enumerate}
\item For all $t \in \Reals$, 
\[ \mathbb{P} (V_1(Z) \leq t) \leq 2 \mathbb{P}(V_2(g,h) \leq t ) \enspace.\]
\item Suppose $\Omega_1$ and $\Omega_2$ are both convex,  and $U$ is convex-concave in $(w_1, w_2)$. Then, for all $t \in \Reals$,
\[ \mathbb{P}(V_1(Z) \geq t) \leq 2 \mathbb{P}(V_2(g,h) \geq t) \enspace.\]
\end{enumerate}

\end{theorem}
\subsection{Large $n,p$ Limit: New Uniform Convergence Results}
\label{sec:prob-analysis}

Let $g \in \Reals^n$ be such that $g_i \stackrel{i.i.d.}{\sim} \mathcal{N}(0,1)$. Recall the definitions of $\lambda_j, w_j$ from Assumption \ref{asmp:cov} and \eqref{eq:newwdef} respectively, and denote the empirical distribution of $\{(\lambda_j, \sqrt{p} w_j,g_j)\}_{i=1}^p$ by $\cQ_p$, that is, 
\begin{equation}\label{eq:eq:cQp}
\cQ_p  = \frac{1}{p} \sum_{j=1}^p \delta_{(\lambda_j, \sqrt{p} w_j, g_j)}.
\end{equation}
Simultaneously, let $\cQ_{\infty} = \cQ$ from Definition~\ref{defn:system-of-equations}, that is, $\cQ_{\infty}=\mu \otimes \mathcal{N}(0,1)$. Define the functions $ V_1^{(\infty,\infty)}(\cdot, \cdot, \cdot), V_2^{(\infty,\infty)}(\cdot, \cdot, \cdot) , V_3^{(\infty,\infty)}(\cdot, \cdot, \cdot) : \Reals^3 \rightarrow \Reals$ as follows 
{\small
\begin{align}
		\label{eq:equation-infinite}
		& V_1^{(\infty,\infty)}(c_1, c_2, s) := c_1 + \nonumber \\
		& \E_{(\Lambda, W, G) \sim \cQ_\infty} \left( \frac{ \Lambda^{-1/2} W \cdot \prox_{s}\left( \Lambda^{1/2}  \Pi_{W^\perp}(G) + \psi^{-1/2} [\partial_1  F_{\kappa}(c_1, c_2) - c_1 c_2^{-1} \partial_2  F_{\kappa}(c_1, c_2)] \Lambda^{1/2}  W  \right) }{\psi^{-1/2} c_2^{-1} \partial_2  F_\kappa(c_1, c_2)  }  \right) \nonumber\\
		& V_2^{(\infty,\infty)}(c_1, c_2, s):= c_1^2+c_2^2 - \nonumber\\
	 &  \E\limits_{(\Lambda, W, G) \sim \cQ_\infty} \left(  \frac{ \Lambda^{-1/2} \prox_{s}\left( \Lambda^{1/2} \Pi_{W^\perp}(G) + \psi^{-1/2} [\partial_1  F_{\kappa}(c_1, c_2) - c_1 c_2^{-1} \partial_2  F_{\kappa}(c_1, c_2)] \Lambda^{1/2}  W  \right)}{ \psi^{-1/2} c_2^{-1} \partial_2  F_\kappa(c_1, c_2)  } \right)^2 \\
		& V_3^{(\infty,\infty)}(c_1, c_2, s) := \nonumber \\
	&1 -  \E\limits_{(\Lambda, W, G) \sim \cQ_\infty} \left| \frac{ \Lambda^{-1} \prox_{s}\left( \Lambda^{1/2}  G + \psi^{-1/2} [\partial_1  F_{\kappa}(c_1, c_2) - c_1 c_2^{-1} \partial_2  F_{\kappa}(c_1, c_2)] \Lambda^{1/2}  W  \right)}{ \psi^{-1/2} c_2^{-1} \partial_2  F_\kappa(c_1, c_2)  }   \right|, \nonumber
	\end{align}}
	where $F_{\kappa}(\cdot,\cdot)$ is given by \eqref{eq:YZ}. 
	
	Then from Proposition \ref{prop:uniqueness}, we immediately obtain the following. 
	\begin{lemma}\label{lem:uniquenesscrucial}
	Given any $(\psi,\kappa)$ such that $\psi > \psi^{\downarrow}(\kappa)$, denote $(c_1^\star, c_2^\star, s^\star) \in \mathbb{R} \times \mathbb{R}_{>0} \times \mathbb{R}_{>0}$ to be the unique solution to the system \eqref{eq:fix-points-l1}. Then for every $\epsilon >0$, there exists $\delta(\epsilon) > 0 $ small enough such that if a triplet $(c_1, c_2, s) \in \mathbb{R}\ \times \mathbb{R}_{>0} \times \mathbb{R}_{>0}$ satisfies
	\begin{align}
		\label{eq:delta-eps}
		|(c_2\vee 1)^{-1} V_1^{(\infty,\infty)}(c_1, c_2, s)| \leq \delta \nonumber\\
		|(c_2 \vee 1)^{-2}  V_2^{(\infty,\infty)}(c_1, c_2, s) | \leq \delta\\
		|(c_2\vee 1)^{-1} V_3^{(\infty,\infty)}(c_1, c_2, s) | \leq \delta ,\nonumber
	\end{align}
	then, $(c_1, c_2, s)$ must be $\epsilon$-close to $(c_1^{\star},c_2^\star,s^{\star})$, 
	\begin{align}
		(c_1, c_2, s) \in \cB\left((c_1^\star, c_2^\star, s^\star), \epsilon\right) \enspace.
	\end{align} 
	\end{lemma}
	
	We next turn to define different empirical versions of \eqref{eq:equation-infinite}, which will be used later. To this end, recall that \eqref{eq:Fkhat}
	\begin{equation}\label{F_nhat}
	\hat{ F}_{\kappa}(c_1, c_2) :=  \left( \widehat \E_n[ (\kappa - c_1 YZ_1 - c_2 Z_2)_{+}^2] \right)^{1/2},
	\end{equation}
	and define 
	{\small
	\begin{align}
		\label{eq:equation-finite-n-p}
		&V_1^{(n,p)}(c_1, c_2, s):= c_1 +\nonumber \\
		&  \E_{(\Lambda, W, G) \sim \cQ_p} \left( \frac{\irL W \cdot \prox_{s}\left( \Lambda^{1/2}  \Pi_{W^\perp}(G) + \psi^{-1/2} [\partial_1 \widehat F_{\kappa}(c_1, c_2) - c_1 c_2^{-1} \partial_2 \widehat F_{\kappa}(c_1, c_2)] \Lambda^{1/2}  W  \right) }{\psi^{-1/2} c_2^{-1} \partial_2 \widehat F_{\kappa}(c_1, c_2)  }  \right) \nonumber\\
		&V_2^{(n,p)}(c_1, c_2, s):=  c_1^2+c_2^2 -\nonumber \\
		& \E\limits_{(\Lambda, W, G) \sim \cQ_p} \left( \frac{\irL \prox_{s}\left( \Lambda^{1/2} \Pi_{W^\perp}(G) + \psi^{-1/2} [\partial_1 \widehat F_{\kappa}(c_1, c_2) - c_1 c_2^{-1} \partial_2 \widehat F_{\kappa}(c_1, c_2)] \Lambda^{1/2}  W  \right)}{ \psi^{-1/2} c_2^{-1} \partial_2 \widehat F_{\kappa}(c_1, c_2)  } \right)^2 \\
		&V_3^{(n,p)}(c_1, c_2, s):=  1 -  \nonumber \\
		&\E\limits_{(\Lambda, W, G) \sim \cQ_p} \left| \frac{ \iL  \prox_{s}\left( \Lambda^{1/2}  G + \psi^{-1/2} [\partial_1 \widehat F_{\kappa}(c_1, c_2) - c_1 c_2^{-1} \partial_2 \widehat F_{\kappa}(c_1, c_2)] \Lambda^{1/2}  W  \right)}{ \psi^{-1/2} c_2^{-1} \partial_2 \widehat F_{\kappa}(c_1, c_2)  }   \right| \nonumber
	\end{align}}
	
	Finally, define the functions $ V_1^{(\infty,p)}(\cdot, \cdot, \cdot), V_2^{(\infty,p)}(\cdot, \cdot, \cdot) , V_3^{(\infty,p)}(\cdot, \cdot, \cdot) : \Reals^3 \rightarrow \Reals$ as follows 
	
	{\small
	\begin{align}
		\label{eq:equation-infinite-p}
		& V_1^{(\infty,p)}(c_1, c_2, s) :=  c_1 +\nonumber \\
		& \E_{(\Lambda, W, G) \sim \cQ_p} \left( \frac{\irL W \cdot \prox_{s}\left( \Lambda^{1/2}  \Pi_{W^\perp}(G) + \psi^{-1/2} [\partial_1  F_{\kappa}(c_1, c_2) - c_1 c_2^{-1} \partial_2  F_{\kappa}(c_1, c_2)] \Lambda^{1/2}  W  \right) }{\psi^{-1/2} c_2^{-1} \partial_2  F_\kappa(c_1, c_2)  }  \right) \nonumber\\
		& V_2^{(\infty,p)}(c_1, c_2, s):=  c_1^2+c_2^2 -\nonumber \\
	 & \E\limits_{(\Lambda, W, G) \sim \cQ_p} \left( \frac{\irL \prox_{s}\left(  \Pi_{W^\perp}(G) + \psi^{-1/2} [\partial_1  F_{\kappa}(c_1, c_2) - c_1 c_2^{-1} \partial_2  F_{\kappa}(c_1, c_2)] \Lambda^{1/2}  W  \right)}{ \psi^{-1/2} c_2^{-1} \partial_2  F_\kappa(c_1, c_2)  } \right)^2 \\
		& V_3^{(\infty,p)}(c_1, c_2, s) :=1 -   \nonumber \\
	&\E\limits_{(\Lambda, W, G) \sim \cQ_p} \left| \frac{\iL \prox_{s}\left( \Lambda^{1/2}  G + \psi^{-1/2} [\partial_1  F_{\kappa}(c_1, c_2) - c_1 c_2^{-1} \partial_2  F_{\kappa}(c_1, c_2)] \Lambda^{1/2}  W  \right)}{ \psi^{-1/2} c_2^{-1} \partial_2  F_\kappa(c_1, c_2)  }   \right|, \nonumber
	\end{align}}
Observe $V_i^{(\infty,p)}(\cdot, \cdot, \cdot)$ and $V_i^{(n,p)}(\cdot, \cdot, \cdot)$ only differs in the following sense: $\widehat F_{\kappa}$ is used in place of $F_{\kappa}$.

With the above preparation, we are now in position to establish \eqref{eq:finitetoinfiniteopt}. Recall the finite $n,p$ optimization problem
	\begin{align}\label{eq:remind-finite}
	\hat \xi^{(n,p)}_{\psi, \kappa}(\lambda, w, g)  &:= \min_{\|\theta\|_1 \leq \sqrt{p}}  \psi^{-1/2} \widehat F_{\kappa} \left( \langle w,  \Lambda^{1/2}\theta \rangle,\| \Pi_{w^\perp}(\Lambda^{1/2}\theta) \|_2  \right) + \frac{1}{\sqrt{p}} \left\langle \Pi_{w^\perp}(g), \Lambda^{1/2} \theta \right\rangle ,
	\end{align}
	and the corresponding infinite-dimensional optimization problem given by
	\begin{align}\label{eq:remind-infinite}
	& \tilde \xi^{(\infty, \infty)}_{\psi, \kappa}( \Lambda, W, G) := \nonumber\\
	& \min_{\| h \|_{L_1(\cQ_{\infty})} \leq 1} \psi^{-1/2} F_{\kappa} \left( \langle w,  \Lambda^{1/2} h \rangle_{L_2(\cQ_{\infty})},\| \Pi_{w^\perp}(\Lambda^{1/2}h) \|_{L_2(\cQ_{\infty})}  \right) +  \left\langle \Pi_{w^\perp}(G), \Lambda^{1/2} h \right\rangle_{L_2(\cQ_{\infty})} \enspace.
	\end{align}

\begin{proposition}[Large $n,p$ limit]
	\label{prop:large-n-p-limit}
	Under the assumptions of Theorem \ref{thm:l-1-margin}, almost surely, 
	\begin{align}
		\lim_{n\rightarrow \infty, p(n)/n = \psi}\hat \xi^{(n,p)}_{\psi, \kappa}(\lambda, w, g) = \tilde \xi^{(\infty, \infty)}_{\psi, \kappa}( \Lambda, W, G),
	\end{align}
	where $(\Lambda,W,G) \sim \cQ_{\infty}$. 
\end{proposition}


\begin{proof}[Proof of Proposition \ref{prop:large-n-p-limit}]
To begin with, recall the KKT conditions \eqref{eq:kkt-explain} and its consequences \eqref{eq:kkt-appendix}-\eqref{eq:fpappendix}, together these establish the following fixed point equations 
\begin{equation*}
 V_1^{(\infty,\infty)}(c_1, c_2, s) =0,  V_2^{(\infty,\infty)}(c_1, c_2, s) =0 ,  V_3^{(\infty,\infty)}(c_1, c_2, s)=0.
\end{equation*}
We postpone the derivation of the KKT conditions later so as to not interrupt the flow.

Note that the objective function in \eqref{eq:remind-finite} is not convex in $\theta$ (due to $\widehat F$). Nonetheless, for any $\theta$ that minimizes the objective, the KKT conditions still hold as first-order necessary conditions. Thus, by arguments similar to that in the proof of Proposition \ref{prop:uniqueness}, with $\theta/\sqrt{p}, \cQ_p, \widehat F_k$ replacing $h,\cQ_{\infty}, F_k$, we obtain the finite sample versions 

\begin{equation} \label{eq:empiricalform}
 V_1^{(n,p)}(c_1, c_2, s) =0,  V_2^{(n,p)}(c_1, c_2, s) =0 ,  V_3^{(n,p)}(c_1, c_2, s)=0.
\end{equation}

	We claim that almost surely, the following uniform convergence result holds, in the region $c_1 \in [0, M], c_2>0, s>0$ 
	\begin{align}
		\label{eqn:uniform-deviation-summary}
		\lim_{n\rightarrow \infty, p(n)/n = \psi} \sup_{c_1 \in [0, M], c_2>0, s>0} (c_2\vee 1)^{-1} |V_1^{(n,p)}(c_1, c_2, s) - V_1^{(\infty,\infty)}(c_1, c_2, s)| = 0  \nonumber\\
		\lim_{n\rightarrow \infty, p(n)/n = \psi} \sup_{c_1 \in [0, M], c_2>0, s>0} (c_2\vee 1)^{-2} |V_2^{(n,p)}(c_1, c_2, s) - V_2^{(\infty,\infty)}(c_1, c_2, s)| = 0  \\
		\lim_{n\rightarrow \infty, p(n)/n = \psi} \sup_{c_1 \in [0, M], c_2>0, s>0} (c_2\vee 1)^{-1} |V_3^{(n,p)}(c_1, c_2, s) - V_3^{(\infty,\infty)}(c_1, c_2, s)| = 0 \nonumber
	\end{align}
	In the following, we will prove the above claims.

\noindent \textbf{The first claim in \eqref{eqn:uniform-deviation-summary}.}	
	By the triangle inequality, 
	\begin{align}
		\label{eq:V1}
&	|V_1^{(n,p)}(c_1, c_2, s) - V_1^{(\infty,\infty)}(c_1, c_2, s)| \\ & \leq  |V_1^{(n,p)}(c_1, c_2, s) - V_1^{(\infty, p)}(c_1, c_2, s)| + |V_1^{(\infty,p)}(c_1, c_2, s) - V_1^{(\infty, \infty)}(c_1, c_2, s)| \enspace.
	\end{align}
	We start with providing a uniform deviation bound in the region $c_1 \in [0, M], c_2>0, s>0$ for 
 \begin{align}
		\label{eq:V1-part1}
		 (c_2\vee 1)^{-1} |V_1^{(n,p)}(c_1, c_2, s) - V_1^{(\infty, p)}(c_1, c_2, s)| \enspace.
	\end{align}
	Note here that $c_2, s$ lie in unbounded regions---such a scenario does not arise in the study of the max-$L_2$-margin, for instance.
	Define 
	\begin{align}
		\hat C^{\uparrow}&:=\psi^{-1/2} [\partial_1  \widehat F_{\kappa}(c_1, c_2) - c_1 c_2^{-1} \partial_2  \widehat F_{\kappa}(c_1, c_2)] 	\\
		\hat C^{\downarrow}&:= \psi^{-1/2} c_2^{-1} \partial_2  \widehat F_{\kappa}(c_1, c_2) 
	\end{align}
	and similarly $C^{\uparrow}, C^{\downarrow}$ by replacing $\widehat F_{\kappa}$ by $F_\kappa$. By the contraction property of the soft-thresholding operator,
	{\small \begin{align}
	\nonumber
	&	Eqn.~\eqref{eq:V1-part1}\leq \\ & (c_2\vee 1)^{-1}\left\{ \frac{  \|\Lambda^{-1/2} G\|_{L_2(\cQ_p)} \|\Lambda^{1/2} W\|_{L_2(\cQ_p)}+ \| W\|^2_{L_2(\cQ_p)} |\hat C^{\uparrow}|  }{|\hat C^{\downarrow}  C^{\downarrow}|} |\hat C^{\downarrow} -  C^{\downarrow}|  + \frac{\| W\|^2_{L_2(\cQ_p)} }{ |C^{\downarrow}| } |\hat C^{\uparrow} -  C^{\uparrow} | \right\} \enspace.
	\end{align}}
	
As in Lemma \ref{lem:self-normalized-uniform-devation}, divide the range of $c_2$ into the regions $(0, M]$ and $(M, \infty)$ respectively.
	For $c_2 \in (0, M]$, multiply both the denominator and nominator by $c_2^2$ to obtain
	\begin{align}
		Eqn.~\eqref{eq:V1-part1}&\leq  \frac{  c_2 L +  |c_2 \hat C^{\uparrow}| L }{ |c_2 \hat C^{\downarrow}|  |c_2 C^{\downarrow}|} |(c_2 \hat C^{\downarrow}) -  (c_2 C^{\downarrow})|  + \frac{L}{|c_2 C^{\downarrow}|} | (c_2 \hat C^{\uparrow}) -  (c_2 C^{\uparrow}) |
	\end{align} 
	where $L^{1/2}$ is a uniform upper bound on $\|\Lambda^{-1/2} G\|_{L_2(\cQ_p)}$, $\|\Lambda^{1/2} W\|_{L_2(\cQ_p)}$, $\| W\|_{L_2(\cQ_p)}$ for all $p$.
	By Lemma~\ref{lem:self-normalized-uniform-devation}, we know that w.p. at least $1 - n^{-2}$ for all $|c_1| \leq M, 0< c_2 \leq M, s>0$
	\begin{align*}
		&| (c_2 \hat C^{\downarrow}) -  (c_2 C^{\downarrow}) | = \psi^{-1/2} |\partial_2 \widehat F_{\kappa}(c_1, c_2) - \partial_2 F_{\kappa}(c_1, c_2) | \precsim \frac{\log n}{\sqrt{n}}\\
		&|(c_2 \hat C^{\uparrow}) -  (c_2 C^{\uparrow}) | \leq  \psi^{-1/2} c_2 \cdot |\partial_2 \widehat F_{\kappa}(c_1, c_2) - \partial_2 F_{\kappa}(c_1, c_2) | \nonumber \\
		&+\psi^{-1/2} |c_1|\cdot |\partial_1 \widehat F_{\kappa}(c_1, c_2) - \partial_1 F_{\kappa}(c_1, c_2)| \precsim \frac{\log n}{\sqrt{n}} 
	\end{align*}
	which ensures that w.p. at least $1 - n^{-2}$ for all $|c_1| \leq M, 0< c_2 \leq M, s>0$,
	\begin{align}
		Eqn.~\eqref{eq:V1-part1} &\leq L'  \cdot \frac{\log n}{\sqrt{n}} \enspace,
	\end{align}
	and the upper bound is uniform for all $p$.
	
	For the second region, $c_2 \in (M, 
	\infty)$, we use the following technique as in Lemma~\ref{lem:self-normalized-uniform-devation}
	\begin{align}
		Eqn.~\eqref{eq:V1-part1}&\leq  (c_2\vee 1)^{-1} \left( c_2 \frac{   L +  | \hat C^{\uparrow}| L }{ |(c_2 \hat C^{\downarrow}|  |c_2 C^{\downarrow}|} |(c_2 \hat C^{\downarrow}) -  (c_2 C^{\downarrow})|  + c_2 \frac{L}{|c_2 C^{\downarrow}| } |  \hat C^{\uparrow} -   C^{\uparrow} | \right)\\
		&\leq \frac{   L +  | \hat C^{\uparrow}| L }{ |(c_2 \hat C^{\downarrow}|  |c_2 C^{\downarrow}|} |(c_2 \hat C^{\downarrow}) -  (c_2 C^{\downarrow})|   + \frac{L}{|c_2 C^{\downarrow}| } |  \hat C^{\uparrow} -   C^{\uparrow} | 
	\end{align}
	By Lemma~\ref{lem:self-normalized-uniform-devation}, we know that w.p. at least $1-n^{-2}$, uniformly for the region $|c_1| \leq M, c_2 > M, s>0$,
	\begin{align*}
		&| (c_2 \hat C^{\downarrow}) -  (c_2 C^{\downarrow}) | = \psi^{-1/2} |\partial_2 \widehat F_{\kappa}(c_1, c_2) - \partial_2 F_{\kappa}(c_1, c_2) | \precsim \frac{\log n}{\sqrt{n}} \\
		& | \hat C^{\uparrow} -   C^{\uparrow} | \leq  \psi^{-1/2} |\partial_2 \widehat F_{\kappa}(c_1, c_2) - \partial_2 F_{\kappa}(c_1, c_2) | \nonumber \\
		& +  \psi^{-1/2} |c_1 c_2^{-1}| \cdot |\partial_1 \widehat F_{\kappa}(c_1, c_2) - \partial_1 F_{\kappa}(c_1, c_2)| \precsim \frac{\log n}{\sqrt{n}} 
	\end{align*}
	 since $c_1 c_2^{-1}$ is bounded by 1.
	
	Putting things together, we have established that w.p. at least $1 - 2n^{-2}$, 
	\begin{align}
		\sup_{|c_1|\leq M, c_2>0, s>0}~   (c_2\vee 1)^{-1}   |V_1^{(n,p)}(c_1, c_2, s) - V_1^{(\infty, p)}(c_1, c_2, s)| \leq \frac{\log n}{\sqrt{n}} \enspace.
	\end{align}
	We remark that the above uniform deviation bound over unbounded region is proved due to a key self-normalization property of the function $\partial_i F_{\kappa}(c_1, c_2), i=1,2$, as derived in Lemma~\ref{lem:self-normalized-uniform-devation}.

	We now proceed to bound the second term in \eqref{eq:V1}
	\begin{align}
		&(c_2\vee 1)^{-1}  |V_1^{(\infty,p)}(c_1, c_2, s) - V_1^{(\infty, \infty)}(c_1, c_2, s)| \\
		&= \left| \left(  \E_{(\Lambda, W, G) \sim \cQ_p} -  \E_{(\Lambda, W, G) \sim \cQ_\infty} \right)  (c_2\vee 1)^{-1}  f_{c_1, c_2, s}(\Lambda, W, G)\right|,
	\end{align}
	where
	\begin{align}
		&f_{c_1, c_2, s}(\Lambda, W, G)  \\
		&: = \left( \frac{\irL W \cdot \prox_{s}\left( \Lambda^{1/2}  \Pi_{W^\perp}(G) + \psi^{-1/2} [\partial_1  F_{\kappa}(c_1, c_2) - c_1 c_2^{-1} \partial_2  F_{\kappa}(c_1, c_2)] \Lambda^{1/2}  W  \right) }{\psi^{-1/2} c_2^{-1} \partial_2  F_\kappa(c_1, c_2)  }   \right).  \nonumber
	\end{align}
	Since $\cQ_p \stackrel{W_2}{\implies} \cQ_{\infty}$, by Theorem 2.7 and Proposition 2.4 in \cite{ambrosio2013user}, we know that (1)
	for any  function $g$ that grows at most quadratically,
	\begin{align}
		\sup_{\Lambda, W, G} ~\frac{|g(\Lambda, W, G)|}{1+ \| (\Lambda, W, G) \|_2^2} < \infty \enspace,\\
		 \lim_{p \rightarrow \infty}~  \left| \left(  \E_{(\Lambda, W, G) \sim \cQ_p} -  \E_{(\Lambda, W, G) \sim \cQ_\infty} \right)  g(\Lambda, W, G)\right| = 0 \enspace,
	\end{align}
	and that (2) $\{ \cQ_p, p \in \mathbb{N} \}$ is 2-uniformly integrable in the following sense: for any $\epsilon>0$, there exists $R_\epsilon$ such that uniformly for $p$,
	\begin{align}
		\label{eqn:2-unif-int}
		\sup_{p \in \mathbb{N}} \int_{\mathbb{R}^3 \setminus \cB_{R_\epsilon}} \| (\Lambda, W, G) \|^2 d \cQ_p < \epsilon \enspace.
	\end{align}
	Here $\mathbb{R}^3 \setminus \cB_{R_\epsilon}$ denotes the complement of a ball of radius $R_\epsilon$ centered at zero. Note that (1) has not yet established the uniform convergence that we desire. We now prove it using the structural form of $f_{c_1, c_2, s}(\Lambda, W, G)$.

	We first verify that  $g_{c_1, c_2, s}:= (c_2 \vee 1)^{-1} f_{c_1, c_2, s}$ satisfies the quadratic growth condition uniformly for all $|c_1| \leq M, c_2>0, s>0$. Observe that
	\begin{align*}
		  |f_{c_1, c_2, s}(\Lambda, W, G)| \leq  \frac{| W \Pi_{W^\perp}(G)| + |C^{\uparrow}| | W^2| }{|C^{\downarrow}|} \leq \frac{ G^2 +  W^2 + |C^{\uparrow}| \cdot  W^2 }{| C^{\downarrow}|} \enspace.
	\end{align*}
	Further,  for all $|c_1|\leq M, 0 \leq c_2 \leq M, s \geq 0$, uniformly for $\Lambda, W, G$ (recall that $\Lambda, W$ has bounded domain)
	\begin{align}
		 \frac{(c_2\vee 1)^{-1}  |f_{c_1, c_2, s}(\Lambda, W, G)| }{ 1+ \Lambda^2 + W^2 + G^2} \leq \frac{c_2 (G^2 +  W^2) + |c_2 C^{\uparrow}|  W^2}{ |c_2 C^{\downarrow}| ( 1+ \Lambda^2 + W^2 + G^2)} <\infty \enspace,
	\end{align} 
	since $|c_2 C^{\uparrow}|$ is  bounded above and $|c_2 C^{\downarrow}| = \psi^{-1/2} |\partial_2 F_{\kappa}|$ is bounded below.
	For the other part where $|c_1|\leq M, c_2 > M, s \geq 0$, since $|c_1 c_2^{-1} |$ is bounded and, thus, $|C^{\uparrow}|$ is bounded, hence
	\begin{align}
		 \frac{(c_2\vee 1)^{-1}  |f_{c_1, c_2, s}(\Lambda, W, G)| }{ 1+ \Lambda^2 + W^2 + G^2} \leq \frac{(G^2 +  W^2) + | C^{\uparrow}|  W^2}{ |c_2 C^{\downarrow}| ( 1+ \Lambda^2 + W^2 + G^2)} <\infty \enspace.
	\end{align}
	Therefore uniformly over $|c_1| \leq M, c_2>0, s>0$, with a universal constant $K$
	\begin{align}
		\label{eqn:quad-growth}
		|g_{c_1, c_2, s}(\Lambda, W, G)| =  (c_2\vee 1)^{-1}  |f_{c_1, c_2, s}(\Lambda, W, G)| \leq K \cdot \| (\Lambda, W, G) \|^2 \enspace.
	\end{align}

	Note that $g_{c_1, c_2, s}(\Lambda, W, G)$ depends on $c_1, c_2, s$. We now prove the convergence of $\E_{\cQ_p} [g_{c_1, c_2, s}]$ to $\E_{\cQ_\infty} [g_{c_1, c_2, s}]$ uniformly over $c_1, c_2, s$. Recall that $\cQ_p$ is $2$-uniformly integrable, hence for any fixed $\epsilon>0$, there exists $R_\epsilon$ such that \eqref{eqn:2-unif-int} holds true. Therefore
	\begin{align}
		\label{eq:uniform-convergence-in-p}
		&|\int g_{c_1, c_2, s} d\cQ_p - \int g_{c_1, c_2, s} d\cQ_\infty| \nonumber\\
		& \leq |\int_{\cB_{R_\epsilon}} g_{c_1, c_2, s} (d\cQ_p - d \cQ_\infty)| + |\int_{\mathbb{R}^3 \setminus \cB_{R_\epsilon}}  g_{c_1, c_2, s} (d\cQ_p - d \cQ_\infty)| \nonumber\\
		&\leq  |\int_{\cB_{R_\epsilon}}  g_{c_1, c_2, s} (d\cQ_p - d \cQ_\infty)| + |\int_{\mathbb{R}^3 \setminus \cB_{R_\epsilon}} g_{c_1, c_2, s} d\cQ_p| + |\int_{\mathbb{R}^3 \setminus \cB_{R_\epsilon}}  g_{c_1, c_2, s} d \cQ_\infty|\enspace, \nonumber \\
		&\leq |\int_{\cB_{R_\epsilon}}  g_{c_1, c_2, s} (d\cQ_p - d \cQ_\infty)| + 2K \epsilon 
	\end{align}
	where the last step uses the quadratic growth condition of $g_{c_1, c_2, s}$ in \eqref{eqn:quad-growth} uniformly over $|c_1| \leq M, c_2>0, s>0$, and 2-uniform integrability \eqref{eqn:2-unif-int}, as
	\begin{align}
		|\int_{\mathbb{R}^3 \setminus \cB_{R_\epsilon}} g_{c_1, c_2, s} d\cQ_p| \leq K \int_{\mathbb{R}^3 \setminus \cB_{R_\epsilon}} \| (\Lambda, W, G) \|^2 d\cQ_p  \leq K\epsilon \enspace.
	\end{align}
	Inside a bounded region $B_{R_\epsilon}$, it is easy to see that $g_{c_1, c_2, s}(\Lambda, W, G)$ is Lipschitz in $(\Lambda, W, G)$ with a uniform Lipschitz constant $L_{R_\epsilon}$ regardless of the choice of $|c_1| \leq M, c_2>0, s>0$. Therefore we have
	\begin{align}
		|\int_{\cB_{R_\epsilon}}  g_{c_1, c_2, s} (d\cQ_p - d \cQ_\infty)| \leq L_{R_\epsilon} W_1(\cQ_p, \cQ_\infty) \leq L_{R_\epsilon} W_2(\cQ_p, \cQ_\infty) \enspace.
	\end{align}
	Now we have proved that for 
	\begin{align}
		\sup_{|c_1|\leq M, c_2 >0, s>0 }~  |\int g_{c_1, c_2, s} d\cQ_p - \int g_{c_1, c_2, s} d\cQ_\infty| &\leq L_{R_\epsilon} W_2(\cQ_p, \cQ_\infty) + 2K \epsilon \enspace, \\
		\lim_{p \rightarrow \infty} \sup_{|c_1|\leq M, c_2 >0, s>0 }~  |\int g_{c_1, c_2, s} d\cQ_p - \int g_{c_1, c_2, s} d\cQ_\infty| & \leq 2K \epsilon \enspace.
	\end{align}
	By the fact that $\epsilon$ can take an arbitrarily small value, we have proved
	\begin{align}
		\lim_{n \rightarrow \infty, p(n)/n = \psi} \sup_{|c_1|\leq M, c_2 >0, s>0 }~(c_2\vee 1)^{-1}  |V_1^{(\infty,p)}(c_1, c_2, s) - V_1^{(\infty, \infty)}(c_1, c_2, s)| = 0 \enspace,
	\end{align}
	which handles the second term in \eqref{eq:V1}.
	
	We combine with the analysis of \eqref{eq:V1-part1} and by Borel-Cantelli Lemma obtain that, almost surely
	\begin{align}
		\lim_{n \rightarrow \infty, p(n)/n = \psi} \sup_{|c_1|\leq M, c_2 >0, s>0 }~(c_2\vee 1)^{-1}  |V_1^{(n,p)}(c_1, c_2, s) - V_1^{(\infty, p)}(c_1, c_2, s)| = 0 \enspace.
	\end{align}
	Thus we have established that uniformly over $|c_1|\leq M, c_2 >0, s>0$, 
		\begin{align}
		\lim_{n \rightarrow \infty, p(n)/n = \psi}\sup_{|c_1|\leq M, c_2 >0, s>0 }~  (c_2\vee 1)^{-1} |V_1^{(n,p)}(c_1, c_2, s) - V_1^{(\infty, p)}(c_1, c_2, s)| = 0, \quad a.s.
	\end{align}

	\noindent \textbf{The second claim in \eqref{eqn:uniform-deviation-summary}.} 
This step follows similarly to the aforementioned analysis, here we only highlight the differences. Once again,
	\begin{align}
		\label{eq:V2}
	& |V_2^{(n,p)}(c_1, c_2, s) - V_2^{(\infty,\infty)}(c_1, c_2, s)| \nonumber \\
	& \leq  |V_2^{(n,p)}(c_1, c_2, s) - V_2^{(\infty, p)}(c_1, c_2, s)| + |V_2^{(\infty,p)}(c_1, c_2, s) - V_2^{(\infty, \infty)}(c_1, c_2, s)| \enspace.
	\end{align}
	Now it suffices to provide a uniform deviation bound for $c_1 \in [0, M], c_2>0, s>0$ 
	\begin{align}
		\label{eq:V2-part1}
		& (c_2\vee 1)^{-2} |V_2^{(n,p)}(c_1, c_2, s) - V_2^{(\infty, p)}(c_1, c_2, s)| \\
		\nonumber &\leq (c_2\vee 1)^{-2} \E_{(\Lambda, W, G)\sim \cQ_p} \left\{ \left(\frac{| \Pi_{W^\perp}(G)| + |\hat C^{\uparrow}| |  W|}{|\hat C^{\downarrow}|} + \frac{| \Pi_{W^\perp}(G)| + |C^{\uparrow}| |  W|}{ |C^{\downarrow}|}  \right)\right.
		\\
		&\left. \quad \times \left( \frac{| \Pi_{W^\perp}(G)| + |\hat C^{\uparrow}| |  W|}{|\hat C^{\downarrow} C^{\downarrow}|} | \hat C^{\downarrow} - C^{\downarrow} | + \frac{|  W|}{|C^{\downarrow}|}| \hat C^{\uparrow} - C^{\uparrow} |\right) \right\}.
	\end{align}
	Again we divide the range of $c_2$ into two parts, $(0, M]$ and $(M, \infty)$.
	For the first part, uniformly over $(c_1, c_2) \in [-M, M] \times (0, M]$, Lemma~\ref{lem:self-normalized-uniform-devation} shows that
	\begin{align}
		|c_2 \hat C^{\downarrow} - c_2 C^{\downarrow} |, |c_2 \hat C^{\uparrow} - c_2 C^{\uparrow} | \precsim \frac{\log n}{\sqrt{n}} \enspace.
	\end{align}
	For the second part, uniformly over $(c_1, c_2) \in [-M, M] \times (M, \infty)$, Lemma~\ref{lem:self-normalized-uniform-devation} shows that
	\begin{align}
		|c_2 \hat C^{\downarrow} - c_2 C^{\downarrow} |, |\hat C^{\uparrow} -  C^{\uparrow} | \precsim \frac{\log n}{\sqrt{n}} \enspace.
	\end{align}
	In either case, one can show that w.p. at least $1 - n^{-2}$, 
	\begin{align}
		\sup_{|c_1|\leq M, c_2>0, s>0 }~ (c_2\vee 1)^{-2} |V_2^{(n,p)}(c_1, c_2, s) - V_2^{(\infty, p)}(c_1, c_2, s)|  \leq L'\cdot \frac{\log n}{\sqrt{n}} \enspace.
	\end{align}
	
	For the term, 
	\begin{align}
		 &(c_2\vee 1)^{-2} |V_2^{(\infty,p)}(c_1, c_2, s) - V_2^{(\infty, \infty)}(c_1, c_2, s)| \\
		 &= (c_2\vee 1)^{-2} \left|  \left( \E_{(\Lambda, W, G) \sim \cQ_p} -  \E_{(\Lambda, W, G) \sim \cQ_\infty} \right)   \tilde f_{c_1, c_2, s}(\Lambda, W, G) \right|,
	\end{align}
	with 
	\begin{align}
	&	\tilde f_{c_1, c_2, s}(\Lambda, W, G) : = \nonumber\\
	& \left( \frac{\irL \prox_{s}\left( \Lambda^{1/2}  \Pi_{W^\perp}(G) + \psi^{-1/2} [\partial_1  F_{\kappa}(c_1, c_2) - c_1 c_2^{-1} \partial_2  F_{\kappa}(c_1, c_2)] \Lambda^{1/2}  W  \right) }{\psi^{-1/2} c_2^{-1}, \partial_2  F_\kappa(c_1, c_2)  }  \right)^2 
	\end{align}
	one can verify that uniformly over $|c_1|\leq M, c_2>0, s>0$ and $\Lambda, W, G$
	\begin{align}
		 \frac{(c_2\vee 1)^{-2}  |\tilde f_{c_1, c_2, s}(\Lambda, W, G)| }{ 1+ \Lambda^2 + W^2 + G^2}  <\infty \enspace.
	\end{align}
	The uniform convergence can be established repeating the argument in \eqref{eq:uniform-convergence-in-p}.
	Therefore,
	\begin{align}
		\lim_{n \rightarrow \infty, p(n)/n = \psi} \sup_{|c_1|\leq M, c_2 >0, s>0 }~(c_2\vee 1)^{-2}  |V_2^{(\infty,p)}(c_1, c_2, s) - V_2^{(\infty, \infty)}(c_1, c_2, s)| = 0, \\
		\lim_{n \rightarrow \infty, p(n)/n = \psi} \sup_{|c_1|\leq M, c_2 >0, s>0 }~(c_2\vee 1)^{-2}  |V_2^{(n,p)}(c_1, c_2, s) - V_2^{(\infty, p)}(c_1, c_2, s)| = 0~~ \text{a.s.} \enspace.
	\end{align}

	\noindent \textbf{The third claim in \eqref{eqn:uniform-deviation-summary}.} The proof of the following uniform convergence for the term involving $V_3$ follows the exact same steps as for $V_1$ and, is therefore, omitted.

	We next establish that for any solution $\hat c_1, \hat c_2, \hat s$ that solves the empirical fixed point equation,
	\begin{align*}
		V_i^{(n,p)}(\hat c_1, \hat c_2, \hat s) = 0
	\end{align*}
	one must have that
	\begin{align}
		\label{eqn:limit-c_1-c_2-s}
		\lim_{n \rightarrow \infty, p(n)/n = \psi} \hat c_1 = c_1^\star, \quad \lim_{n \rightarrow \infty, p(n)/n = \psi} \hat c_2 = c_2^\star, \quad		\lim_{n \rightarrow \infty, p(n)/n = \psi} \hat s = s^\star
	\end{align}
	where $(c_1^\star, c_2^\star, s^\star)$ is the unique solution for the fixed point equation
	\begin{align*}
		V_i^{(\infty,\infty)}( c_1^\star, c_2^\star,  s^\star) = 0 \enspace.
	\end{align*}
	This follows by standard arguments on combining \eqref{eqn:uniform-deviation-summary} and Lemma \ref{lem:uniquenesscrucial}.
	For any $\epsilon>0$, there exist $\delta > 0$ small enough, that satisfies Eqn.~\ref{eq:delta-eps}. By the uniform convergence \eqref{eqn:uniform-deviation-summary}, for that particular $\delta$, there exist $n,p$ large enough, such that for $(\hat c_1, \hat c_2, \hat s)$
	\begin{align}
		(1\vee \hat c_2)^{-1} |V_1^{(n,p)}(\hat c_1, \hat c_2, \hat s) - V_1^{(\infty, \infty)}(\hat c_1, \hat c_2, \hat s)| \leq \delta \enspace.
	\end{align}
	Recall that $V_1^{(n, p)}(\hat c_1, \hat c_2, \hat s) = 0$, which implies
	\begin{align}
		(1\vee \hat c_2)^{-1} | V_1^{(\infty, \infty)}(\hat c_1, \hat c_2, \hat s)| \leq \delta \enspace,
	\end{align}
	therefore we know that for all $n,p$ large enough,
	\begin{align}
		(\hat c_1, \hat c_2, \hat s) \in \cB((c_1^\star, c_2^\star, s^\star), \epsilon) \enspace.
	\end{align}
	Note this holds for arbitrary $\epsilon$. Therefore, we have proved Eqn.~\eqref{eqn:limit-c_1-c_2-s}.
	
	We remark that this convergence result implies the following: any optimizer $\hat \theta$ of the finite $n,p$ optimization problem $\hat \xi^{(n,p)}_{\psi, \kappa}(\lambda, w, g)$ must satisfy the necessary condition
	\begin{align}
	\| \hat \theta \|^2 \asymp	\| \Lambda^{1/2} \hat \theta \|_2^2 = \langle w, \Lambda^{1/2} \hat \theta \rangle^2 + \| \Pi_{w^\perp }\hat \theta \|_2^2 = \hat c_1^2 + \hat c_2^2 \leq 2 (c_1^\star)^2 + 2 (c_2^\star)^2 < 4R^2\label{eq:rangebounds}
	\end{align}
	for some absolute constant $R > 0$, for sufficiently large $n$ and $p$.	This established property will be useful in the next paragraph.

	Given Eqn.~\eqref{eqn:limit-c_1-c_2-s}, one can verify by the KKT condition that the optimal value of finite $n,p$ optimization problem $\hat \xi^{(n,p)}_{\psi, \kappa}(\lambda, w, g)$ can be expressed in the form
	\begin{align}
		\hat T(\psi, \kappa; \hat c_1, \hat c_2, \hat s) := \psi^{-1/2}[\widehat F_{\kappa}(\hat c_1, \hat c_2) - \hat c_1 \partial_1 \widehat F_{\kappa}(\hat c_1, \hat c_2) - \hat c_2 \partial_2 \widehat F_{\kappa}(\hat c_1, \hat c_2) ] - \hat s
	\end{align}
	where $\hat c_1, \hat c_2, \hat s$ are solutions to the empirical fixed point equations $V_i^{(n,p)}(\hat c_1, \hat c_2, \hat s) = 0, i=1,2,3$ (that may not be unique for fixed $n,p$).
	Now recall that we have proved for sufficiently large $n,p$, $\hat c_1, \hat c_2$ lie in a neighborhood of fixed radius $R$ (does not grow with $n,p$) around $c_1^\star, c_2^\star$, say denoted by $\cB(c_1^\star, R), \cB(c_2^\star, R)$. It is easy to show that $\widehat F_{\kappa}$ satisfies the uniform convergence bound 
	\begin{align}
		\lim_{n \rightarrow \infty, p(n)/n = \psi} \sup_{c_1, c_2 \in \cB(c_1^\star, R), \cB(c_2^\star, R)}  |\widehat F_{\kappa}(c_1, c_2) - F_{\kappa}(c_1, c_2)| =0 \quad a.s.
	\end{align}
	By Lemma~\ref{lem:self-normalized-uniform-devation}, $\partial_1 \widehat F_{\kappa}$ and $\partial_2 \widehat F_{\kappa}$ all satisfy uniform convergence over $|c_1| \leq M, c_2 >0$. Therefore
	\begin{align}
		& \quad \lim_{n \rightarrow \infty, p(n)/n = \psi} ~\hat T(\psi, \kappa; \hat c_1, \hat c_2, \hat s)\\
		 &= \lim_{n \rightarrow \infty, p(n)/n = \psi} ~\psi^{-1/2}[ F_{\kappa}(\hat c_1, \hat c_2) - \hat c_1 \partial_1  F_{\kappa}(\hat c_1, \hat c_2) - \hat c_2 \partial_2  F_{\kappa}(\hat c_1, \hat c_2) ] - \hat s  \quad \text{(unif.~conv.)} \\
		& = \psi^{-1/2}[ F_{\kappa}(c_1^\star,  c_2^\star) - c_1^\star \partial_1  F_{\kappa}(c_1^\star,  c_2^\star) - c_2^\star \partial_2  F_{\kappa}( c_1^\star,  c_2^
		\star) ] - s^\star = T(\psi, \kappa)  \enspace.
	\end{align}
	Recall from Corollary \ref{cor:values} that the RHS equals $\tilde \xi^{(\infty, \infty)}_{\psi, \kappa}( \Lambda, W, G)$. Therefore, we have shown that the LHS limit exists and is unique. Therefore
	\begin{align*}
		\lim_{n\rightarrow \infty, p(n)/n = \psi}~ \hat \xi^{(n,p)}_{\psi, \kappa}(\lambda, w, g) & = \lim_{n\rightarrow \infty, p(n)/n = \psi} ~\hat T(\psi, \kappa; \hat c_1, \hat c_2, \hat s)\\
		& = T(\psi, \kappa)  = \tilde \xi^{(\infty, \infty)}_{\psi, \kappa}( \Lambda, W, G) \enspace.
	\end{align*}
\end{proof}

Below, we introduce a key lemma used in the uniform convergence proof in Proposition~\ref{prop:large-n-p-limit}. This lemma appears to be new to the literature.
\begin{lemma}[Self-normalization and uniform deviation, Lemma~\ref{lem:self-normalized-uniform-devation}]
	For $i = 1, 2$, we have with probability at least $1 - n^{-2}$, 
	\begin{align}\label{eq:derivconverge}
		\sup_{|c_1|\leq M, c_2>0} |\partial_i \widehat F_{\kappa}(c_1, c_2) - \partial_i F_{\kappa}(c_1, c_2)| \leq C \cdot \frac{\log n}{\sqrt{n}} \enspace,
	\end{align}
	where $C$ is a constant that does not depend on $n$.
\end{lemma}
\begin{proof}[Proof of Lemma~\ref{lem:self-normalized-uniform-devation}]
	The proof uses a key self-normalization property of the partial derivatives of $F_\kappa$, that ensure good concentration behavior even when $c_2$ is large. We remark that this structural property makes our uniform convergence result over unbounded region possible in Proposition~\ref{prop:large-n-p-limit}. Note that 
	\begin{align}
		\partial_1 \widehat F_{\kappa}(c_1, c_2) = -\frac{\widehat \E_n [ YZ_1 \sigma(\kappa - c_1 YZ_1 - c_2 Z_2) ]}{( \widehat \E_n [ \sigma^2(\kappa - c_1 YZ_1 - c_2 Z_2)])^{1/2}} \enspace,\\
		\partial_2 \widehat F_{\kappa}(c_1, c_2) = -\frac{\widehat \E_n [ Z_2 \sigma(\kappa - c_1 YZ_1 - c_2 Z_2) ]}{( \widehat \E_n [ \sigma^2(\kappa - c_1 YZ_1 - c_2 Z_2)])^{1/2}} \enspace,
	\end{align}
	where $\sigma(t) := \max(t, 0)$ satisfies the positive homogeneity $\sigma(|c| t) = |c| \sigma(t)$. 
	
	We prove the claim by dividing $c_2$ into two regions, $(0, M]$ and $(M, \infty)$.
	
	In the first region, where $(c_1, c_2) \in [-M, M]\times (0, M]$, it is easy to verify that $R_1(c_1, c_2) := YZ_1 \sigma(\kappa - c_1 YZ_1 - c_2 Z_2)$, $R_2(c_1, c_2):=Z_2 \sigma(\kappa - c_1 YZ_1 - c_2 Z_2)$ and $R_0(c_1, c_2):=\sigma^2(\kappa - c_1 YZ_1 - c_2 Z_2)$ are all sub-exponential random variables with sub-exponential parameters being at most a constant (depends on $M$), since $\sigma(\kappa - c_1 YZ_1 - c_2 Z_2), YZ_1, Z_2$ are all sub-Gaussian random variables. Denote the $\epsilon$-covering net as $\cN_{\epsilon}([-M, M]\times (0, M])$, we know that on this bounded region, with probability at least $1- n^{-2}$, 
	\begin{align}
		&\sup_{(c_1, c_2) \in [-M, M]\times (0, M]} ~\big|\widehat \E_n [R_j(c_1, c_2)] - \E  [R_j(c_1, c_2)] \big|  \nonumber\\
		& \leq   \sup_{(c_1', c_2') \in \cN_\epsilon} ~\big|\widehat \E_n [R_j(c_1', c_2')] - \E  [R_j(c_1', c_2')] \big| \nonumber \\
		& +  \sup_{(c_1, c_2) \in [-M, M]\times (0, M]} \inf_{(c_1', c_2') \in \cN_\epsilon} ~\big|\widehat \E_n [R_j(c_1, c_2)] - \widehat\E_n  [R_j(c_1', c_2')] \big| \nonumber\\
		& \quad \quad + \sup_{(c_1, c_2) \in [-M, M]\times (0, M]} \inf_{(c_1', c_2') \in \cN_\epsilon} ~\big| \E [R_j(c_1, c_2)] - \E  [R_j(c_1', c_2')] \big| \nonumber\\
		& \precsim \frac{\log \frac{1}{\epsilon^2}}{\sqrt{n}} + (\log n + 1) \epsilon \precsim \frac{ \log n}{\sqrt{n}}, ~\forall j \in 0,1,2 \enspace.
	\end{align}
	The above bound is derived with $\epsilon \asymp 1/\sqrt{n}$.
	Recall that $\E [R_0(c_1, c_2)] = F_{\kappa}(c_1, c_2) > 0$. Then for $n$ large enough, the claim follows since
	\begin{align*}
		& |\partial_1 \widehat F_{\kappa}(c_1, c_2) - \partial_1  F_{\kappa}(c_1, c_2)| \leq  \frac{\big|\widehat \E_n [R_1(c_1, c_2)] -  \E [R_1(c_1, c_2)] \big| }{\sqrt{ \E [R_0(c_1, c_2)] }}  \nonumber \\
		& + \frac{\big| \sqrt{\widehat \E_n [R_0(c_1, c_2)]} -  \sqrt{\E [R_0(c_1, c_2)]} \big| \cdot \big| \widehat \E_n [R_1(c_1, c_2)] \big|}{\sqrt{ \E [R_0(c_1, c_2)] \widehat \E_n [R_0(c_1, c_2)]} } \precsim \frac{\log n}{\sqrt{n}}
	\end{align*}
	w.p. at least $1-n^{-2}$ uniformly for all $|c_1|\leq M, 0<c_2\leq M$.
	
	For the second region (unbounded), where $(c_1, c_2) \in [-M, M]\times (M, \infty)$, we use the following self-normalization property of $\partial_i \widehat F_{\kappa}(c_1, c_2)$
	\begin{align}
		\partial_1 \widehat F_{\kappa}(c_1, c_2) = -\frac{\widehat \E_n [ YZ_1 \sigma(\kappa c_2^{-1} - c_1 c_2^{-1} YZ_1 -  Z_2) ]}{( \widehat \E_n [ \sigma^2(\kappa c_2^{-1} - c_1 c_2^{-1} YZ_1 - Z_2)])^{1/2}} \enspace, \label{eq:partial-1}\\
		\partial_2 \widehat F_{\kappa}(c_1, c_2) = -\frac{\widehat \E_n [ Z_2 \sigma(\kappa c_2^{-1} - c_1 c_2^{-1} YZ_1 - Z_2) ]}{( \widehat \E_n [ \sigma^2(\kappa c_2^{-1} - c_1 c_2^{-1} YZ_1 - Z_2)])^{1/2}} \enspace. \label{eq:partial-2}
	\end{align}
	Now the regions for the parameters of interest are bounded since
	\begin{align}
		(c_2^{-1}, c_1 c_2^{-1}) \in [0, 1/M) \times (-1, 1). 
	\end{align}
	Now define $a = c_2^{-1}, b= c_1 c_2^{-1}$, $\tilde R_1(a, b) := YZ_1 \sigma(\kappa a - b YZ_1 -  Z_2)$, $\tilde R_2(a, b):=Z_2 \sigma(\kappa a - b YZ_1 - Z_2)$ and $\tilde R_0(a, b):=\sigma^2(\kappa a - b YZ_1 - Z_2)$ are all sub-exponential random variables with sub-exponential parameters being at most a constant on the region $(c_1, c_2) \in [-M, M]\times (M, \infty)$. A standard $\epsilon$-covering $\cN_\epsilon([0, 1/M) \times (-1, 1))$ on $(a, b) := (c_2^{-1}, c_1 c_2^{-1})$ completes the proof for the region $(c_1, c_2) \in [-M, M]\times (M, \infty)$, since
	\begin{align}
		&\sup_{(c_1, c_2) \in [-M, M]\times (M, \infty)} ~\big|\widehat \E_n [\tilde R_j(c_2^{-1}, c_1 c_2^{-1})] - \E  [\tilde R_j(c_2^{-1}, c_1 c_2^{-1})] \big| \nonumber \\
		& \leq   \sup_{(a, b) \in \cN_\epsilon} ~\big|\widehat \E_n [\tilde R_j(a, b)] - \E  [\tilde R_j(a, b)] \big| \nonumber \\
		&+  \sup_{(c_1, c_2) \in [-M, M]\times (M, \infty)} \inf_{(a, b) \in \cN_\epsilon} ~\big|\widehat \E_n [\tilde R_j(c_2^{-1}, c_1 c_2^{-1})] - \widehat\E_n  [\tilde R_j(a, b)] \big| \nonumber\\
		& \quad \quad + \sup_{(c_1, c_2) \in [-M, M]\times (M, \infty)} \inf_{(a, b) \in \cN_\epsilon} ~\big| \E [\tilde R_j(c_2^{-1}, c_1 c_2^{-1})] - \E  [\tilde R_j(a, b)] \big| \nonumber\\
		& \precsim \frac{\log \frac{1}{\epsilon^2}}{\sqrt{n}} + (\log n + 1) \epsilon \precsim \frac{ \log n}{\sqrt{n}}, ~\forall j \in 0,1,2 \enspace.
	\end{align}
	The proof can be completed following standard algebra based on the expression \eqref{eq:partial-1} and \eqref{eq:partial-2}, since
	\begin{align}
		\partial_j \widehat F_{\kappa}(c_1, c_2) = - \frac{\widehat \E_n [\tilde R_j(c_2^{-1}, c_1 c_2^{-1}) ]}{\sqrt{\widehat \E_n [\tilde R_0(c_2^{-1}, c_1 c_2^{-1}) ] }} \enspace.
	\end{align}
\end{proof}

\subsection{Proof Outline for Generalization Error}

\begin{proof}[Proof of Theorem \ref{thm:gen-error}] The proof follows by an adaptation of \citep[Section E]{montanari2019generalization}, on using Theorem \ref{thm:l-1-margin} and Proposition \ref{prop:large-n-p-limit}. Here we provide an outline of the argument. Note that since the model \eqref{eq:DGP} involves Gaussian covariates, by rotation, we can equivalently express it as a model where all but the first coordinate of the true signal is zero. Thus, 

\begin{equation}
 \mathbb{P}_{(\bx, \by) }\left( \by \cdot \bx^\top \hat \theta_{n,\ell_1} < 0 \right)  = \mathbb{P}\left(c_{1,n} Y Z_1 + \sqrt{1-c_{1,n}^2}Z_2 \right),
\end{equation}
where $c_{1,n} = \frac{\langle \hat{\theta}_{n,\ell_1},\theta_{\star}\rangle_{\Lambda}}{\|\hat{\theta}_{n,\ell_1} \|_{\Lambda} \|\theta_{\star} \|_{\Lambda}}$ and $(Y,Z_1,Z_2)$ satisfies the joint distribution given by \eqref{eq:YZ}. Recall that $\langle u,v\rangle_{\Lambda} = u^{\top}\Lambda v, \| u\|_{\Lambda} = u^{\top} \Lambda u$. Thus it suffices to show that 

\begin{equation}\label{eq:c1goal}
c_{1,n} \stackrel{\text{a.s.}}{\rightarrow} c_1^{\star},
\end{equation}
 where $c_1^{\star}$ is defined following \eqref{eq:analyic-gen-error}. 

Now, recall that the min-$\ell_1$-norm interpolant solves 
\begin{align}
	\xi_{\psi, \kappa}^{(n,p)} = \min_{\| \theta \|_1 \leq \sqrt{p}} \frac{1}{\sqrt{p}}\| (\kappa \mathbf{1} - (y\odot X)\theta)_+ \|_2 \enspace. \nonumber
\end{align}
For any compact set $\Theta_p$, define 
$$\xi_{\psi, \kappa}^{(n,p)}(\Theta_p) = \min_{\theta \in \Theta_p} \frac{1}{\sqrt{p}}\| (\kappa \mathbf{1} - (y\odot X)\theta)_+ \|_2 \enspace. \nonumber $$

If one can show that 
\begin{equation}\label{eq:compact}
\xi_{\psi, \kappa}^{(n,p)} (\Theta_p) > \xi_{\psi, \kappa}^{(n,p)} ,
\end{equation} 
then $\mathbb{P} (\hat{\theta}_{n,\ell_1} \in \Theta_p) \rightarrow 0$
 as $n \rightarrow \infty$. This is the key idea behind the proof. Naturally, this suggests choosing $\Theta_p$ to be compact sets of the form 
 $$\Theta_p = \left[\theta: \|\theta \|_1 \leq \sqrt{p}, c_{1,n} \in [c_1^{\star}-\epsilon,c_1^{\star}+\epsilon]^{\text{c}} \right],$$
and establishing \eqref{eq:compact} for every $\epsilon > 0$. To formally establish this argument, define $\hat{k}_n = \min_{1 \leq i \leq n} y_i x_i^{\top}\hat{\theta}_{n,\ell_1}$ and note that

\begin{equation}\label{eq:hatkappa}
\hat{k}_n = \sup_{\kappa}\{\xi_{\psi, \kappa}^{(n,p)}= 0 \}.
\end{equation}

Further, define  $\xi_{\psi, \kappa}^{(n,p)}(c) = \min_{\|\theta \|_1 \leq \sqrt{p}, \frac{\langle \theta,\theta_{\star}\rangle_{\Lambda}}{\|\theta \|_{\Lambda} \|\theta_{\star} \|_{\Lambda}}=c} \frac{1}{\sqrt{p}}\| (\kappa \mathbf{1} - (y\odot X)\theta)_+ \|_2 $ and note that 
 
 \begin{equation}\label{eq:c1}
 c_{1,n} \in \{c:\xi_{\psi, \hat{\kappa}_n}^{(n,p)}(c) = 0 \}.
 \end{equation}
 
Now, for any $c_1, c_2 \in [-1,1]$, define 
$$ \xi_{\psi, \kappa}^{(n,p)} (c_1,c_2) = \min \left\{\min_{c \leq c_1 } \xi_{\psi, \kappa}^{(n,p)}(c), \min_{c \geq c_2 } \xi_{\psi, \kappa}^{(n,p)}(c)\right\}.$$

To show \eqref{eq:c1goal}, from \eqref{eq:hatkappa}-\eqref{eq:c1}, the final step then involves establishing that for any $\epsilon > 0$
$$ \lim_{n,p \infty, p/n \rightarrow \psi}\mathbb{P}[\xi_{\psi, \hat{\kappa}_n}^{(n,p)} (c_1^{\star}-\epsilon,c_1^{\star}+\epsilon) > 0 ] =1.$$
This can be established by analytic arguments similar to \citep[Section E]{montanari2019generalization}, on using the limiting characterizations of $\hat{\kappa}_n$ and $\xi_{\psi, \hat{\kappa}_n}^{(n,p)} $ from Theorem \ref{thm:l-1-margin} and Proposition \ref{prop:large-n-p-limit}. 

\end{proof}

\subsection{Uniqueness Results}
\label{sec:uniqueness-result}
		We next present the proof of Proposition \ref{prop:uniqueness}. The first part of the proof is the same as that in \cite{montanari2019generalization}, but we include it for the sake of completeness. The second part of the proof, though draws inspiration from \cite{montanari2019generalization}, has different executions in this $\ell_1$ case.
	
\begin{proof}[Proof of Proposition \ref{prop:uniqueness}]
	To analyze the equation system \eqref{eq:fix-points-l1}, we will, in fact, begin by examining the objective function in  \eqref{eq:remind-infinite} as a function of $h$, that is, define 
	\[\mathcal{R}_{\psi,\kappa,Q_{\infty}}= \psi^{-1/2} F_{\kappa} \left( \langle W,  \Lambda^{1/2} h \rangle_{L_2(\cQ_{\infty})},\| \Pi_{W^\perp}(\Lambda^{1/2}h) \|_{L_2(\cQ_{\infty})}  \right) +  \left\langle \Pi_{W^\perp}(G), \Lambda^{1/2} h \right\rangle_{L_2(\cQ_{\infty})}, \]
	and consider the optimization problem
	\begin{equation}\label{eq:Rvalue}
	\mathrm{minimize}  \qquad \mathcal{R}_{\psi,\kappa,Q_{\infty}}(h) \qquad \mathrm{s.t.} \qquad \|h \|_{L_1(Q_{\infty})} \leq 1.
	\end{equation}

	Due to  the form of the constraint set, it follows from the Banach-Alaoglu theorem that the minimum is achieved for some $h^{\star} \in L_2(\cQ_{\infty})$. Further, using the fact that $F_{\kappa}$ is convex and increasing with respect to the second argument (see \citep[Lemma 5.3.(b)]{montanari2019generalization}), it can be shown that the function $h \rightarrow  \mathcal{R}_{\psi,\kappa,Q_{\infty}}$ is strictly convex. This immediately implies uniqueness of the minimizer, in the sense that, given two minimizers $h^{\star}$ and $\tilde{h}$, one must have $\mathbb{P}[h^{\star} \neq \tilde{h}] = 0$.  
	 Then the unique minimizer is determined by the KKT conditions, which in this case can be expressed as 
 
	 \begin{align}\label{eq:kkt-explain}
		\Lambda^{1/2} \Pi_{W^\perp}(G) + \psi^{-1/2} \Lambda^{1/2}  \left[ \partial_1 F_{\kappa}(c_1, c_2) W + \partial_2 F_{\kappa}(c_1, c_2) \Pi_{W^\perp}(Z) \right] + s \cdot \partial \| h \|_{L_1(\cQ_{\infty})} = 0 \enspace, \nonumber \\
		s(1- \| h \|_{L_1(\cQ_{\infty})}) = 0 \enspace,  \\
		s\geq 0, \| h \|_{L_1(\cQ_{\infty})} \leq 1 \enspace.\nonumber
	\end{align}
	Above, $Z$ is given by 
	\[ Z = \twopartdef{\frac{\Pi_{W^\perp}(\Lambda^{1/2}h)}{\|\Pi_{W^{\perp}}(\Lambda^{1/2}h) \|}}{\|\Pi_{W^{\perp}}(\Lambda^{1/2}h) \|>0}{Z'(G,\Lambda,W) \quad \mathrm{s.t.} \quad \|Z' \|\leq 1}{\|\Pi_{W^{\perp}}(\Lambda^{1/2}h) \|=0},\]
	and
	 
	 \begin{equation}\label{eq:c1c2def}
	 c_1=  \langle W,  \Lambda^{1/2} h \rangle_{L_2(\cQ_{\infty})}, c_2= \| \Pi_{W^\perp}(\Lambda^{1/2}h) \|_{L_2(\cQ_{\infty})}.
	 \end{equation}
	 
	We claim that any solution of \eqref{eq:kkt-explain} and the associated dual variable $s$ satisfy $s > 0$ and $\| \Pi_{W^{\perp}}(\Lambda^{1/2}h)\| > 0$. The former follows directly from \citep[Section B.3.3]{montanari2019generalization}, but we describe the detail here for the sake of completion. Recall that $(\Lambda, W ) \sim \mu$ defined in \eqref{eq:limit-measure-2}. From properties of $\mu$ it follows that $\Lambda > 0, \| W\|=1$. Suppose if possible that $s=0$, then \eqref{eq:kkt-explain} implies that

\begin{equation}\label{eq:firstrel}
 \Pi_{W^\perp}(G) + \psi^{-1/2} \left[ \partial_1 F_{\kappa}(c_1, c_2) W + \partial_2 F_{\kappa}(c_1, c_2) \Pi_{W^\perp}(Z) \right] = 0. 
 \end{equation}
	
	Taking inner products with $W$ on both sides, we obtain $\psi^{-1/2} \left[ \partial_1 F_{\kappa}(c_1, c_2) \right] = 0.$	Using this relation back in \eqref{eq:firstrel}, we obtain that 
	
	\begin{align*}
	 \psi^{-1/2} \partial_2 F_{\kappa}(c_1, c_2) \Pi_{W^\perp}(Z) & = - \Pi_{W^\perp}(G) \\
	 \implies   \psi^{-1/2} \partial_2 F_{\kappa}(c_1, c_2) & \geq \|\Pi_{W^\perp}(G)\|,
	 \end{align*}
	 
	 by taking norm on both sides and noting that (i) the partial derivative with respect to the second coordinate of $F_\kappa(\cdot,\cdot)$ is always positive (ii) $\| \Pi_{W^\perp}(Z)\| \leq \|Z \| \leq 1.$ From \citep[Lemma B.1]{montanari2019generalization}, we know that if $(c_1,c_2)$ is a tuple satisfying $ \partial_1 F_{\kappa}(c_1, c_2) =0$, then the partial derivative with respect to the second coordinate at the same point can be at most square root of the separability threshold, that is, $\partial_2 F_{\kappa}(c_1, c_2) \leq \min_{c \in \mathbb{R}} F_0(c,1) = \sqrt{\psi^{\star}(\rho,f)}$. Together this yields that $\sqrt{\psi}\|\Pi_{W^\perp}(G) \| \leq \sqrt{\psi^{\star}(\rho,f)}$,which, from the definition of $\psik$ \eqref{eq:psis}, and the fact that $W,G$ are independent, contradicts our assumption that $\psi > \psik$ in the hypothesis of the proposition.

	 We next proceed to show that for any solution $h$, $c_2=\| \Pi_{W^{\perp}}(\Lambda^{1/2}h)\| > 0$. Suppose by contradiction that $c_2 = 0$. By decomposing $h$ in the direction of $W$ and $W^{\perp}$, observe that in this case 
	 
	 \begin{equation}\label{eq:hspecial}
	 \Lambda^{1/2} h = c_1 W, 
	 \end{equation}
	 
	  where recall the definition of $c_1$ from \eqref{eq:c1c2def}. Since we established $s > 0$, for any solution, $ \| h \|_{L_1(\cQ_{\infty})} = 1$. This yields that  in this case,
	 
	 \begin{equation}\label{eq:c1value}
	  |c_1| = \zeta 	 
	 \end{equation}
	 
	 Now divide the problem into two cases: Case (i): $c_1 > 0$ and Case (ii): $c_1 < 0$. Here we only show the argument  for Case (i), since the other case follows similarly. 
	From the first equation in \eqref{eq:kkt-explain}, we have 
	
	\begin{equation}\label{eq:kktfirst}
	\Pi_{W^\perp}(G) + \psi^{-1/2}\left[ \partial_1 F_{\kappa}(c_1, c_2) W + \partial_2 F_{\kappa}(c_1, c_2) \Pi_{W^\perp}(Z) \right] + s \cdot\Lambda^{-1/2} \partial \| h \|_{L_1(\cQ_{\infty})} = 0
	\end{equation}
	
	Taking inner product with $W$ and using $\mathbb{E}[W^2]=1$, and using the facts that $c_1=\zeta, c_2=0$, \eqref{eq:hspecial}, and $ \| h \|_{L_1(\cQ_{\infty})} =1$ for a solution $h$, we obtain 
	
	\begin{align}\label{eq:simplify}
	\psi^{-1/2}\partial_1 F_{\kappa}(\zeta,0) + s \langle \Lambda^{-1/2}W, \partial \| h \|_{L_1(\cQ_{\infty})} \rangle   & = 0\nonumber \\
	\psi^{-1/2}\partial_1 F_{\kappa}(\zeta,0) + s c_1^{-1}\langle h, \partial \| h \|_{L_1(\cQ_{\infty})} \rangle & = 0 \nonumber \\
	\psi^{-1/2}\partial_1 F_{\kappa}(\zeta,0) + s c_1^{-1}  \| h \|_{L_1(\cQ_{\infty})} & = 0 \nonumber \\
	s = - c_1 	\psi^{-1/2}\partial_1 F_{\kappa}(\zeta,0) 
	\end{align}
	
	Since $s > 0$, this yields that $\partial_1 F_{\kappa}(\zeta,0)  < 0$, which implies that by definition, $\psi$ should be above the threshold $\psi_{+}(\kappa)$ that satisfies 
	
	\begin{equation}\label{eq:psirange}
	\mathbb{E}\left[ \left\{\psi_{+}(\kappa)^{1/2} \Pi_{W^{\perp}}(G) + \partial_1 F_{\kappa}(\zeta,0)(W - \zeta \Lambda^{-1/2}\text{sign}(\zeta \Lambda^{-1/2}W))\right\}^2\right] = \partial_2^2F_{\kappa}(\zeta,0). 
	\end{equation}

Now plugging \eqref{eq:simplify} back in \eqref{eq:kktfirst} and using \eqref{eq:hspecial}, we obtain that 
	
	\begin{align}\label{eq:reduce1}
	\psi^{1/2}\Pi_{W^\perp}(G) +\left[ \partial_1 F_{\kappa}(\zeta, 0) W + \partial_2 F_{\kappa}(\zeta, 0) \Pi_{W^\perp}(Z) \right] & -  c_1\partial_1 F_{\kappa}(\zeta,0)  \cdot\Lambda^{-1/2} \text{sign}(\zeta \Lambda^{-1/2}W)  = 0 \nonumber \\
	\psi^{1/2}\Pi_{W^\perp}(G) +\partial_1 F_{\kappa}(\zeta, 0)(W-c_1\Lambda^{-1/2} \text{sign}(\zeta \Lambda^{-1/2}W)) & = - \partial_2 F_{\kappa}(\zeta, 0) \Pi_{W^\perp}(Z)
	\end{align}
If we take $\ell_2$ norm on both sides of the above, and recall that $\| \Pi_{W^{\perp}}(Z)\|_2 < 1$ for $c_2=0$, we obtain 
	
	\[\mathbb{E} \left[\{\psi^{1/2}\Pi_{W^\perp}(G) +\partial_1 F_{\kappa}(\zeta, 0)(W-c_1\Lambda^{-1/2} \text{sign}(\zeta \Lambda^{-1/2}W))\}^2  \right] < \partial_2F_{\kappa}(\zeta,0)^2.\]
	
	However, this contradicts the range of $\psi$ determined by \eqref{eq:psirange}. The case of $c_1<0$ can be similarly handled on recalling the fact that, by definition, $\psi > \psi_{-}(\kappa) $ introduced in \eqref{eq:psis}.
	Thus, we conclude that $c_2>0$.

Now that we have established that  $c_2 $ and $s$ must be strictly positive when $(c_1,c_2,s)$ solves our equation system, we can proceed to explicitly identify the formula for the solution $h^{\star}$ to \eqref{eq:kkt-explain}. The KKT conditions yield that 
	
	 \begin{align}\label{eq:kkt-appendix}
	  \psi^{-1/2} c_2^{-1} \partial_2 F_{\kappa}(c_1, c_2) \Lambda^{1/2} h  &+ s \cdot \Lambda^{-1/2} \partial \| h \|_{L_1(\cQ_{\infty})} \nonumber \\
	&    =-(  \Pi_{W^\perp}(G) + \psi^{-1/2}  \left[ \partial_1 F_{\kappa}(c_1, c_2)  - c_1 c_2^{-1} \partial_2 F_{\kappa}(c_1, c_2) \right] W) 
	\end{align}
	 
	 From \citep[Section B.3]{montanari2019generalization}, $\partial_2 F_{\kappa}(\zeta,0) > 0$, so we may rewrite the above as follows:
	 
	 \[ h + \frac{s \cdot \Lambda^{-1/2} \partial \| h \|_{L_1(\cQ_{\infty})} }{ \psi^{-1/2} c_2^{-1} \partial_2 F_{\kappa}(c_1, c_2) \Lambda^{1/2}} = -\frac{(  \Pi_{W^\perp}(G) + \psi^{-1/2}  \left[ \partial_1 F_{\kappa}(c_1, c_2)  - c_1 c_2^{-1} \partial_2 F_{\kappa}(c_1, c_2) \right] W)}{\psi^{-1/2} c_2^{-1} \partial_2 F_{\kappa}(c_1, c_2) \Lambda^{1/2}}.  \]
	 
	 Now the above implies that the solution $h^{\star}$ is given by
	 \begin{align*}
	 h^{\star} & = \prox_{\frac{s \cdot \Lambda^{-1/2}}{ \psi^{-1/2} c_2^{-1} \partial_2 F_{\kappa}(c_1, c_2) \Lambda^{1/2}}}\left(-\frac{(  G+ \psi^{-1/2}  \left[ \partial_1 F_{\kappa}(c_1, c_2)  - c_1 c_2^{-1} \partial_2 F_{\kappa}(c_1, c_2) \right] W)}{\psi^{-1/2} c_2^{-1} \partial_2 F_{\kappa}(c_1, c_2) \Lambda^{1/2}}\right), \\
	& =  -\frac{\Lambda^{-1}}{\psi^{-1/2} c_2^{-1} \partial_2 F_{\kappa}(c_1, c_2)} \prox_{s}(\Lambda^{1/2}G + \psi^{-1/2}  \left[ \partial_1 F_{\kappa}(c_1, c_2)  - c_1 c_2^{-1} \partial_2 F_{\kappa}(c_1, c_2) \right] \Lambda^{1/2}W )
	 \end{align*}

	Plugging this in the system 
	\begin{align}\label{eq:fpappendix}
		c_1 = \langle \Lambda^{1/2}  h^{\star}, W \rangle_{L_2(\cQ_{\infty})}, \qquad 	c_1^2 + c_2^2 = \|  \Lambda^{1/2}  h^{\star} \|^2_{L_2(\cQ_{\infty})}, \qquad \| h^{\star}\|_{L_1(\cQ_{\infty})} = 1
		\end{align}
		yields the fixed point equations \eqref{eq:fix-points-l1}. Since the solution $h^{\star}$ is unique, the values  $c_1 := \langle \Lambda^{1/2}  h^{\star}, W \rangle_{L_2(\cQ_{\infty})} \, \,$, $c_2 := \| \Pi_{W^\perp}(\Lambda^{1/2} h^{\star}) \|_{L_2(\cQ_{\infty})}$ and the value $s$ satisfying \eqref{eq:kkt-appendix} are also unique and, furthermore, $c_2$ and $s$ are strictly positive. 
\end{proof}

	We obtain a key representation for $\tilde \xi^{(\infty, \infty)}_{\psi, \kappa}( \Lambda, W, G)$ as a byproduct of the above. On taking inner products with $\Lambda^{1/2}h$ on both sides of \eqref{eq:kkt-appendix} leads to the following.

\begin{corollary}\label{cor:values}
Under the assumptions of Proposition \ref{prop:uniqueness}, the minimum value of the optimization problem \eqref{eq:Rvalue} is given by 
\[\tilde \xi^{(\infty, \infty)}_{\psi, \kappa}( \Lambda, W, G)  = \psi^{-1/2} \left[ F_{\kappa}(c_1, c_2) - c_1 \partial_1 F_{\kappa}(c_1, c_2)  - c_2 \partial_2 F_{\kappa} (c_1, c_2) \right] - s,\]
where $(c_1,c_2,s) \in \Reals \times \Reals_{>0} \times \Reals_{>0}$ forms the unique solution to \eqref{eq:fix-points-l1}. Hence, the above equals $T(\psi,\kappa)$ defined in \eqref{eq:Tfunction}. 	\end{corollary}

\begin{remark}
For the setting of Corollary \ref{cor:lq}, note that $F_{\kappa}(\cdot,\cdot)$ remains the same as that in the case of the $\ell_1$ geometry. Therefore, the arguments in Section \ref{sec:prob-analysis} naturally extend  to the $\ell_q$ geometry (as long as $q \leq 2$), and those in the current section can also be extended to show uniqueness of the system \eqref{eq:system-Lq} on changing the definition of $\zeta$ and establishing bounds on $\langle W,\Lambda^{-1/2} \partial \| h \|_{L_q(\cQ_{\infty})} \rangle$  appropriately. 
\end{remark}

\subsection{Optimization Results}

\begin{proof}
[Proof of Proposition \ref{prop:AdaBoost}]
We will show the convergence of $\gba$ as a special instantiation of the \textit{Mirror Descent} proof. We will establish the result for two scenarios: (1) AdaBoost, with $X_{ij} \in \{ \pm 1\}$, and (2) $\gba$ from Section \ref{sec:crucial}, with bounded continuous $|X_{ij}| \leq M$ and a shrinkage on the learning rate (the specifics will be made clear in the proof below). Note that in the discrete case (Case (1)), Steps (a) and (b) in the \gba from Section \ref{sec:crucial} could be replaced by 
\begin{align*}
 v_{t+1} & := \argmin_{v \in \{ \pm e_j \}_{j\in [p]}} \sum_{i\in [n]} \eta_t[i]\cdot \mathbbm{I}_{y_i x_i^\top v \leq 0} \\
 \alpha_{t} & = \frac{1}{2} \log \left( \frac{1 - \sum_{i\in [n]} \eta_t[i]\cdot \mathbbm{I}_{y_i x_i^\top v_{t+1} \leq 0}}{ \sum_{i\in [n]} \eta_t[i]\cdot \mathbbm{I}_{y_i x_i^\top v_{t+1} \leq 0}}  \right) . 
 \end{align*}

We will need some background before stating the mirror descent proof. For $x \in \mathbb{R}^n$, define the entropy 
\begin{align}
	R(x) = \sum_{i=1}^n x[i] \log(x[i]) + \mathbbm{I}_{\Delta_n}(x) \enspace.
\end{align}
Here $\mathbbm{I}_{\Delta_n}$ is the indicator function on the probability simplex $\Delta_n$.
The Fenchel conjugate of $R$, denoted by $R^\star$, reads,
\begin{align}
	\label{eq:fenchel}
	R^\star(x) = \log \left( \sum_{i=1}^n \exp(x[i]) \right) \enspace.
\end{align}
One can verify that $R$ is $1$-strongly convex w.r.t. the $\ell_1$ norm, and that 
$R^\star$ is $1$-strongly smooth w.r.t. the $L_\infty$ norm.

First, let us recall the dual formulation of $\ell_1$-margin, and the von Neumann's minimax theorem 
\begin{align}
	\kappa_{n,\ell_1} = \max_{\|\theta\|_1 \leq 1} \min_{i\in [n]} e_i^\top Z \theta = \min_{\eta \in \Delta_n} \max_{\|\theta\|_1 \leq 1} \eta^\top Z \theta = \min_{\eta \in \Delta_n}  \| Z^\top \eta \|_\infty \enspace.
\end{align}
Therefore, for any $\eta \in \Delta_n$, $\kappa_{n,\ell_1} \leq \| Z^\top \eta \|_\infty$.

It is easy to verify that the (1) AdaBoost algorithm defined above is equivalent to the following mirror descent algorithm:
\begin{itemize}
	\item $\ell_1$-margin $\gamma_t := \max_{j \in [p]} |\eta_t^\top Z e_j| = \| Z^\top \eta_t \|_\infty \geq \kappa_{n,\ell_1}$ ;
	\item Learning rate is $\alpha_t = \frac{1}{2} \log \frac{1+\gamma_t}{1-\gamma_t}$ since
	\begin{align}
		\min_{v \in \{ \pm e_j \}_{j\in [p]}} \sum_{i\in [n]} \eta_t[i]\cdot \mathbbm{I}_{y_i x_i^\top v \leq 0}  = \min_{v \in \{ \pm e_j \}_{j\in [p]}} \sum_{i\in [n]} \eta_t[i]\cdot \mathbbm{I}_{ - y_i x_i^\top v \geq 0} \\
	= \frac{1}{2}( - \max_{j \in [p]} |\eta_t^\top Z e_j| + 1) \enspace;
	\end{align}
	\item Updates on $\eta_t \in \Delta_n$ (mirror descent) reduce to
	\begin{align}
		\label{eq:mirror-descent-step}
		\nabla R(\eta_t) &= - Z \theta_t ~~ \text{(map to mirror space),}\\
		Z \theta_{t+1} &= Z \theta_{t} + \alpha_t Z v_{t+1} ~~\text{(descent step),} \\
		\nabla R^\star( - Z \theta_{t+1}) &= \eta_{t+1} ~~ 	\text{(inverse map).}
	\end{align}
\end{itemize}

Now we are ready to prove the final statement. Due to the fact that $R^\star$ is strongly smooth w.r.t. the $L_\infty$ norm
\begin{align*}
	&R^\star(-Z \theta_{t+1}) - R^\star(-Z \theta_{t}) \\
	& \leq \langle - \alpha_t Z v_{t+1}, \nabla R^\star (-Z \theta_{t}) \rangle + \frac{1}{2} \| \alpha_t Z v_{t+1} \|_\infty^2 \\
	& \leq - \alpha_t \langle Z v_{t+1}, \eta_t \rangle  + \frac{1}{2} \alpha^2_t \|Z v_{t+1} \|_{\infty}^2\\
	& = -\alpha_t \| Z^\top \eta_t \|_\infty + \frac{1}{2} \alpha^2_t  \quad \text{(here we use the fact that $|Z_{ij}| \leq 1$)}\\
	& = -\alpha_t \gamma_t + \frac{1}{2} \alpha^2_t  \leq -\frac{\gamma_t^2}{2} ( 1+ o(\gamma_t)) \enspace.
\end{align*}
The above derives the reduction in $R^\star$ for each step.

For the (2) $\gba$ from Section \ref{sec:crucial}, with $|X_{ij}| \leq M$, define a shrinkage on the learning rate $\alpha_t(\beta)$ with a constant factor $\beta>0$,
\begin{align}
	\alpha_t(\beta) = \beta \cdot \eta_t^\top Z v_{t+1} \enspace.
\end{align}
A good choice of $\beta$ will be clear in a second.
Then
\begin{align*}
	&R^\star(-Z \theta_{t+1}) - R^\star(-Z \theta_{t}) \\
	& = -\alpha_t(\beta) \| Z^\top \eta_t \|_\infty + \frac{1}{2} \alpha_t^2(\beta) \|Z v_{t+1} \|_{\infty}^2 \quad \text{(here we use the fact that $|Z_{ij}| \leq M$)}\\
	& = -\beta \gamma^2_t + \frac{M^2}{2} \beta^2 \gamma^2_t   = - \frac{\gamma_t^2}{2M^2} \enspace
\end{align*}
where the last step uses the choice of $\beta = 1/M^2$.

Now telescoping with the terms $R^\star(-Z \theta_{t+1}) - R^\star(-Z \theta_{t})$, we have
\begin{align}
	R^\star(-Z \theta_{T}) - R^\star(-Z \theta_{0}) & \leq - \frac{\sum_{t=0}^{T-1} \gamma_t^2}{2M^2}  \leq - T \frac{\kappa_{n,\ell_1}^2}{2M^2}  \quad \text{(recall $\gamma_t \geq \kappa_{n,\ell_1}$)} \\
	\sum_{i \in [n]} \mathbbm{I}_{-y_i x_i^\top \theta_{T} >0} &\leq \sum_{i \in [n]} \exp(-y_i x_i^\top \theta_{T}) = \exp(R^\star(-Z \theta_{T})) \leq ne \cdot \exp(- T \frac{\kappa_{n,\ell_1}^2}{2M^2}) \enspace.
\end{align}
The proof is now complete.

\end{proof}

\begin{proof}[Proof of Corollary~\ref{cor:boost-converge-max-margin}]
	The proof follows from Proposition~\ref{prop:AdaBoost} and a re-scaling technique in \cite{zhang2005boosting}'s asymptotic analysis. Here instead, we spell out a non-asymptotic result. For any $\kappa > 0$
	\begin{align}
		\sum_{i \in [n]} \mathbbm{I}_{\frac{y_i x_i^\top \theta_t}{\| \theta_t \|_1} \leq \kappa} & \leq \sum_{i \in [n]} \exp(\kappa \| \theta_t \|_1 - y_i x_i^T \theta_t) \enspace, \\
		& \leq \exp(\kappa \| \theta_t \|_1) \exp \left( R^\star(- Z \theta_t) \right) \enspace,
	\end{align}
	with $R^\star$ defined in \eqref{eq:fenchel}. Due to the proof in Proposition~\ref{prop:AdaBoost}, we know
	\begin{align}
		R^\star(- Z \theta_T) &\leq R^\star( -Z \theta_0) - \sum_{t=0}^{T-1} \left( \beta \gamma^2_t - \frac{\beta^2 \gamma^2_t }{2} M^2 \right) \\
		& \leq \log (ne) - \sum_{t=0}^{T-1} \beta \gamma_t \left[ \gamma_t - \frac{\beta}{2} \gamma_t M^2 \right] \enspace.
	\end{align}
	In addition, due to the coordinate update of $\theta_t$, we know
	\begin{align}
		\| \theta_T \|_1 \leq \sum_{t=0}^{T-1} \| \alpha_t v_{t+1} \|_1 \leq \sum_{t=0}^{T-1}  \beta \gamma_t \enspace.
	\end{align}
	Therefore
	\begin{align}
		\label{eq:re-scale-trick}
		\sum_{i \in [n]} \mathbbm{I}_{\frac{y_i x_i^\top \theta_t}{\| \theta_t \|_1} \leq \kappa} \leq n e \cdot \exp\left\{  - \sum_{t=0}^{T-1}  \beta \gamma_t \left[ \gamma_t - \frac{\beta}{2} \gamma_t M^2  - \kappa \right] \right\} \enspace.
	\end{align}
	Recall that $\gamma_t \geq \kappa_{n,\ell_1}$ for all $t$, 
	we know that
	\begin{align}
		\sum_{i \in [n]} \mathbbm{I}_{\frac{y_i x_i^\top \theta_t}{\| \theta_t \|_1} \leq \kappa} \leq n e 
		\cdot \exp \left( - T \beta \kappa_{n,\ell_1} \left[ \kappa_{n,\ell_1}(1- \frac{\beta M^2}{2}) - \kappa \right] \right) \enspace.
	\end{align}
	With the choice of
	\begin{align}
		\beta &= \frac{1 -\kappa/\kappa_{n,\ell_1}}{M^2}, \quad \text{and} \\
		T &\geq \log (1.01ne) \cdot \frac{2 M^2 \kappa_{n,\ell_1}^{-2} }{(1 - \kappa/\kappa_{n,\ell_1})^2},
	\end{align}
	we know that
	\begin{align}
		\sum_{i \in [n]} \mathbbm{I}_{\frac{y_i x_i^\top \theta_t}{\| \theta_t \|_1} \leq \kappa } \leq \frac{1}{1.01} < 1 \enspace.
	\end{align}
	which implies that $\min_{i \in [n]} \frac{y_i x_i^\top \theta_T}{\| \theta_T \|_1} > \kappa$.
	Therefore for any $\epsilon <1$, plug in $\kappa = \kappa_{n,\ell_1} \cdot (1- \epsilon)$
	\begin{align}
		T \geq \log(1.01 ne) \cdot \frac{2M^2 \kappa_{n,\ell_1} ^{-2}}{\epsilon^2} \enspace,
	\end{align}
	we must have that
	\begin{align}
		\min_{i \in [n]} \frac{y_i x_i^\top \theta_T}{\| \theta_T \|_1} > \kappa_{n,\ell_1} \cdot (1- \epsilon) \enspace.
	\end{align}
\end{proof}

\begin{proof}[Proof of Corollary~\ref{cor:Lq-margin}]
	The proof follows by modifying some steps of our proof in the $q = 1$ case. Recall the notations in \eqref{eq:mirror-descent-step}, 
	\begin{align*}
		&R^\star(-Z \theta_{t+1}) - R^\star(-Z \theta_{t}) \\
		& \leq \langle - \alpha_t Z v_{t+1}, \nabla R^\star (-Z \theta_{t}) \rangle + \frac{1}{2} \| \alpha_t Z v_{t+1} \|_\infty^2 \\
		& \leq - \alpha_t \langle Z v_{t+1}, \eta_t \rangle  + \frac{1}{2} \alpha^2_t \|Z v_{t+1} \|_{\infty}^2\\
		& = -\alpha_t \| Z^\top \eta_t \|_{q_\star} + \frac{1}{2} \alpha^2_t \max_{i \in [n]} |\langle Z_{i\cdot}, v_{t+1} \rangle |^2  \quad \text{(here we use the fact that $|Z_{ij}| \leq M$)}\\
		& \leq  -\alpha_t \gamma_t + \frac{M^2 p^{\frac{2}{q_\star}}}{2} \alpha^2_t  = -\beta \gamma_t^2 + \frac{M^2 p^{\frac{2}{q_\star}}}{2} \beta^2 \gamma_t^2 
	\end{align*}
	with $\gamma_t = \| Z^\top \eta_t \|_{q_\star}$.

	Observe that
	\begin{align}
		\| \theta_T \|_{q} \leq \sum_{t=0}^{T-1} \| \alpha_t v_{t+1} \|_q \leq \sum_{t=0}^{T-1}  \beta \gamma_t \enspace.
	\end{align}

Plug in the above to the argument in \eqref{eq:re-scale-trick}, we have
\begin{align}
	\sum_{i \in [n]} \mathbbm{I}_{\frac{y_i x_i^\top \theta_t}{\| \theta_t \|_q} \leq \kappa} &\leq n e \cdot \exp\left\{  - \sum_{t=0}^{T-1}  \beta \gamma_t \left[ \gamma_t - \frac{\beta}{2} \gamma_t M^2 p^{\frac{2}{q_\star}}  - \kappa \right] \right\} \\
	&\leq n e \cdot \exp\left\{  - \sum_{t=0}^{T-1}  \beta \gamma_t \left[ \gamma_t (1 - \frac{\beta}{2} M^2 p^{\frac{2}{q_\star}})  - \kappa \right] \right\} \\
	& \leq n e 
		\cdot \exp \left( - T \beta \kappa_{n,\ell_q}^2 \left[ (1- \frac{\beta}{2}M^2 p^{\frac{2}{q_\star}}) - \frac{\kappa}{\kappa_{n,\ell_q}} \right] \right) \enspace.
\end{align}
where the last step uses the Sion's Minimax Theorem,
\begin{align}
	\gamma_t = \| Z^\top \eta_t \|_{q_\star} \geq \min_{\eta \in \Delta} \max_{\| \theta \|_{q} \leq 1} \eta^\top Z \theta = \max_{\| \theta \|_{q} \leq 1} \min_{i \in [n]} y_i x_i^\top \theta = \kappa_{n, \ell_q} \enspace.
\end{align}
The proof is complete if we plug in 
\begin{align}
	\beta = \frac{1-\kappa/\kappa_{n, \ell_q}}{p^{\frac{2}{q_\star}} M^2} .
\end{align}
\end{proof}

For completeness, we show that the min-$\ell_1$-norm interpolation, is equivalent to the max-$\ell_1$-margin formulation. We use this fact several places in the main text.
\begin{proposition}
The following two formulations are equivalent 
\begin{align}\label{eq:form-1} 
\text{Formulation I:} \quad I^\star:= \max \left\{ \kappa ~|~ \exists \theta, ~\| \theta \|_1 \leq 1,~~ \text{s.t.} ~~ \forall i\leq n,~ y_i x_i^\top\theta \geq \kappa \right\} 
\end{align}
\begin{align}
\text{Formulation II:} \quad II^\star:= \min ~\| \theta \|_1, ~~\text{s.t.}~~ \forall i\leq n, ~ y_i x_i^\top \theta \geq 1 
\end{align}
and that 
\begin{align*}
I^\star = 1/ II^\star \enspace. 
\end{align*}
\end{proposition}
\begin{proof}
Suppose that $\theta_\star$ solves $II$, then take $\theta = \theta_\star/II^\star$ satisfy $\| \theta \| =1$, then $$ I^\star \geq 1/II^\star \enspace. $$ Suppose that $I^\star$ is the optimal solution for $I$, then there exist a $\theta, \| \theta \| \leq 1$ such that $y_i x_i^\top (\theta/I^\star) \geq 1$, then $$ II^\star \leq \| \theta/I^\star \|_{1} \leq 1/I^\star \enspace. $$ 
\end{proof}


\section{Extended Derivations}
We collect here the detailed derivations in Section~\ref{sec:extend}, where robustness of the assumptions is investigated.
\subsection{Derivations in Section \ref{subsec:gmmext}}
\label{app:gmmext}

This section provides details on the results mentioned in Section \ref{subsec:gmmext}. Recall the generalization of GMMs considered in \eqref{eq:rank2gmm}: we observe i.i.d.~samples $(x_i,y_i)$ such that $\mathbb{P}[y_i=1] = v = 1- \mathbb{P}[y_i=-1]$ and 
$x_i = y_i \theta_\star + m_i \tilde{\theta} + \tilde{x}_i$. Here, $\tilde{x}_i \sim \mathcal{N}(0,\Lambda)$ with $\Lambda$ diagonal, $(y_i, m_i, \tilde{x}_i)$ are independent, and  $m_i$ is symmetric around zero so that $y_i \odot m_i \stackrel{\mathrm{d}}{=} m_i$.
Stacking $\tilde{x}_i$'s as the rows within a matrix $\tilde{X}$, we observe that $\tilde{X} = Z \Lambda^{1/2}$, where $Z$ has i.i.d.~$\mathcal{N}(0,1)$ entries. Similarly,  stacking $x_i$'s into rows of $X$, we obtain that 

\begin{equation}\label{eq:Xdist}
X = y\theta_{\star}^{\top}+m\tilde{\theta}^{\top} + \tilde{X},
\end{equation}

where $y=[y_1,\hdots,y_n]^{\top},m=[m_1,\hdots,m_n]^{\top}$. Recall from \eqref{eq:equivalence} that the max-min-$\ell_1$-margin properties can be characterized by analyzing the following optimization problem: 
\[\xi_n = \min_{\| \theta \|_{1} \leq \sqrt{p}}\max_{\|\lambda\|_2 \leq 1,\lambda \geq 0} \frac{1}{\sqrt{p}}[\lambda^{\top}(\kappa \boldsymbol{1} - (y \odot X)\theta)].\]

In the context of our model \eqref{eq:Xdist}, $y \odot X \stackrel{\mathrm{d}}{=} \boldsymbol{1}.\theta_{\star}^{\top} + m\tilde{\theta}^{\top}+Z\Lambda^{1/2}$, therefore $\xi_n$ simplifies to 

\begin{align*}
 \xi_n & = \min_{\| \theta \|_{1} \leq \sqrt{p}}\max_{\|\lambda\|_2 \leq 1,\lambda \geq 0}  \frac{1}{\sqrt{p}}[\lambda^{\top}\{(\kappa -\langle \theta_{\star},\theta \rangle)\boldsymbol{1} -\langle \tilde{\theta},\theta\rangle m - Z \Lambda^{1/2} \theta\}],
 \end{align*}
and by an application of CGMT, this is asymptotically equivalent to analyzing the following optimization problem
\begin{align*}
  \min_{\| \theta \|_{1} \leq \sqrt{p}}\max_{\|\lambda\|_2 \leq 1,\lambda \geq 0}  \frac{1}{\sqrt{p}} [\langle \lambda, (\kappa -\langle \theta_{\star},\theta \rangle)\boldsymbol{1} -\langle \tilde{\theta},\theta\rangle m -\|\Lambda^{1/2}\theta \|_2 z \rangle - \|\lambda \|_2 \langle g, \Lambda^{1/2} \theta\rangle],
 \end{align*}
 where $z,g$ are independent vectors with entries i.i.d.~$\mathcal{N}(0,1)$. Maximizing over $\lambda$, this further reduces to 
 
 \begin{equation}\label{eq:reduction}
 \min_{\|\theta \|_{1} \leq \sqrt{p}} [\frac{1}{\sqrt{p}} \langle g, \Lambda^{1/2} \theta\rangle +  \| \nu_{+} \|_2],
 \end{equation}
 
 where $\nu = (\kappa -\langle \theta_{\star},\theta \rangle)\boldsymbol{1} -\langle \tilde{\theta},\theta\rangle m -\|\Lambda^{1/2}\theta \|_2 z$. Define

 \[\hat{F}_{\kappa}(c_1,c_2,c_3) = \sqrt{\hat{\mathbb{E}}_n[(\kappa - c_1 - c_2\tilde{Z} - c_3 M)_{+}^2]}, \]
where $M,\tilde{Z}$ denote random vectors with distribution $M_i \stackrel{\mathrm{d}}{=} m_i$ and $\tilde{Z}_i \sim \mathcal{N}(0,1)$, all entries i.i.d. with $M, Z $ independent of each other and $\hat{E}_n$ denoting the corresponding empirical distribution. With this notation, \eqref{eq:reduction} simplifies to 

 \begin{equation}\label{eq:finitedim}
 \min_{\|\theta \|_1 \leq \sqrt{p}} [\frac{1}{p} \langle g, \Lambda^{1/2} \theta \rangle + \psi^{-1/2} \hat{F}_{\kappa} (\langle \theta_{\star},\theta\rangle, \langle \tilde{\theta},\theta\rangle, \|\Lambda^{1/2} \theta \|_2)].
 \end{equation}
 Using tricks similar to those in Proposition \eqref{prop:large-n-p-limit}, we have that $\xi_n$ must converge to the following  infinite-dimensional version 
 
 \begin{equation}\label{eq:limitgmm}
 \xi_{\infty} :=  \min_{\|h \|_{L_1(Q)} \leq 1} [\langle G, \Lambda^{1/2} h\rangle_{L_2(Q)} + \psi^{-1/2} F_{\kappa} ( \langle h_{\star},h\rangle, \langle \tilde{h},h \rangle, \|\Lambda^{1/2} h \|_{L_2(Q)})  ],
 \end{equation}
 
where $h_{\star}, \tilde{h}$ correspond to $\theta_{\star}, \tilde{\theta}$ respectively. Here we use the same trick as in \eqref{eq:infinite} and go over to the space $\{ h : \Reals^4 \rightarrow \Reals, h \in \mathcal{L}^2(\cQ) \}$, where $\cQ= \mu \otimes \mathcal{N}(0,1)$ with $\mu $ given as follows: the empirical probability distributions $\sum_{j=1}^p \delta_{(\lambda_i,\sqrt{p}\theta_{\star}^{\top}e_i,\sqrt{p}\tilde{\theta}^{\top}e_i)}/p \stackrel{W_2}{\implies} \mu$. 
To rigorize these arguments, we assume that the data is in the asymptotically linearly separable regime. Note that the exact threshold for separability here will be different from $\psi^{\star}$ since the data-generating scheme is different in this context. With $F_{\kappa}$ denoting $F_\kappa (c_1, c_2, c_3)$, the limiting version of $\hat{F}_{\kappa}$, where $c_1= \langle h_{\star},h\rangle, c_2=  \langle \tilde{h},h \rangle, c_3 =  \|\Lambda^{1/2} h \|_{L_2(Q)}) $, the KKT conditions corresponding to $\xi_{\infty}$ can then be characterized as follows,

\[ \Lambda^{1/2} G + \psi^{-1/2}[\partial_1F_{\kappa} h_{\star} + \partial_2 F_{\kappa} \tilde{h} + \Lambda^{1/2} \partial_3 F_{\kappa} Z] + s \partial\|h \|_{L_1(Q)} = 0 ,\]

where $Z = \Lambda^{1/2}h/\| \Lambda^{1/2} h\| $ if the denominator is strictly positive, and $Z'$ with $\| Z'\| \leq 1$ when the denominator is zero.  Rewriting things, we obtain 

\[ \Lambda^{1/2} G + \psi^{-1/2}[\partial_1F_{\kappa} h_{\star} + \partial_2 F_{\kappa} \tilde{h} ]+\psi^{-1/2} \Lambda c_3^{-1}\partial_3 F_{\kappa} h + s \partial\|h \|_{L_1(Q)} = 0.\]

Using properties of the proximal mapping operator, this yields

\begin{equation}\label{eq:hsolonecomponent}
h_{\mathrm{sol}}  = - \frac{\mathrm{prox_s} (\Lambda^{1/2}G +\psi^{-1/2} (\partial_1 F_{\kappa} h_{\star} + \partial_{2}F_{\kappa}\tilde{h}))}{\Lambda \psi^{-1/2}c_3^{-1}\partial_3 F_{\kappa}}. 
\end{equation}

Assuming that $(\Lambda, h_{\star},\tilde{h},G) \sim Q= \mu \otimes \mathcal{N}(0,1)$, the system of equations governing the behavior of the max-$\ell_1$-margin and min-$\ell_1$-interpolant is then given by 
\begin{align}\label{eq:syseqlatentone}
c_1 & = \mathbb{E}_{(\Lambda,h_{\star},\tilde{h},G) \sim Q} h_\star h_{\mathrm{sol}}  \nonumber \\
c_2 & = \mathbb{E}_{(\Lambda,h_{\star},\tilde{h},G) \sim Q} \tilde{h} h_{\mathrm{sol}} \nonumber \\
c_3^2 & = \mathbb{E}_{(\Lambda,h_{\star},\tilde{h},G) \sim Q}(\Lambda^{1/2} h_{\mathrm{sol}} )^2 \nonumber \\ 
1 & = \mathbb{E}_{(\Lambda,h_{\star},\tilde{h},G) \sim Q} |h_{\mathrm{sol}}| 
\end{align}
For the formal argument that $\xi_n \rightarrow \xi_\infty$ as $n,p \rightarrow \infty$, we assume that the aforementioned equation system admits a unique solution. We expect that arguments similar to Proposition \eqref{prop:uniqueness} can be used to prove this in the regime where the data is asymptotically linearly separable. 

The aforementioned arguments for the model in \eqref{eq:Xdist} naturally extend to the following, 
\begin{equation}\label{eq:l-component}
x_i = y_i \theta_{\star} + \sum_{c=1}^\ell m_{i,c} \tilde{\theta}_{c} +\tilde{x}_i,
\end{equation}
where $\tilde{x}_i \sim \mathcal{N}(0,\Lambda), \Lambda \,\, \text{diagonal}, \,\, y \odot (m_{i,1},\hdots,m_{i,l}) \stackrel{\mathrm{d}}{=}  (m_{i,1},\hdots,m_{i,l})$ for all $i$ and $y_i$'s, $m_i$'s, $\tilde{x}_i$'s are independent. Note that, in this data generation scheme, the covariance between features is a rank $\ell$ perturbation of a diagonal.
Define $\mu_{\ell}$ to  be the probability distribution given by the limit $\sum_{j=1}^p \delta_{(\lambda_i,\sqrt{p}\theta_{\star}^{\top}e_i,\sqrt{p}\tilde{\theta}^{\top}_{1}e_i,\sqrt{p}\tilde{\theta}^{\top}_{2}e_i, \hdots, \sqrt{p}\tilde{\theta}^{\top}_{\ell}e_i)} \stackrel{W_2}{\implies} \mu_\ell$, and let $(\Lambda, h_{\star},\tilde{h}_1, \tilde{h}_2, \hdots, \tilde{h}_\ell,G) \sim Q_\ell= \mu_\ell \otimes \mathcal{N}(0,1)$. Then the system of equations governing the behavior of the max-$\ell_1$-margin and min-$\ell_1$-interpolant is given by 

\begin{align*}
c_1 & = \mathbb{E}_{(\Lambda,h_{\star},\tilde{h}_1,\hdots,\tilde{h}_\ell,G) \sim Q_{\ell}} h_\star h_{\mathrm{sol},\ell}  \nonumber \\
c_{i+1} &= \mathbb{E}_{(\Lambda,h_{\star},\tilde{h}_1,\hdots,\tilde{h}_\ell,G) \sim Q_{\ell}} \tilde{h}_i h_{\mathrm{sol},\ell}, ~~i = 1, \ldots, \ell,\\
c_{\ell+2}^2 & = \mathbb{E}_{(\Lambda,h_{\star},\tilde{h}_1,\hdots, \tilde{h}_\ell, G) \sim Q}(\Lambda^{1/2} h_{\mathrm{sol},\ell} )^2 \nonumber \\ 
1 & = \mathbb{E}_{(\Lambda,h_{\star},\tilde{h}_1,\hdots, \tilde{h}_{\ell},G) \sim Q} |h_{\mathrm{sol}}|,
\end{align*}

where $h_{\mathrm{sol},\ell}$ is defined as follows

\[h_{\mathrm{sol},\ell} :=  - \frac{\mathrm{prox_s} (\Lambda^{1/2}G +\psi^{-1/2} (\partial_1 F_{\kappa} h_{\star} + \partial_{2}F_{\kappa}\tilde{h}_1+ \partial_{3} F_{\kappa} \tilde{h}_2+ \hdots + \partial_{\ell+1}F_{\kappa} \tilde{h}_\ell))}{\Lambda \psi^{-1/2}c_{\ell+2}^{-1}\partial_{\ell+2} F_{\kappa}}\]

 and $F_{\kappa}$ equals 

\[F_{\kappa}(c_1,\hdots, c_{\ell+2}) = \sqrt{\mathbb{E}(\kappa - c_1 -c_2 M_1-c_3 M_2 -\hdots c_{\ell+1}M_{\ell}-c_{\ell+2} \tilde{Z})_{+}^2};\]

 here $\tilde{Z} \sim \mathcal{N}(0,1)$, independently of $(M_1, \hdots, M_\ell)$, which has the same distribution as $(m_{i,1},\hdots, m_{i,\ell}). $

\subsection{Derivations in Section~\ref{sec:universality}}\label{app:universality}

\begin{proof}[Proof of Theorem~\ref{thm:universalitythm}]
 To overcome the difficulty of non-differentiability (due to $\ell_1$) and the non-strongly convexity of our problem, we need to introduce a Gaussian smoothing technique and an extra $\ell_2$ regularization term. To start, 
 define the ReLU function
\begin{align}
	h(t) = \max(t, 0)
\end{align}
and the smoothed version of the ReLU
\begin{align}
	\label{eqn:smoothed-hinge}
	h_\delta(t) := \E_{\bg \sim \cN(0, 1)}[ h(t+\delta \bg)]
\end{align}
For the purposes of this section, we denote $x_i$ (resp.~$\tilde{x}_i$) to mean the random features $a_i$ (resp.~$b_i$), where $a_i, b_i$ are as defined in Section \ref{sec:universality}.

To prove \eqref{eq:sufficient}, we define for fixed $\kappa, \lambda >0$, the following perturbed Lagrangian that is strongly convex in $\theta$
\begin{align}
	&\cL_k^{\kappa, \lambda}(\theta ~;~ \epsilon, \delta ) :=\nonumber \\
	&  \sum_{i=1}^{k} h_\delta^2(\kappa - y_i \tfrac{1}{\sqrt{p}} x_i^\top \theta ) + \sum_{i=k+1}^{n} h_\delta^2(\kappa - y_i \tfrac{1}{\sqrt{p}} \tilde x_i^\top \theta ) + \lambda \sum_{j=1}^p \big( h_\delta(\theta_j) + h_\delta(-\theta_j) - 1 \big) + \epsilon \sum_{j=1}^p \tfrac{1}{2}\theta_j^2 \nonumber \\
&	\cL_{\backslash k}^{\kappa, \lambda}(\theta ~;~  \epsilon, \delta ) := \nonumber \\
	& \sum_{i=1}^{k-1} h_\delta^2(\kappa - y_i \tfrac{1}{\sqrt{p}} x_i^\top \theta ) + \sum_{i=k+1}^{n} h_\delta^2(\kappa - y_i \tfrac{1}{\sqrt{p}} \tilde x_i^\top \theta ) + \lambda \sum_{j=1}^p \big( h_\delta(\theta_j) + h_\delta(-\theta_j) - 1 \big) + \epsilon \sum_{j=1}^p \tfrac{1}{2} \theta_j^2.
\end{align}
Further, define
\begin{align}
	\Phi_k^{\kappa, \lambda}(\epsilon, \delta ) :=  \min_{\theta \in \Reals^p}~ \cL_k^{\kappa, \lambda}(\theta ~;~  \epsilon, \delta ) \nonumber \\
	\Phi_{\backslash k}^{\kappa, \lambda}(\epsilon, \delta) :=  \min_{\theta \in \Reals^p}~ \cL_{\backslash k}^{\kappa, \lambda}(\theta  ~;~ \epsilon, \delta ) \nonumber\\
	\theta^\star_{\backslash k} := \argmin_{\theta \in \Reals^p}~\cL_{\backslash k}^{\kappa, \lambda}(\theta  ~;~ \epsilon, \delta ) 
\end{align}
and finally the leave-one-out Hessian
\begin{align}
	H_{\backslash k}(\epsilon) := \nabla^2_{\theta} \cL_{\backslash k}(\theta ~;~ \lambda ~;~ \epsilon, \delta ) ~|_{\theta = \theta^\star_{\backslash k}} \succeq \epsilon \bI_p
\end{align}
with the full expression
\begin{align}
	H_{\backslash k}(\epsilon) & = \tfrac{2}{p} \sum_{i=1}^{k-1} h_\delta(\kappa - y_i \tfrac{1}{\sqrt{p}} x_i^\top \theta^\star_{\backslash k}) h'_\delta(\kappa - y_i \tfrac{1}{\sqrt{p}} x_i^\top \theta^\star_{\backslash k}) \cdot  x_i x_i^\top \nonumber \\
	&+ \tfrac{2}{p} \sum_{i=k+1}^{n} h_\delta(\kappa - y_i \tfrac{1}{\sqrt{p}} \tilde x_i^\top \theta^\star_{\backslash k}) h'_\delta(\kappa - y_i \tfrac{1}{\sqrt{p}} \tilde x_i^\top \theta^\star_{\backslash k}) \cdot  \tilde x_i \tilde x_i^\top \nonumber \\
	&+ \lambda \text{diag}\left\{  h''_{\delta}(\theta^\star_{\backslash k,j}) + h''_{\delta}(-\theta^\star_{\backslash k,j}) \right\} +  \epsilon  I_p \;.
\end{align}
Our goal is to show $ \Phi_n(A,\lambda) - \Phi_n(B,\lambda)  \stackrel{\mathbb{P}}{\rightarrow}  0$, where these are defined in Section \ref{sec:universality}. In our notation here, this reduces to establishing that for all $\lambda > 0,$
$$ \tfrac{1}{p} \Phi_n^{\kappa, \lambda}(0, 0) - \tfrac{1}{p} \Phi_0^{\kappa, \lambda}(0, 0) \stackrel{\mathbb{P}}{\rightarrow} 0.$$
By a standard probability argument (see for instance \cite{hu2020universality}), it suffices to show that $ \E\left[ \phi\big(\tfrac{1}{p} \Phi_n^{\kappa, \lambda}(0, 0) \big) \right] - \E\left[ \phi\big(\tfrac{1}{p} \Phi_0^{\kappa, \lambda}(0, 0) \big) \right] \rightarrow 0$ for any bounded test function $\phi$ that has bounded derivatives up to the third order.  
To bound this difference, we will approximate the problems at $(0,0)$, that is, $ \Phi_n^{\kappa, \lambda}(0, 0), \Phi_0^{\kappa, \lambda}(0, 0)$ with the corresponding problems for positive $\epsilon, \delta$. To control this approximation, we further need to control the approximation error of $h(\cdot)$, using $h_{\delta}(\cdot)$, and the derivatives of $h_{\delta}(\cdot)$. This is achieved in Lemma~\ref{lem:smoothing}. On working out this argument, we obtain that 
\begin{align}\label{eq:approximation}
	& \left| \E\left[ \phi\big(\tfrac{1}{p} \Phi_n^{\kappa, \lambda}(0, 0) \big) \right] - \E\left[ \phi\big(\tfrac{1}{p} \Phi_0^{\kappa, \lambda}(0, 0) \big) \right] \right|  \nonumber \\
	& \leq \left| \E\left[ \phi\big(\tfrac{1}{p} \Phi_n^{\kappa, \lambda}(\epsilon, \delta) \big) \right] - \E\left[ \phi\big(\tfrac{1}{p} \Phi_0^{\kappa, \lambda}(\epsilon, \delta) \big) \right] \right| + C \cdot \| \phi' \|_{L^\infty} (\delta + \epsilon), \quad \text{Lemma~\ref{lem:smoothing}}  \nonumber \\
	& \leq \sum_{k=0}^{n-1} \underbrace{\left| \E\left[ \phi\big(\tfrac{1}{p} \Phi_{k+1}^{\kappa, \lambda}(\epsilon, \delta) \big) \right] - \E\left[ \phi\big(\tfrac{1}{p} \Phi_{k}^{\kappa, \lambda}(\epsilon, \delta) \big) \right] \right|}_{(i)} +  C \cdot \| \phi' \|_{L^\infty} (\delta + \epsilon).
\end{align}
Above $C$ involves universal constants and the scaled norms $\| \theta^{\star}_{n}\|^2/p, \|\theta^{\star}_0 \|^2/p$, where these denote optimizers of the objective functions in $ \Phi_n^{\kappa, \lambda}(0, 0),  \Phi_0^{\kappa, \lambda}(0, 0)$. Thus, it suffices to control $(i)$, which we achieve by a Lindeberg argument. Denoting $\E_x$ to be the expectation with respect to $x$, keeping all other random variables fixed, we note that

\begin{align}\label{eq:ubone}
	(i) & \leq \| \phi' \|_{L^\infty}  \tfrac{1}{p} \E\left| \E_{x_k} [\Phi_{k}^{\kappa, \lambda}(\epsilon, \delta)] - \E_{\tilde x_k}  [\Phi_{k-1}^{\kappa, \lambda}(\epsilon, \delta) ] \right|. 
\end{align}

Define $\ell_{\delta}(x,y) = h_{\delta}^2(\kappa - y.x)$. Then, as in \cite[Eqn.~35]{hu2020universality},  we can define the following quadratic approximation to the leave-one-out problem $\Phi_{\backslash k}^{\kappa, \lambda}(\epsilon, \delta)$

\begin{align}
	\Psi_k(x) := \Phi_{\backslash k}^{\kappa, \lambda}(\epsilon, \delta) + \min_{\theta}~ \left\{ \frac{1}{2} (\theta - \theta^\star_{\backslash k})^\top H_{\backslash k}(\epsilon) (\theta - \theta^\star_{\backslash k}) + \ell_{\delta}(\tfrac{1}{\sqrt{p}} x^\top \theta, y_k) \right\}.
\end{align}

With this notation, bounding the RHS of \eqref{eq:ubone} breaks down to two tasks---controlling the error of quadratic approximation
\begin{align}
	(ii) :=  \max \big\{ ~|\Phi_{k}^{\kappa, \lambda}(\epsilon, \delta) - \Psi_k(x_k)|, ~|\Phi_{k-1}^{\kappa, \lambda}(\epsilon, \delta) - \Psi_k(\tilde x_k)| ~\big\}
\end{align}
and the error term
\begin{align}
	(iii) & := \big|\E_{x_k} [\Psi_k(x_k)] - \E_{\tilde x_k} [\Psi_k(\tilde x_k)] \big| 
	\end{align}
	
Define the Moreau envelope to be
\begin{align}
	\cM_k(t, \gamma_k) := \min_{s \in \Reals}~\big\{  \ell_\delta(s, y_k) + \frac{(t-s)^2}{2\gamma_k} \big\},
\end{align}
where the regularization parameter is defined to be
\begin{align}
	\gamma_k = \tfrac{1}{p} \E_{\bx} \big[\bx^\top (H_{\backslash k}(\epsilon))^{-1} \bx \big] \leq \epsilon^{-1} \;.
\end{align}
Due to the matching second moment of $x_k$ and $\tilde{x}_k$, the above $\gamma_k$ stays the same when $\bx$ is either $x_k$ or $\tilde x_k$ in distribution.
Then, $(iii)$
 can be upper bounded by
 \begin{align}
	\underbrace{\big|\E_{x_k} [\cM_k(\tfrac{1}{\sqrt{p}} x_k^\top \theta^\star_{\backslash k}, \gamma_k)] - \E_{\tilde x_k} [\cM_k(\tfrac{1}{\sqrt{p}} \tilde x_k^\top \theta^\star_{\backslash k}, \gamma_k)] \big|}_{(iv)} ~+ ~ (v),
\end{align}
with
\begin{align}
	(v)&:= \quad \big|\E_{x_k} [\cM_k(\tfrac{1}{\sqrt{p}} x_k^\top \theta^\star_{\backslash k}, \gamma_k) - \cM_k(\tfrac{1}{\sqrt{p}} x_k^\top \theta^\star_{\backslash k}, \gamma(x_k) ) ] \big|  \nonumber \\
	& + \big|\E_{\tilde x_k} [\cM_k(\tfrac{1}{\sqrt{p}} \tilde x_k^\top \theta^\star_{\backslash k}, \gamma_k) - \cM_k(\tfrac{1}{\sqrt{p}} \tilde x_k^\top \theta^\star_{\backslash k}, \gamma(\tilde x_k) ) ] \big| \\
	&\text{where}~~\gamma(x) := \tfrac{1}{p} x^\top (H_{\backslash k}(\epsilon))^{-1} x. \nonumber
\end{align}
Thus, it suffices to bound $(ii), (iv)$ and $(v).$ This requires controlling the approximation error of $h(\cdot)$ using $h_\delta(\cdot)$, derivatives of $h_{\delta}(\cdot)$ and the Moreau envelope. We achieve these in Lemma~\ref{lem:moreau-envelope}-\ref{lem:smoothing}, and using these, we claim the following bounds,
\begin{align}
	(ii) & \leq \tfrac{\text{poly}(\epsilon^{-1}, \delta^{-1})}{\sqrt{p}} \text{polylog}(p)  \;, \label{eq:term1} \\
	(iv) & \leq \tfrac{\text{poly}(\epsilon^{-1})}{\sqrt{p}} \textrm{polylog}(p)\;, \label{eq:term2}\\
	(v) & \leq  \tfrac{\epsilon^{-2}}{\sqrt{p}}  \textrm{polylog}(p)\;, \label{eq:term3}
\end{align}
We will prove these invoking Lemma~\ref{lem:smoothing}-\ref{lem:moreau-envelope}  and techniques from \cite[Lemma1,2,24]{hu2020universality}. Before we present the proofs, note that, together with \eqref{eq:approximation}, this implies that with proper choice of $\epsilon, \delta  = p^{-c_1}$
\begin{align}
	\left| \E\left[ \phi\big(\tfrac{1}{p} \Phi_n^{\kappa, \lambda}(0, 0) \big) \right] - \E\left[ \phi\big(\tfrac{1}{p} \Phi_0^{\kappa, \lambda}(0, 0) \big) \right] \right| \precsim p^{-c_2},
\end{align}
with some $c_1, c_2>0$. The above is true since it can be shown that the constant $C$ in \eqref{eq:approximation} is $O(1)$, by arguments similar to \eqref{eqn:l2-bound}. The rest of the proof thus focuses on establishing \eqref{eq:term1}--\eqref{eq:term3}.

\textbf{Proof of Eqn. \eqref{eq:term3}}
\begin{proof}[Proof of Eqn. \eqref{eq:term3}]
	To bound the term (v), it suffices to control
	$
		\big|\E_{x_k} [ \cM_k(t, \gamma(x_k) )  - \cM_k( t, \gamma_k)  ] \big|
	$
	with $t= \tfrac{1}{\sqrt{p}} x_k^\top \theta^\star_{\backslash k}$. First, calculate the partial derivative of $\cM_k( t, \gamma)$ w.r.t. $\gamma$, 
	\begin{align}
		| \frac{\partial}{\partial \gamma} \cM_k( t, \gamma) | = \frac{1}{2} | \ell'_\delta( s^\star, y_k)|^2
	\end{align}
	where $s^\star$ satisfies $\ell'_\delta(s^\star, y_k)+ \frac{s^\star -t}{\gamma} = 0$ (for fixed $t, \gamma$). By \eqref{eqn:first-order-estimate-proximal}, the following upper bound holds
	\begin{align}
		| \frac{\partial}{\partial \gamma} \cM_k( t, \gamma) | \leq 2 ( 2\kappa + \delta + |t|)^2 \;.
	\end{align}
	With the above, we know
	\begin{align}
		\big|\E_{x_k} [ \cM_k(t, \gamma(x_k) )  - \cM_k( t, \gamma_k)  ] \big| & = \big| \E_{x_k} \big[ \frac{\partial}{\partial \gamma} \cM_k( t,  \tilde \gamma) ( \gamma(x_k) - \gamma_k ) \big] \big| \\
		& \leq \sqrt{ \E_{x_k} \big| \frac{\partial}{\partial \gamma} \cM_k( t,  \tilde \gamma) \big|^2  \cdot \E_{x_k}  ( \gamma(x_k) - \gamma_k )^2 } \\
		& \precsim \sqrt{  \big[ \kappa^4+\delta^4+ \E_{x_k} (\tfrac{1}{\sqrt{p}} x_k^\top \theta^\star_{\backslash k})^4 \big] \cdot \E_{x_k} ( \frac{1}{p} x_k^\top (H_{\backslash k}(\epsilon))^{-1} x_k - \gamma_k )^2 } \\
		& \precsim \sqrt{  \big[ \kappa^4+\delta^4 + \| \tfrac{1}{\sqrt{p}}  \theta^\star_{\backslash k} \|^4 \big] \cdot \frac{\epsilon^{-2}}{p}}
	\end{align}
	and hence it suffices to bound $\| \tfrac{1}{\sqrt{p}}  \theta^\star_{\backslash k} \|$. Note that, in the definition of $\gamma_k$ above, $x$ has the same distribution as $x_k$ so that the term $E_{x_k} ( \frac{1}{p} x_k^\top (H_{\backslash k}(\epsilon))^{-1} x_k - \gamma_k )^2$ can effectively be treated as variance of a chi-square random variable with $p$ degrees of freedom, scaled by $p$. We know
	\begin{align}
		\label{eqn:l2-bound}
		\frac{\epsilon}{2} \|\theta^\star_{\backslash k}  \|^2 \leq \cL_{\backslash k}^{\kappa, \lambda}(\theta^\star_{\backslash k}  ~;~  \epsilon, \delta ) + p\lambda \leq \cL_{\backslash k}^{\kappa, \lambda}(0  ~;~  \epsilon, \delta ) + p\lambda \leq n (\kappa+\delta)^2 + p\lambda \;.
	\end{align}
	Putting things together, we have
	\begin{align}
		\big|\E_{x_k} [ \cM_k(t, \gamma(x_k) )  - \cM_k( t, \gamma_k)  ] \big|  \precsim \epsilon^{-2} \frac{1}{\sqrt{p}}
	\end{align}
\end{proof}

\textbf{Proof of Eqn. \eqref{eq:term2}}
\begin{proof}[Proof of Eqn. \eqref{eq:term2}]
	The proof follows directly from Lemma~\ref{lem:moreau-envelope} and \cite[Lemma 2]{hu2020universality} since Lemma~\ref{lem:moreau-envelope} verifies the needed condition needed in  \cite[Lemma 2]{hu2020universality}.
\end{proof}

\textbf{Proof of Eqn. \eqref{eq:term1}}
\begin{proof}[Proof of Eqn. \eqref{eq:term1}]
	
	Define the following three minimizers 
	\begin{align}
		\theta^\star(x_k) &= \argmin_{\theta} \big\{ \cL_{\backslash k}^{\kappa, \lambda}(\theta ~;~  \epsilon, \delta ) + h_\delta^2(\kappa - y_k \tfrac{1}{\sqrt{p}} x_k^\top \theta) \big\} \\
		\tilde \theta(x_k) &= \argmin_{\theta} \big\{ \frac{1}{2} (\theta - \theta^\star_{\backslash k})^\top H_{\backslash k}(\epsilon) (\theta - \theta^\star_{\backslash k}) + h_\delta^2(\kappa - y_k \tfrac{1}{\sqrt{p}} x_k^\top \theta) \big\} \\
		 \theta^\star_{\backslash k} & = \argmin_{\theta} \cL_{\backslash k}^{\kappa, \lambda}(\theta ~;~  \epsilon, \delta ) 
	\end{align}
	
	To upper bound (ii), we need to control the quadratic approximation
	\begin{align}
		& \big| \Phi_{k}^{\kappa, \lambda}(\epsilon, \delta) - \Psi_{k}(x_k) \big|  \nonumber \\
		 & \leq \max_{\theta \in \{\theta^\star(x_k), \tilde \theta(x_k) \}} \big\{ \cL_{\backslash k}^{\kappa, \lambda}(\theta ~;~  \epsilon, \delta ) -  \cL_{\backslash k}^{\kappa, \lambda}(\theta_{\backslash k}^\star ~;~  \epsilon, \delta ) - \frac{1}{2} (\theta - \theta^\star_{\backslash k})^\top H_{\backslash k}(\epsilon) (\theta - \theta^\star_{\backslash k})  \big\} 
	\end{align}
	
	By a Taylor expansion up to the third order and the mean value theorem, the above expression can be bounded by 
	\begin{align}
		\label{eqn:decomposition}
		\frac{1}{6} n \cdot \max_{i\neq k} \underbrace{|2 h'''_\delta(t) h_\delta(t) + 6 h''_\delta(t) h'_\delta(t) |_{t =\kappa - y_i \frac{1}{\sqrt{p}}x_i^\top \theta_{m}}}_{(b)}  \cdot \underbrace{\big| \tfrac{1}{\sqrt{p}} x_i^\top (\theta - \theta^\star_{\backslash k}) \big|^3}_{(a)} \nonumber \\
		 + \frac{1}{6} \sum_{j=1}^p \lambda \underbrace{|h_{\delta}'''(t) - h_{\delta}'''(-t)|_{t = \theta_{m}[j]} \cdot \big|\theta[j] - \theta^\star_{\backslash k}[j] \big|^3}_{(c)},
	\end{align}
	where $\theta_m = (1-\xi) \theta^\star_{\backslash k} + \xi \theta$ is an intermediate point. Here when we take $\max_{i \neq k}$, we slightly abuse the notation: for $i=1, \ldots, k-1$, (a) and (b) is as stated with $x_i$ ; for $i=k+1, \ldots, n$, (a) and (b) should have $\tilde x_i$ substituting $x_i$. In the rest of the proof, when we control (a) and (b), the proof follows the same say with either $\tilde x_i$ or $x_i$.
	
	\textbf{Case 1: $\theta = \tilde \theta(x_k)$.}
	To handle the case of $\theta = \tilde \theta(x_k)$, note that
	\begin{align}
		\label{eqn:moreau-first-order}
		\tilde \theta(x_k) - \theta^\star_{\backslash k} =  2 \big[ h_\delta'(t) h_\delta(t) \big]_{t = \kappa - y_k \frac{1}{\sqrt{p}}x_k^\top  \tilde \theta(x_k)}  [H_{\backslash k}(\epsilon)]^{-1} \tfrac{1}{\sqrt{p}} y_k x_k \;.
	\end{align}
	With this fact in mind, we continue to control each term (a), (b), (c).
	
	\textbf{Term (a)}:
	\begin{align}
		& \quad \big| \tfrac{1}{\sqrt{p}} x_i^\top (\tilde \theta(x_k) - \theta^\star_{\backslash k}) \big|  \\
		& = 2 | h_\delta'(t) h_\delta(t)  |_{t = \kappa - y_k \frac{1}{\sqrt{p}}x_k^\top  \tilde \theta(x_k)} \cdot  \big| \frac{1}{p} x_i^\top [H_{\backslash k}(\epsilon)]^{-1} x_k \big|  \quad \text{by \eqref{eqn:moreau-first-order}} \\
		& \precsim | h_\delta(t) |_{t = \kappa - y_k \frac{1}{\sqrt{p}}x_k^\top  \theta_{\backslash k}^\star} \tfrac{\epsilon^{-1} \text{polylog}(p)}{p^{0.5}} \\
		& \precsim \tfrac{\epsilon^{-1.5} \text{polylog}(p)}{p^{0.5}}  \label{eqn:term-a}
	\end{align}
	where the second to last step uses two facts (1) $|h_\delta'(t)| \leq 1$ and $h^2_\delta(\kappa - y_k \frac{1}{\sqrt{p}}x_k^\top  \theta_{\backslash k}^\star) \geq h^2_\delta(\kappa - y_k \frac{1}{\sqrt{p}}x_k^\top  \tilde \theta(x_k)) + \frac{1}{2} (\tilde \theta(x_k) -  \theta_{\backslash k}^\star)^\top H_{\backslash k}(\epsilon) (\tilde \theta(x_k) -  \theta_{\backslash k}^\star) \geq h^2_\delta(\kappa - y_k \frac{1}{\sqrt{p}}x_k^\top  \tilde \theta(x_k)) $,  (2) \cite[Lemma 10]{hu2020universality}.
Recalling the estimate on $\tfrac{1}{p} \|  \theta_{\backslash k}^\star \|^2 \precsim \epsilon^{-1}$ as in \eqref{eqn:l2-bound}, we obtain the last step.
	
	\textbf{Term (b)}:
	\begin{align}
		&\quad |2 h'''_\delta(t) h_\delta(t) + 6 h''_\delta(t) h'_\delta(t) |_{t =\kappa - y_i \frac{1}{\sqrt{p}}x_i^\top \theta_{m}} \nonumber \\
		& \precsim \delta^{-2} (|t| + \delta) \quad \text{here $t =\kappa - y_i \tfrac{1}{\sqrt{p}}x_i^\top \theta_{m}$, by Lemma~\ref{lem:smoothing}} \nonumber\\
		& \precsim  \delta^{-2} \left\{ \kappa + \max \big\{ |\tfrac{1}{\sqrt{p}}x_i^\top \theta^\star_{\backslash k} |, |\tfrac{1}{\sqrt{p}}x_i^\top \tilde \theta(x_k) | \big\} \right\} \nonumber\\
		& \precsim \delta^{-2} \left\{   |\tfrac{1}{\sqrt{p}}x_i^\top \theta^\star_{\backslash k} | +  \big| \tfrac{1}{\sqrt{p}} x_i^\top (\tilde \theta(x_k) - \theta^\star_{\backslash k}) \big| \right\} \nonumber\\
		& \precsim \delta^{-2} \left\{   |\tfrac{1}{\sqrt{p}}x_i^\top \theta^\star_{\backslash k} | +  \tfrac{\epsilon^{-1.5} \text{polylog}(p)}{p^{1/2}}  \right\} \quad \text{by \eqref{eqn:term-a}} \nonumber\\
		& \precsim \delta^{-2} \left\{   |\tfrac{1}{\sqrt{p}}x_i^\top \theta^\star_{\backslash \{k,i\}} | + |\tfrac{1}{\sqrt{p}}x_i^\top ( \theta^\star_{\backslash k} - \theta^\star_{\backslash \{k,i\}} ) | +  \tfrac{\epsilon^{-1.5} \text{polylog}(p)}{p^{1/2}}  \right\} \nonumber\\
		& \precsim \delta^{-2} \left\{  \epsilon^{-0.5} \vee \epsilon^{-1}\textrm{polylog}(p) +  \tfrac{\epsilon^{-1.5} \text{polylog}(p)}{p^{1/2}}  \right\}  \;.
	\end{align}
	The last two steps use the following fact:
	\begin{align}
		 &\quad \cL_{\backslash \{k,i\}}^{\kappa, \lambda}(\theta^\star_{\backslash \{k,i\}} ~;~  \epsilon, \delta )  + h_\delta^2(\kappa - y_i \tfrac{1}{\sqrt{p}} x_i^\top \theta^\star_{\backslash \{k,i\}} ) \nonumber\\
		 &\geq \cL_{\backslash \{k,i\}}^{\kappa, \lambda}( \theta^\star_{\backslash k} ~;~  \epsilon, \delta )  + h_\delta^2(\kappa - y_i \tfrac{1}{\sqrt{p}} x_i^\top  \theta^\star_{\backslash k} ) \nonumber \\
		 & \geq \cL_{\backslash \{k,i\}}^{\kappa, \lambda}(\theta^\star_{\backslash \{k,i\}} ~;~  \epsilon, \delta )  + \frac{\epsilon}{2} \| \theta^\star_{\backslash k} -  \theta^\star_{\backslash \{k,i\}} \|^2 + h_\delta^2(\kappa - y_i \tfrac{1}{\sqrt{p}} x_i^\top  \theta^\star_{\backslash k} )  \quad \text{by strong convexity,}
	\end{align}
	and hence we have
	\begin{align}
		\frac{\epsilon}{2} \| \theta^\star_{\backslash k} -  \theta^\star_{\backslash \{k,i\}} \|^2 \leq h_\delta^2(\kappa - y_i \tfrac{1}{\sqrt{p}} x_i^\top \theta^\star_{\backslash \{k,i\}} ) \leq \big( \kappa + \delta + \tfrac{1}{\sqrt{p}} \|\theta^\star_{\backslash \{k,i\}} \| \cdot \text{polylog}(p) \big)^2 \nonumber\\
		\| \theta^\star_{\backslash k} -  \theta^\star_{\backslash \{k,i\}} \|^2 \leq \epsilon^{-2} \text{polylog}(p) 
	\end{align}
	as $\tfrac{1}{p} \|  \theta_{\backslash \{k, i\}}^\star \|^2 \precsim \epsilon^{-1}$.

	\textbf{Term (c)}: First, observe that by Lemma~\ref{lem:smoothing}, $h'''_\delta \precsim \delta^{-2}$, thus we only need to control 
	\begin{align}
		&\quad \delta^{-2} \sum_{j=1}^p  \big|\theta[j] - \theta^\star_{\backslash k}[j] \big|^3 \nonumber\\
		&\precsim \delta^{-2}  \big| h_\delta'(t) h_\delta(t) \big|^3_{t = \kappa - y_k \frac{1}{\sqrt{p}}x_k^\top  \tilde \theta(x_k)} \sum_{j=1}^p \big| h_j \tfrac{1}{\sqrt{p}} x_{k} \big|^3  \quad \text{where $h_j \in \Reals^n$ is the $j-$th column of $[H_{\backslash k}(\epsilon)]^{-1}$} \nonumber\\
		&\precsim \delta^{-2} \epsilon^{-1.5} p  \big(\tfrac{\| h_j\| }{p^{0.5}} \big)^3\text{polylog}p  \quad \text{since $\| h_j\| \leq \epsilon^{-1}$} \label{eq:intermstep}\nonumber\\
		& \precsim  \tfrac{\delta^{-2} \epsilon^{-4.5} \text{polylog}(p)}{p^{0.5}} \;.
	\end{align} 
	Putting the upper bounds on (a), (b) and (c) together, we effectively have \eqref{eqn:decomposition} with $\theta = \tilde \theta(x_k)$ is upper bounded by
	\begin{align}
			\eqref{eqn:decomposition}  & \precsim \tfrac{\delta^{-2} \epsilon^{-5.5}}{p^{0.5}} \text{polylog}(p) \big(1 \vee \tfrac{\epsilon^{-1} \text{polylog}(p)}{p^{0.5}} \big) + \tfrac{\delta^{-2} \epsilon^{-4.5}}{p^{0.5}}  \text{polylog}(p) \nonumber\\
			& = \tfrac{\text{poly}(\epsilon^{-1}, \delta^{-1})}{p^{0.5}} \text{polylog}(p) \;.
	\end{align}

	\textbf{Case 2: $\theta = \theta^\star(x_k)$.}
	Compared to \textbf{Case 1}, here we need an additional fact that controls the deviation 
	\begin{align}
		\| \theta^\star(x_k) - \tilde \theta(x_k)  \| \;.
	\end{align}
	Due to the strong convexity of $\cL_{k}^{\kappa, \lambda}(\theta ~;~ \epsilon,\delta)$, we know
	\begin{align}
		& \| \theta^\star(x_k) - \tilde \theta(x_k)  \|  \leq \epsilon^{-1} \| \nabla \cL_{k}^{\kappa, \lambda}(\theta^\star(x_k) ~;~ \epsilon,\delta) - \nabla \cL_{k}^{\kappa, \lambda}(\tilde \theta(x_k)  ~;~ \epsilon,\delta)   \| \nonumber\\
		& = \epsilon^{-1}\| \nabla \cL_{k}^{\kappa, \lambda}(\tilde \theta(x_k)  ~;~ \epsilon,\delta)  \| \nonumber\\
		& = \epsilon^{-1} \| \nabla \cL_{k}^{\kappa, \lambda}(\tilde \theta(x_k)  ~;~ \epsilon,\delta) - \nabla \cL_{\backslash k}^{\kappa, \lambda} (\theta_{\backslash k}^\star ~;~ \epsilon, \delta) \| \nonumber\\
		& = \epsilon^{-1} \| \nabla \cL_{\backslash k}^{\kappa, \lambda}(\tilde \theta(x_k)  ~;~ \epsilon,\delta) + \nabla h_\delta^2(\kappa - y_k \tfrac{1}{\sqrt{p}} x_k^\top \tilde \theta(x_k) ) - \nabla \cL_{\backslash k}^{\kappa, \lambda} (\theta_{\backslash k}^\star ~;~ \epsilon, \delta) \| \nonumber \\
		& = \epsilon^{-1} \| \nabla \cL_{\backslash k}^{\kappa, \lambda}(\tilde \theta(x_k)  ~;~ \epsilon,\delta) -  H_{\backslash k}(\epsilon) (\tilde \theta(x_k) - \theta_{\backslash k}^\star)  - \nabla \cL_{\backslash k}^{\kappa, \lambda} (\theta_{\backslash k}^\star ~;~ \epsilon, \delta) \| \nonumber\\
		& \precsim \epsilon^{-1} \left\| \sum_{i=1}^{k-1} [2h'''_\delta(t) h_\delta(t) + 6 h''_\delta(t) h'_\delta(t)]_{t = \kappa - y_i \frac{1}{\sqrt{p}}x_i^\top \theta_{m}} \big( \tfrac{1}{\sqrt{p}} x_i^\top (\tilde \theta(x_k) - \theta^\star_{\backslash k} ) \big)^2 y_i \tfrac{1}{\sqrt{p}} x_i \right. \nonumber \\ 
		& \quad \quad  + \left. \sum_{i=k+1}^{n} [2h'''_\delta(t) h_\delta(t) + 6 h''_\delta(t) h'_\delta(t)]_{t = \kappa - y_i \frac{1}{\sqrt{p}} \tilde x_i^\top \theta_{m}} \big( \tfrac{1}{\sqrt{p}} \tilde x_i^\top (\tilde \theta(x_k) - \theta^\star_{\backslash k}) \big)^2 y_i \tfrac{1}{\sqrt{p}} \tilde x_i   \right\| \label{eqn:third-order-from-first-order} \\
		&\quad \quad + \sqrt{ \sum_{j=1}^p \lambda^2 | h'''_\delta(t) - h'''_\delta(-t) |^2_{t = \theta_m[j]} \big( \tilde \theta(x_k)[j] - \theta^\star_{\backslash k}[j] \big)^4 }  \nonumber\label{eqn:third-order-from-first-order-regularization}
	\end{align}
	where $\theta_m$ lies between $\tilde{\theta}(x_k)$ and $\theta^\star_{\backslash k}$.
	Using the same arguments as in \textbf{Case 1} for terms (a) and (b), we know
	\begin{align}
		\alpha_i & := \left| [2h'''_\delta(t) h_\delta(t) + 6 h''_\delta(t) h'_\delta(t)]_{t = \kappa - y_i \frac{1}{\sqrt{p}} \tilde x_i^\top \theta_{m}} \big( \tfrac{1}{\sqrt{p}} \tilde x_i^\top (\tilde \theta(x_k) - \theta^\star_{\backslash k} ) \big)^2 y_i \right| \nonumber\\
		& \precsim \delta^{-2} \epsilon^{-0.5} \cdot \big( \tfrac{\epsilon^{-1.5} \text{polylog}(p)}{p^{0.5}} \big)^2 \;.
	\end{align} 
	Therefore \eqref{eqn:third-order-from-first-order} can be bounded as follows
	\begin{align}
		\eqref{eqn:third-order-from-first-order} & \leq \epsilon^{-1} \left\| \big[ \tfrac{1}{\sqrt{p}} x_1, \ldots, \tfrac{1}{\sqrt{p}} x_{k-1}, \tfrac{1}{\sqrt{p}} \tilde x_{k+1}, \ldots, \tfrac{1}{\sqrt{p}} \tilde x_{n} \big] \right\|_{\rm op} \| \alpha \| \nonumber \\
		& \precsim \epsilon^{-1}  \sqrt{n} \cdot \delta^{-2} \epsilon^{-0.5} \big( \tfrac{\epsilon^{-1.5} \text{polylog}(p)}{p^{0.5}} \big)^2 \precsim \tfrac{\delta^{-2}\epsilon^{-4.5}}{p^{0.5} } \text{polylog}(p)
	\end{align}
	where we use the fact $ \left\| \big[ \tfrac{1}{\sqrt{p}} x_1, \ldots, \tfrac{1}{\sqrt{p}} x_{k-1}, \tfrac{1}{\sqrt{p}} \tilde x_{k+1}, \ldots, \tfrac{1}{\sqrt{p}} \tilde x_{n} \big] \right\|_{\rm op}  = O_P(1)$.
	For \eqref{eqn:third-order-from-first-order-regularization}, we know from the argument in bounding term (c) in \textbf{Case 1} that
	\begin{align}
		\eqref{eqn:third-order-from-first-order-regularization} &\precsim \delta^{-2} \sqrt{\epsilon^{-3} p \tfrac{1}{p^2} \text{polylog}(p)} \precsim \tfrac{\delta^{-2}\epsilon^{-4.5}}{p^{0.5} } \text{polylog}(p) \;.
	\end{align}
	Therefore, we have established that
	\begin{align}
		\label{eqn:l2-eval-x_k}
		\| \theta^\star(x_k) - \tilde \theta(x_k)  \|  \precsim \tfrac{\delta^{-2}\epsilon^{-4.5}}{p^{0.5} } \text{polylog}(p) \;.
	\end{align}
\end{proof}

Now we revisit the terms (a), (b), (c) in the case where $\theta = \theta^\star(x_k)$.

	\textbf{Term (a)}:
	\begin{align}
		& \quad \big| \tfrac{1}{\sqrt{p}} x_i^\top ( \theta^\star(x_k) - \theta^\star_{\backslash k}) \big|  \nonumber\\
		& \leq \big| \tfrac{1}{\sqrt{p}} x_i^\top ( \tilde \theta(x_k) - \theta^\star_{\backslash k}) \big|  + \big| \tfrac{1}{\sqrt{p}} x_i^\top ( \tilde \theta(x_k) -\theta^\star(x_k) ) \big|  \nonumber\\
		& \precsim  \tfrac{\epsilon^{-1.5}}{p^{0.5}} \text{polylog}(p) + \tfrac{\delta^{-2}\epsilon^{-4.5}}{p^{0.5} } \text{polylog}(p) \quad \text{by \eqref{eqn:term-a} and \eqref{eqn:l2-eval-x_k}} \;.
	\end{align}
	
	\textbf{Term (b)}:
	\begin{align}
		&\quad |2 h'''_\delta(t) h_\delta(t) + 6 h''_\delta(t) h'_\delta(t) |_{t =\kappa - y_i \frac{1}{\sqrt{p}}x_i^\top \theta_{m}} \nonumber \\
		&\precsim  \delta^{-2} \left\{ \kappa + \max \big\{ |\tfrac{1}{\sqrt{p}}x_i^\top \theta^\star_{\backslash k} |, |\tfrac{1}{\sqrt{p}}x_i^\top  \theta^\star(x_k) | \big\} \right\} \nonumber \\
		& \precsim \delta^{-2} \left\{   |\tfrac{1}{\sqrt{p}}x_i^\top \theta^\star_{\backslash k} | +  \big| \tfrac{1}{\sqrt{p}} x_i^\top ( \theta^\star(x_k) - \theta^\star_{\backslash k}) \big| \right\} \nonumber \\
		&\precsim \delta^{-2} \left\{  \epsilon^{-0.5} \vee \epsilon^{-1}\textrm{polylog}(p)  + \tfrac{\delta^{-2}\epsilon^{-4.5}}{p^{0.5}} \textrm{polylog}(p) \right\} \;.
	\end{align}
	
	\textbf{Term (c)}:
	\begin{align}
			&\quad \delta^{-2} \sum_{j=1}^p  \big|\theta^\star(x_k)[j] - \theta^\star_{\backslash k}[j] \big|^3  \nonumber \\
			&\precsim \delta^{-2} \sum_{j=1}^p  \big|\theta^\star(x_k)[j] - \tilde \theta(x_k)[j]\big|^3 + \big|\tilde \theta(x_k)[j] -  \theta^\star_{\backslash k}[j]  \big|^3 \nonumber \\
			&\precsim \delta^{-2} \left\{  \sum_{j=1}^p (\theta^\star(x_k)[j] -  \tilde \theta(x_k)[j])^2 \right\}^{3/2} + \delta^{-2} \sum_{j=1}^p  \big|\tilde \theta(x_k)[j] -  \theta^\star_{\backslash k}[j]  \big|^3 \nonumber \\
			& \precsim  \delta^{-2} \| \theta^\star(x_k) - \tilde \theta(x_k) \|^3 + \tfrac{\delta^{-2}\epsilon^{-4.5}}{p^{0.5}} \text{polylog}(p) \nonumber \\
			&\precsim \tfrac{\delta^{-8}\epsilon^{-13.5}}{p^{1.5}} \text{polylog}(p) + \tfrac{\delta^{-2}\epsilon^{-4.5}}{p^{0.5}} \text{polylog}(p) \;.
	\end{align}
	
	Again, putting  the upper bounds on (a), (b) and (c) together, we have shown \eqref{eqn:decomposition} with $\theta =  \theta^\star(x_k)$ is upper bounded by
	\begin{align}
			\eqref{eqn:decomposition} 
			& \precsim \tfrac{\text{poly}(\epsilon^{-1}, \delta^{-1})}{p^{0.5}} \text{polylog}(p) \;.
	\end{align}
\end{proof}

\subsection{Supporting Lemmas}	
Throughout the proof of Theorem~\ref{thm:universalitythm}, we rely on the following two lemmas, and the proof is complete on proving these.

\begin{lemma}[Moreau envelope]
	\label{lem:moreau-envelope}
	Assume that $\delta < \tfrac{\kappa}{2} \epsilon$, the following estimates on the Moreau envelope hold,
	\begin{align*}
		\cM_k(t, \gamma_k) &\leq (\kappa + \delta + |t|)^2 \;, \\
		\cM'_k(t, \gamma_k) &\leq 2 (2\kappa+\delta+|t| ) \;.
	\end{align*}
\end{lemma}
\begin{proof}[Proof of Lemma~\ref{lem:moreau-envelope}]
	For the zeroth order estimate, we have
	\begin{align}
		0 \leq \cM_k(t, \gamma_k) \leq \ell_{\delta}(t, y_k)  = h_{\delta}^2 (\kappa - y_k t) \leq  (\kappa + \delta + |t|)^2 \;.
	\end{align}
	
	For the first order estimate, we have by the Envelope Theorem
	\begin{align}
		\cM'_k(t, \gamma_k)  &= \ell'_\delta(s, y_k) ~|_{s = s^\star(t)} \nonumber \\
		& = - 2 h_{\delta}(\kappa - y_k s^\star) h'_{\delta}(\kappa - y_k s^\star) y_k \nonumber \\
		|\cM'_k(t, \gamma_k)| &\leq 2 (\kappa+\delta+|s^\star(t)|)
	\end{align}
	where $s^\star(t)$ is the solution to the equation on $s$, for any fixed $t$ (proximal map)
	\begin{align}
		s + \gamma_k \ell'_\delta(s, y_k) = t \;.
	\end{align}
	Due to the non-expansiveness of the proximal map, we have
	\begin{align}
		\label{eqn:first-order-estimate-proximal}
		|\cM'_k(t, \gamma_k)| \leq 2 (\kappa+\delta+|t| + |s^\star(0)|) \leq 2 (\kappa+\delta+|t| + \kappa)
	\end{align}
	where the last step uses the fact 
	\begin{align}
		 \frac{s^\star(0)}{\gamma_k}  = 2 h_{\delta}(\kappa - y_k s^\star(0)) h'_{\delta}(\kappa - y_k s^\star(0)) y_k
	\end{align}
	and if $y_k s^\star(0)>\kappa$, we will reach a contradiction $2\delta < \kappa \epsilon <\tfrac{\kappa}{\gamma_k} \leq 2 |h_{\delta}(\kappa - y_k s^\star(0)) h'_{\delta}(\kappa - y_k s^\star(0))| \leq 2\delta$.
\end{proof}

\begin{lemma}[Gaussian smoothing]
	\label{lem:smoothing}
	The following estimates hold true
	\begin{align*}
		| h_{\delta}(t) - h(t) | &\leq \sqrt{\tfrac{2}{\pi}} \delta \;,\\
		| h'_{\delta}(t)| &\leq 1 \;,\\
		| h''_{\delta}(t)| &\leq 2\delta^{-2} |t| + 3 \sqrt{\tfrac{2}{\pi}}\delta^{-1} \;, \\
		| h'''_{\delta}(t)| &\leq 6 \delta^{-2} \;.
	\end{align*}
\end{lemma}
\begin{proof}[Proof of Lemma~\ref{lem:smoothing}]
	For the zeroth order estimate, we have
	\begin{align}
		| h_{\delta}(t) - h(t) |  = \E[ |h(t+\delta \bg) - h(t)|] \leq \delta \E[ |\bg|] = \sqrt{\tfrac{2}{\pi}} \delta \;.
	\end{align}
	For the first order estimate,
	\begin{align}
		|h'_\delta(t)| &= \big| \int_{\Reals} \tfrac{1}{\sqrt{2\pi}\delta} e^{\frac{(s-t)^2}{2\delta^2}} \cdot \tfrac{(t-s)}{\delta^2} \cdot h(s) \dif{s} \big| = \delta^{-1}| \E[\bg h(t+\delta \bg)] | \nonumber \\
		& \leq \delta^{-1}\big| \E[\bg h(t)] \big| + \delta^{-1} \E[|\bg| \cdot| h(t+\delta \bg) - h(t) | ] \leq 1 \;.
	\end{align}
	For the second order estimate,
	\begin{align}
		|h''_\delta(t)| = \delta^{-2} \big| \E[(1+\bg^2) h(t + \delta \bg)] \big| \leq 2\delta^{-2} |t| + 3 \sqrt{\tfrac{2}{\pi}}\delta^{-1}  \;.
	\end{align}
	For the third order estimate,
	\begin{align}
		|h'''_\delta(t)| = \delta^{-3} \big| \E[(3\bg + \bg^3) h(t + \delta \bg)] \big| \leq 6 \delta^{-2} \;.
	\end{align}
\end{proof}


\end{document}